\documentclass{article}

\usepackage[english]{babel}

\usepackage[letterpaper,top=2cm,bottom=2cm,left=3cm,right=3cm,marginparwidth=1.75cm]{geometry}

\usepackage{bm}
\usepackage{amsmath}
\usepackage{amssymb}
\usepackage{amsthm}
\usepackage{graphicx}
\usepackage{parallel}
\usepackage{ stmaryrd }
\usepackage{tikz}
\usetikzlibrary{calc}
\usepackage{multicol}
\usepackage{ stmaryrd }
\usepackage{ latexsym }
\usepackage[shortlabels]{enumitem}
\usepackage{hyperref}
\hypersetup{colorlinks,allcolors=blue,breaklinks=true}

\newcommand{\forkindep}[1][]{
  \mathrel{
    \mathop{
      \vcenter{
        \hbox{\oalign{\noalign{\kern-.3ex}\hfil$\vert$\hfil\cr
              \noalign{\kern-.7ex}
              $\smile$\cr\noalign{\kern-.3ex}}}
      }
    }\displaylimits_{#1}
  }
}

\newcommand{\boxzero}{{\fbox{$0$}}}
\newcommand{\boxone}{{\fbox{$1$}}}

\newcommand{\hor}{{\operatorname{hor}}}
\newcommand{\ver}{{\operatorname{ver}}}

\newcommand{\mipo}{\operatorname{MiPo}}
\newcommand{\Char}{\operatorname{Char}}

\newcommand{\NN}{\mathbb{N}}

\newcommand{\QQ}{\mathbb{Q}}
\newcommand{\RR}{\mathbb{R}}
\newcommand{\FF}{\mathbb{F}}
\newcommand{\ZZ}{\mathbb{Z}}
\newcommand{\VV}{\mathbb{V}}

\newcommand{\QQp}[1]{\mathbb{Q}[X]_{\operatorname{irr}}^{#1}}
\newcommand{\Kp}[1]{K[X]_{\operatorname{irr}}^{#1}}
\newcommand{\Cc}{{\mathcal{C}}}
\newcommand{\Ccalg}{{\mathcal{C}^{\operatorname{alg}}}}
\newcommand{\Cctrans}{{\mathcal{C}^{\operatorname{tr}}}}

\newcommand{\mm}{\mathcal{M}}
\newcommand{\nn}{\mathcal{N}}
\newcommand{\bb}{\mathcal{B}}
\newcommand{\Gg}{\mathcal{G}}
\newcommand{\ii}{\mathcal{I}}
\newcommand{\jj}{\mathcal{J}}
\newcommand{\kk}{\mathcal{K}}
\newcommand{\rr}{\mathcal{R}}
\newcommand{\dd}{\mathcal{D}}

\newcommand{\set}[1]{{\{#1\}}}
\newcommand{\Set}[1]{{\left\{#1\right\}}}
\newcommand{\spanA}[2]{{{\langle#1\rangle_{#2}}}}

\newcommand{\dotcup}{\mathbin{\dot{\cup}}}
\newcommand{\Fac}{{\operatorname{Fac}}}
\newcommand{\Ker}{{\operatorname{Ker}}}
\newcommand{\ACF}{{\operatorname{ACF}}}
\newcommand{\Image}{{\operatorname{Im}}}

\newcommand{\dcl}{{\operatorname{dcl}}}

\newcommand{\cl}{{\operatorname{cl}}}
\newcommand{\Diag}{{\operatorname{Diag}}}
\newcommand{\Id}{{\operatorname{Id}}}
\newcommand{\Th}{{\operatorname{Th}}}

\newcommand{\acl}{{\operatorname{acl}}}

\newcommand{\ua}{{\underline{a}}{}}
\newcommand{\ub}{{\underline{b}}{}}
\newcommand{\uc}{{\underline{c}}{}}
\newcommand{\ud}{{\underline{d}}{}}

\newcommand{\ug}{{\underline{g}}{}}
\newcommand{\uh}{{\underline{h}}{}}

\newcommand{\us}{{\underline{s}}{}}
\newcommand{\ut}{{\underline{t}}{}}
\newcommand{\uu}{{\underline{u}}{}}
\newcommand{\uv}{{\underline{v}}{}}
\newcommand{\uw}{{\underline{w}}{}}
\newcommand{\ux}{{\underline{x}}{}}
\newcommand{\uy}{{\underline{y}}{}}
\newcommand{\uz}{{\underline{z}}{}}

\newcommand{\uzero}{{\underline{0}}}
\newcommand{\ulambda}{{\underline{\lambda}}}

\newcommand{\li}{{\scalebox{0.5}{{$\operatorname{li}$}}}}
\newcommand{\ld}{{\scalebox{0.5}{{$\operatorname{ld}$}}}}

\newcommand{\lii}{{\scalebox{0.5}{$\operatorname{li}$}}}
\newcommand{\ldd}{{\scalebox{0.5}{$\operatorname{ld}$}}}

\newcommand{\tilv}{{\Tilde{v}}}

\newcommand{\MM}{\mathbb{M}}

\newcommand{\uxvec}{{\underline{\Vec{x}}}}

\newcommand{\xvec}{{\Vec{x}}}

\newcommand{\Hfour}{{$(\operatorname{H4})$}}

\newcommand{\TKvs}{{T_{K\operatorname{-vs}}}}
\newcommand{\TKvsThe}{{T_{K\operatorname{-vs},\theta}}}
\newcommand{\TKvsTheC}{{T^C_{K\operatorname{-vs},\theta}}}

\newcommand{\RCF}{{\operatorname{RCF}}}

\newcommand{\tp}{{\operatorname{tp}}}

\newcommand{\LK}{{L_K}}
\newcommand{\NIP}{$\operatorname{NIP}$}

\newcommand{\TPtwo}{$\operatorname{TP}_2$}
\newcommand{\NATP}{$\operatorname{NATP}$}

\newcommand{\SOP}{$\operatorname{SOP}$}
\newcommand{\dprk}{\operatorname{dp-rk}}

\newcommand{\LKThe}{{L_{K,\theta}}}
\newcommand{\LRC}{{L_{R_C}}}

\newcommand{\Lr}{L_{\operatorname{r}}}

\newcommand{\emptyseq}{{\scalebox{0.5}{$\langle\rangle$}}}

\newcommand{\commonBaseTheorem}{\hyperref[theorem_r_c_element]{Common Base Theorem}}

\newcommand{\drawTextHelper}[5]{
\node[anchor=center, scale = 1] at (#1 * #4 + 0.5 * #4, -#2 * #5 - 0.5 * #5) {$#3$};
}
\newcommand{\drawText}[3]{\drawTextHelper{#1}{#2}{#3}{1.0}{0.6}}

\newcommand{\drawBorderHelper}[6]{\draw [black, line width=0.75]
(#1 * #5 + 0.105, -#2 * #6) --
(#1 * #5, -#2 * #6) --
(#1 * #5, -#4 * #6) --
(#1 * #5 + 0.105, -#4 * #6)
(#3 * #5 - 0.105, -#2 * #6) --
(#3 * #5, -#2 * #6) --
(#3 * #5, -#4 * #6) --
(#3 * #5 - 0.105, -#4 * #6);
\node[anchor=east, scale = 1.0] (Frame) at (0.15, -0.5 * #2 * #6 + -0.5 * #4 * #6 -0.08) {};}
\newcommand{\drawBorder}[4]{\drawBorderHelper{#1}{#2}{#3}{#4}{1.0}{0.6}}

\newcommand{\drawHDotsHelper}[5]{
\draw [line width=1.0, line cap=round, dash pattern=on 0 off 9.5 * #4]
        (#1 * #4, -#2 * #5 - 0.5 * #5) --
        (#1 * #4 + #3 * #4 + 0.02 * #4, -#2 * #5 - 0.5 * #5);    
}
\newcommand{\drawHDots}[3]{\drawHDotsHelper{#1}{#2}{#3}{1.0}{0.6}}

\newcommand{\drawVDotsHelper}[5]{
\draw [line width=1.0, line cap=round, dash pattern=on 0pt off 9.5 * #5]
        (#1 * #4 + 0.5 * #4, -#2 * #5) --
        (#1 * #4 + 0.5 * #4, -#2 * #5 -#3 * #5 - 0.02 * #5);    
}
\newcommand{\drawVDots}[3]{\drawVDotsHelper{#1}{#2}{#3}{1.0}{0.6}}

\newtheorem*{notation}{Notation}

\newtheorem{theorem}{Theorem}[section]
\newtheorem*{theorem*}{Theorem}
\newtheorem*{setting*}{Setting}
\newtheorem*{corollary*}{Corollary}
\newtheorem{theoremi}{Theorem}

\newtheorem{definition}[theorem]{Definition}
\newtheorem{fact}[theorem]{Fact}

\newtheorem{remark}[theorem]{Remark}
\newtheorem{lemma}[theorem]{Lemma}
\newtheorem{corollary}[theorem]{Corollary}
\newtheorem{example}[theorem]{Example}
\newtheorem{observation}[theorem]{Observation}

\newtheorem{subclaim}{Claim}[theorem]
\newtheorem{subdefinition}[subclaim]{Definition}

\newtheorem*{subdefinition*}{Definition}
\newtheorem*{subclaim*}{Claim}

\newenvironment{innerproof}[1][Proof]
{\begin{proof}[#1]}
{\end{proof}}

\usepackage{titlesec}
\titleformat{\section}
  {\normalfont\Large\bfseries\boldmath}{\thesection}{1em}{}
\titleformat{\subsection}
  {\normalfont\large\bfseries\boldmath}{\thesubsection}{1em}{}
\titleformat{\subsubsection}
  {\normalfont\normalsize\bfseries\boldmath}{\thesubsubsection}{1em}{}

\title{Model Theory of Generic Vector Space Endomorphisms V: The o-Minimal Case}
\author{Leon Chini}

\newcommand{\Addresses}{{
  \bigskip
  \footnotesize

  \textsc{Mathematisches Institut, Universität Bonn, Endenicher Allee 60, D-53115 Bonn, Germany}\par\nopagebreak
  \textit{E-mail address}: 
  \href{mailto:lchini@uni-bonn.de}{\tt lchini@uni-bonn.de}
}}

\begin{document}
\maketitle

\begin{abstract}
\noindent This paper further studies the model companion of an endomorphism acting on a vector space, possibly with extra structure. Let $T$ be a model-complete theory that $\varnothing$-defines an infinite $K$-vector space $\mathbb{V}$. In previous work, we introduced a family $\{T^C_\theta : C \in \mathcal{C}\}$ of extensions of the theory $T_\theta := T \cup \{\text{``$\theta$ is an endomorphism of $\mathbb{V}$''}\}$ that parameterizes all consistent extensions of the form  
$$  
    T_\theta \cup \left\{\sum\nolimits_{k}\bigcap\nolimits_{l}\operatorname{Ker}(\rho_{j, k, l}[\theta]) = \sum\nolimits_{k}\bigcap\nolimits_{l} \operatorname{Ker}(\eta_{j, k, l}[\theta]) : j \in \mathcal{J}\right\},  
$$  
where all sums and intersections are finite, all the $\rho[\theta]$'s and $\eta[\theta]$'s are polynomials over $K$ with $\theta$ plugged in, and $\mathcal{J}$ is some possibly infinite index set. We also presented a sufficient condition that implies that every $T^C_\theta$ has a model companion $T\theta^C$. In this paper, we study the case where $T$ is an o-minimal expansion of the theory of ordered groups. Doing so, we obtain a new family of theories that have $\operatorname{TP}_2$ and $\operatorname{SOP}$, and are $\operatorname{NATP}$, as well as distal non-o-minimal theories of various $\operatorname{dp}$-ranks with and without the exchange property.
\end{abstract}

\tableofcontents
\section{Introduction}

This paper is a continuation of \cite{Chi25}, \cite{Chi25b}, \cite{Chi26}, and \cite{Chi26b} which deal with the model companion of an endomorphism acting on a vector space, possibly with extra structure. For the relevance of this line of inquiry and a description of earlier work, see the introduction of \cite{Chi25}. The goal of this paper is to study the case where the starting theory in which the vector space lives is an o-minimal expansion of the theory of ordered groups.

We call a theory \textit{o-minimal} if it expands the theory of dense linear orders without endpoints and, in every model, every definable subset of the underlying order in one variable is a finite union of points and intervals, with endpoints in that model or in $\set{\pm\infty}$.
Note that some authors also consider non-dense structures such as $(\ZZ, <)$ to be o-minimal, but since we will never encounter such structures, we assume density throughout.
O-minimality was first introduced by van den Dries in \cite{vdD84} as a setting in which some results from semialgebraic geometry generalize.
It was further studied in \cite{PS86} by Pillay and Steinhorn, who gave this property the name o-minimality.
Definable sets in o-minimal theories are particularly well understood, since every definable set can be decomposed into finitely many definable sets of a particularly nice form, called \textit{cells}.
Using this decomposition, one can, for example, assign a meaningful dimension to every definable set.
Overall, o-minimality provides a powerful framework for geometry that avoids phenomena such as Cantor sets, space-filling curves, or even infinite discrete sets.
Thus o-minimality can be regarded as a notion of ``tame geometry''.
There are many other notions of ``tame geometry'' that generalize o-minimality, such as local o-minimality \cite{TV09}, weak o-minimality \cite{MMS00}, or $d$-minimality \cite{Mil05}.
Although o-minimality implies \NIP{}, some of these other notions of ``tame geometry'' generally do not express tameness in the sense of combinatorial complexity or neostability, but rather in the sense of the topology induced by the order.
O-minimality has striking applications to Diophantine geometry; for example, Pila and Zannier reproved the Manin-Mumford conjecture in \cite{PZ08}, where a crucial part relies on the so-called Pila-Wilkie theorem \cite{PW06}, which roughly states that for any set definable in an o-minimal expansion of $(\RR, 0, 1, +, \cdot)$, the number of rational points of height $\leq N$ lying outside its algebraic part grows slower than any power of $N$.
Similarly, the André-Oort conjecture was proven in full generality in \cite{PSTEG24} by Pila, Shankar, and Tsimerman.

We now discuss the results from our previous papers. Let $L$ be a language and $T$ a model-complete $L$-theory with an infinite $\varnothing$-definable $K$-vector space $\VV$ in every model.
Given a consistent set $C$ of constraints on an endomorphism $\theta$, which encodes conditions of the form  
$$
    \sum\nolimits_{k}\bigcap\nolimits_{l}\Ker(\rho_{k, l}[\theta]) = \sum\nolimits_{k}\bigcap\nolimits_{l} \Ker(\eta_{k, l}[\theta])
$$  
(where all sums and intersections are finite,  
and all the $\rho[\theta]$'s and $\eta[\theta]$'s are polynomials over $K$ with the endomorphism $\theta$ plugged in), we define the following theory in the language $L_\theta := L \cup \set{\theta}$:  
$$
T^C_\theta := T \cup \set{\text{``$\theta$ is an endomorphism of $\VV$''}} \cup \set{\text{``$\theta$ satisfies the constraints in $C$''}}.
$$
Note that these sets of constraints are a simplification for the sake of this introduction and that we will actually use so-called kernel configurations instead (see Definition \ref{def_kernel_conf}).
In \cite{Chi25}, we showed that $T^C_\theta$ has a model companion $T\theta^C$ if $T$ satisfies a certain condition \Hfour{}, which corresponds to Definition 1.11 in \cite{dEl21b} and basically states that ``$\psi(\ux; \uy)$ implies no finite disjunction of non-trivial linear dependencies in $\ux$ over $\VV$'' can be expressed as an $L$-formula $\sigma_\psi(\uy)$ for every $L$-formula $\psi(\ux; \uy)$. These results, as well as all others relevant to this paper, will be recalled in Section \ref{sec_prelim_res}. In \cite{Chi25b}, we studied definable sets in $T\theta^C$, the completions of this theory, as well as the algebraic closure in its models. We also obtained a quantifier-elimination result in \cite{Chi25b}, which we will apply in this paper. In \cite{Chi26}, we studied certain reducts of the model companion $T\theta^C$; for example, we were able to show that given $(\mm, \theta) \models T^C_\theta$ existentially closed and some $\rho \in K[X]$, the structure $(\mm, \Ker(\rho[\theta]))$ is either interdefinable with $\mm$ or an existentially closed model of the theory which expands $T$ by a predicate for a vector subspace. Finally, in \cite{Chi26b}, we proved that our construction preserves the neostability property \NATP{} if $T$ satisfies \Hfour{}. Here \NATP{} stands for ``not the \textit{antichain tree property}'', and is a mutual generalization of $\operatorname{NSOP}_1$ and $\operatorname{NTP}_2$ recently introduced by Ahn and Kim in \cite{AK24}.

We study our construction in the case where $T$ is an o-minimal expansion of the theory of ordered groups.
Note that this does not mean that $\VV$ is the ($\QQ$-)vector space induced by the group structure; for example, one can choose $T = \RCF{}$, the theory of real closed fields, and $\VV$ to be the $\QQ$-vector space induced by multiplication on the positive elements.
The work of Block Gorman in \cite{Blo23} gives us many o-minimal examples that satisfy \Hfour{}.
For example, any o-minimal expansion of $\RCF{}$ satisfies \Hfour{} with the above-mentioned vector space $\VV$ if and only if no partial exponential function is definable. In the following, ``linear'' roughly means that definable functions are piecewise affine (see Definition \ref{def_linear_new}) and we call a set of constraints $C$ \textit{trivial} if it implies that the endomorphism $\theta$ is the multiplication with some $\lambda \in K$.

\begin{theoremi} \label{thmi_dichotommy}
    Let $C$ be a non-trivial set of constraints, and let $T$ be a complete and model-complete o-minimal expansion of the theory of ordered groups.
    Then one of the following holds (see Lemma \ref{lemma_omin_sat_hfour}):
    \begin{enumerate}[(i)]
        \item The theory $T$ is linear and satisfies \Hfour{}.
        In this case:
        \begin{enumerate}[(a)]
            \item The vector space $\VV$, as a group, is isomorphic to $(M^d, 0, +)$ for $d = \dim(\VV)$ (Theorem \ref{lemma_group_lin_iso}).
            \item The theory $T\theta^C$ is distal, and hence \NIP{} (Theorem \ref{theorem_still_distal}).
            \item The $\dprk{}$ of $T\theta^C$ is given in terms of $\dim_K(R_C)$ and $\dim_K(\mathfrak{G})$, where $\mathfrak{G}$ denotes the $K$-vector space of all germs of definable endomorphisms of $\VV$ at $0$ (Theorem \ref{corollary_dp_rank_btter_fml}).
            \item The closure operator $\acl_{L_\theta}$ has the exchange property if and only if $R_C$ is a field and $\dim_K(\mathfrak{G}) = 1$ (Theorem \ref{theorem_acl_fail}).
        \end{enumerate}
        \item The theory $T$ is non-linear and satisfies \Hfour{}.
        In this case:
        \begin{enumerate}[(a)]
            \item The theory $T\theta^C$ has \TPtwo{} and \SOP{}, and is \NATP{} (Section \ref{sec_neo} and Corollary \ref{corollary_tP_two}).
            \item The closure operator $\acl_{L_\theta}$ does not have the exchange property.
        \end{enumerate}
        \item The theory $T$ is non-linear and does not satisfy \Hfour{}.
    \end{enumerate}
\end{theoremi}

\noindent In particular, every linear o-minimal expansion of the theory of ordered groups satisfies \Hfour{}.
The author believes that, in case (iii) of Theorem \ref{thmi_dichotommy}, the theory $T^C_\theta$ has no model companion, which would give a clean trichotomy.
However, we can prove this only in the following special case, using the results of \cite{Blo23}:

\begin{theoremi}[Theorem \ref{theorem_prove_conj}]
    Let $C$ be a non-trivial set of constraints, and let $T$ be a complete and model-complete o-minimal expansion of the theory of ordered groups.
    Furthermore, suppose that $\VV$ is, as a set, an open subinterval of the line, that the operations on $\VV$ are continuous, and that $K = \QQ$.
    Then $T^C_\theta$ has a model companion if and only if $T$ satisfies \Hfour{}.
\end{theoremi}

\noindent It is easy to see that the model companion $T\theta^C$ is not o-minimal unless $C$ is trivial (Observation 5.2 in \cite{Chi25b}).
Having an o-minimal open core is another notion of tame geometry: Given a structure, the open core is the reduct obtained by adding predicates for the closures of all sets definable with parameters and then removing the original language. We would like to note that in \cite{Chi25b}, we have already shown that whenever $T$ is o-minimal and satisfies \Hfour{}, the model companion $T\theta^C$ has an o-minimal open core, even if $T$ does not expand the theory of ordered groups (Theorem 5.3 in \cite{Chi25b}).

Overall, studying $T\theta^C$ when $T$ is an o-minimal expansion of the theory of ordered groups yields many interesting new theories.
On the one hand, we obtain non-o-minimal distal theories with an o-minimal open core, both with and without the exchange property, and with various values of $\dprk{}$.
On the other hand, we obtain many new examples of theories that have \TPtwo{} and \SOP{}, and are \NATP{}.
One of them is $\RCF{}\!\theta^C$, as already conjectured in \cite{dEl25}.
In general, working with sets of constraints $C \neq C_0$, where $C_0$ is the set of constraints that ensures that every $\rho[\theta]$ is injective for $\rho \in K[X] \setminus \set{0}$, can add a lot of notational complexity.
Thus we believe that $\RCF{}\!\theta^{C_0}$ is, in some sense, the perfect toy example for open problems regarding our construction, such as classifying all definable endomorphisms of $\VV$ or eliminating imaginaries.
On the one hand, $\RCF{}\!\theta^{C_0}$ is still very ``wild'' in terms of combinatorial complexity, but on the other hand, the ingredients $\RCF{}$ and $C_0$ are both very ``tame''.
\\

\noindent \textbf{Acknowledgment.} The author would like to thank Christian d'Elbée and Philipp Hieronymi for reading parts of an earlier draft of this paper and for providing valuable feedback.

\section{Preliminary Results} \label{sec_prelim_res}

In this chapter, we provide a recap of all relevant results from \cite{Chi25}, \cite{Chi25b}, and \cite{Chi26}. We also show some new results, which follow easily from these.

\subsection{The setting}  \label{sec_setting}
Let $L$ be a first-order language, let $T$ be a model-complete $L$-theory, and let $K$ be a field.
Furthermore, assume that the theory of $K$-vector spaces is definable in $T$.
By this we mean that there are $L$-formulas $\Omega_{\VV}(\ux)$, $\Omega_{0}(\ux)$, $\Omega_{+}(\ux_1, \ux_2, \uy)$, and $\Omega_{\lambda\cdot}(\ux, \uy)$ for each $\lambda \in K$ such that, in every model $\mm \models T$, they define an infinite $K$-vector space $(\VV, 0, +, (\lambda \cdot)_{\lambda \in K})^{\mm}$.
In the case $K = \QQ$, one may also view $(\VV, 0, +)^{\mm}$ as a torsion-free divisible abelian group, since these are precisely the $\QQ$-vector spaces.
If a model $\mm \models T$ is given, then $\VV$ denotes $\VV^\mm$.
If the model is denoted with $\mm'$ instead of $\mm$, then we will write $\VV'$ instead of $\VV$.
If no model of $T$ is clear from the context, then we will still use the letter $\VV$ to denote a $K$-vector space.
We may say ``vector space'' instead of ``$K$-vector space'', ``polynomial'' instead of ``$K$-polynomial'', ``linearly independent'' instead of ``$K$-linearly independent'', and so on. Definable will always mean $\varnothing$-definable.

\begin{example}
    \label{example_main}
    Two of our main examples are as follows:
    \begin{enumerate}[(i)]
        \item Let $\LK = \set{0, +, (\lambda \cdot)_{\lambda\in K}}$ and let $\TKvs$ be the theory of $K$-vector spaces, with the obvious formulas.
        Then, for any $\mm \models \TKvs$, we choose $(\VV, 0, +, (\lambda \cdot)_{\lambda\in K}) = \mm$.
        Similarly, we can also work with ordered divisible abelian groups, since these are precisely ordered $\QQ$-vector spaces.
        \item Let $K = \QQ$, $\Lr = \set{0, 1, +, \cdot, <}$, and $T = \RCF$, the theory of real closed fields.
        Then, for any $\rr \models \RCF{}$, we choose $(\VV, 0, +, (q \cdot)_{q\in \QQ})^{\rr}$ to be $(\rr_{>0}, 1, \cdot, (x \mapsto x^q)_{q\in \QQ})$.
    \end{enumerate}
\end{example}

\noindent Notationally, we will treat $\VV$ as a unary set, or even as a separate sort.
Any $\LK$-term or formula can then be viewed as an $L$-definable function or an $L$-formula, respectively.
Note that $0$, $+$, and $(\lambda\cdot)_{\lambda\in K}$ need not belong to $L$, so they are not necessarily $L$-terms.
For a given theory, there can be multiple definable vector spaces, as in the case of $\RCF{}$.
Thus, whenever a theory $T$ is given, we actually mean the tuple $(T, \Omega_{\VV}, \Omega_{0}, \Omega_{+}, (\Omega_{\lambda\cdot})_{\lambda \in K})$.

We now define $L_\theta := L(\theta) := L \cup \set{\theta}$, where $\theta$ is a function symbol not contained in $L$.
This has the drawback that, for example, if $\rr \models \RCF$ and the positive elements are viewed as a $\QQ$-vector space, then $\theta$ must also be defined outside $\VV = \rr_{>0}$.
In this case, one might try to set $\theta(0) = 0$ and $\theta(-1) \in \set{1, -1}$ so as to extend $\theta$ to an endomorphism of $(\rr, 1, \cdot)$, but then the ambient structure is no longer a vector space.
For the sake of uniformity, we instead define $\theta(x)$ to be the neutral element of $\VV$ for all $x \not\in \VV$ and set
$$
T_\theta := T \cup \set{\text{``$\theta_{\restriction \VV}$ is a $(\VV, 0, +, (\lambda \cdot)_{\lambda\in K})$-endomorphism''}} \cup \set{\forall x \not\in \VV: \theta(x) = 0}.
$$
In particular, $\theta^n(x) = 0$ for all $x \not\in \VV$ and all $n > 0$.
For consistency, we also define $\theta^0(x) = 0$ whenever $x \not\in \VV$, and $x + y = 0$ whenever either $x \not\in \VV$ or $y \not\in \VV$.
Practically, we will ignore the behavior of $\theta$ outside of $\VV$ and simply treat $\theta$ as a function defined only on $\VV$.
For example, we set $\Ker(\theta) := \set{v \in \VV : \theta(v) = 0}$, and similarly define the kernel for any function that is an endomorphism of $\VV$.

Note that if $\VV$ is an $n$-ary set, then we actually need $n$ different $n$-ary function symbols $\theta_1, \dots, \theta_n$ rather than a single unary function symbol $\theta$.
However, as mentioned above, we will treat elements of $\VV$ as singletons in order to simplify notation and will therefore pretend that $\theta$ is unary.

\begin{definition}
    Given a polynomial $\rho \in K[X]$ and any $d \geq \deg(\rho)$, we let $\rho[\theta]$ denote the endomorphism of $\VV$ defined by
    $
    \rho[\theta](v) := \sum\nolimits_{i=0}^{d} (\rho)_i \cdot \theta^i(v)
    $
    for all $v \in \VV$. Here each $(\rho)_i$ is the respective coefficient of $X^i$ in $\rho$.
\end{definition}

\begin{notation}
    If $\theta$ is clear from the context we may write $\Ker(\rho)$ and $\Image(\rho)$ instead of $\Ker(\rho[\theta])$ and $\Image(\rho[\theta])$.
\end{notation}

\noindent It is easy to check that $\rho[\theta] + \eta[\theta] = (\rho + \eta)[\theta]$ and $\rho[\theta] \circ \eta[\theta] = (\rho \cdot \eta)[\theta]$. Many classical facts for polynomials, such as Bézout's identity or Euclidean division, can be translated to this setting.

\subsection{Kernel configurations and extensions of $T_\theta$}

\noindent As stated in the introduction, we consider a family $\set{T^C_\theta : C \in \Cc}$ of extensions of $T_\theta$.

\begin{notation}
    We let $\Kp{}$ denote the set of all monic irreducible polynomials over $K$.
\end{notation}

\noindent We start by defining our index set $\Cc$:

\begin{definition} \label{def_kernel_conf}
    We call a pair $(c, d)$ a \textbf{kernel configuration} if
    \begin{enumerate}[(i)]
        \item $c \colon \Kp{} \to \NN \cup \set{\infty}$ is a function; and
        \item $d \in \NN_{> 0} \cup \set{\infty}$ is either $\infty$ or satisfies $d = \sum_{f \in \Kp{}} \deg(f) \cdot c(f)$.
    \end{enumerate}
    We let $\Cc$ denote the set of all kernel configurations.
    Given such a kernel configuration $C = (c, d) \in \Cc$, we set $C(f) := c(f)$ for all $f \in \Kp{}$.
    We define the \textbf{degree} of $C$ by $\deg(C) := d$.
    We say that $C$ is \textbf{algebraic} if $\deg(C) < \infty$.
    In this case, we define the \textbf{minimal polynomial} of $C$ by $\mipo(C) := \prod_{f \in \Kp{}} f^{C(f)}$
    (since $\deg(C) < \infty$, only finitely many factors are different from $1$, so this product is well defined).
    We say that $C$ is \textbf{transcendental} if $\deg(C) = \infty$.
    We let $\Ccalg$ and $\Cctrans$ denote the sets of all algebraic and all transcendental kernel configurations, respectively.
\end{definition}

\noindent Note that the set of kernel configurations depends on the field $K$.
We are now ready to define the family $\set{T^C_\theta : C \in \Cc}$:

\begin{definition} \label{def_T_C_theta}
    Given $C \in \Cc$ and an endomorphism $\theta \colon \VV \to \VV$, we say that $\theta$ is a \textbf{$\mathbf{C}$-endomorphism} if one of the following holds:
    \begin{enumerate}[(i)]
        \item $C$ is algebraic and $\Ker(\mipo(C)) = \VV$, that is, $\mipo(C)[\theta] = 0$.
        \item $C$ is transcendental and $\Ker(f^{C(f)}) = \Ker(f^{C(f)+1})$ for all $f \in \Kp{}$ with $C(f) < \infty$.
    \end{enumerate}
    We define $T^C_\theta := T_\theta \cup \set{\text{``$\theta$ is a $C$-endomorphism''}}$.
\end{definition}

\noindent The family $\set{T^C_\theta : C \in \Cc}$ might seem a bit arbitrary at first; however, notice that any consistent extension of the form
$$
    T_\theta \cup \Set{\sum\nolimits_{k}\bigcap\nolimits_{l}\Ker(\rho_{j, k, l}[\theta]) = \sum\nolimits_{k}\bigcap\nolimits_{l} \Ker(\eta_{j, k, l}[\theta]) : j \in \jj},
$$
where all sums and intersections are finite,  all the $\rho_{j, k, l}$'s and $\eta_{j, k, l}$'s are polynomials over $K$, and $\jj$ is a potentially infinite index set, is equivalent to some $T^C_\theta$ (see Corollary 2.13 in \cite{Chi25} - the proof heavily uses consequences of Bézout's identity). Also note that every $T^C_\theta$ is consistent, and that $T^C_\theta \not\equiv T^{C'}_\theta$ whenever $C \neq C'$ (see Lemma 2.19 in \cite{Chi25}). So the set $\Cc$ parametrizes all consistent extensions as described above, in some sense. Some concrete examples:
\begin{enumerate}[(i)]
    \item Let $C_\infty$ be transcendental with $C_\infty(f) = \infty$ for all $f \in \Kp{}$. One can check that $T_\theta^{C_\infty}$ is $T_\theta$. 
    \item Let $C_0$ be transcendental with $C_0(f) = 0$ for all $f \in \Kp{}$. One can check that $T_\theta^{C_0}$ is $T_\theta \cup \set{\text{``$\rho[\theta]$ is injective''} : \rho \in K[X] \setminus \set{0}}$. In an existentially closed model of $T_\theta^{C_0}$, the maps $\rho[\theta]$ are isomorphisms, so one can solve systems of equations of the form $\bigwedge_{k=1}^m \sum_{l=1}^n \rho_{k, l}[\theta](x_l) = y_k$ just like in a $K(X)$-vector space.
    \item Let $C_f$ be algebraic with $\mipo(C_f) = f \in \Kp{}$. One can, similarly to (ii), solve such systems of equations as in a $K[X]/(f)$-vector space. Here, however, one has the potential advantage that the sequence $\set{\theta^i(v) : i \in \omega}$ is already determined by $\set{\theta^i(v) : 0 \leq i < \deg(f)}$.
\end{enumerate}
These are, in some cases, the ``easiest'' examples to work with, and it can often be helpful to first work with one of these kernel configurations and then turn to the general case. By contrast, the ``hardest'' kernel configurations to work with are those for which $\set{f \in \Kp{} : 0 < C(f) < \infty}$ is infinite.

In \cite{Chi25}, we also showed that every $T^C_\theta$ is inductive (see Lemma 3.2 there). Hence, the model companion of each $T^C_\theta$ exists if and only if the class of existentially closed models of $T^C_\theta$ is elementary. In this case, the model companion is exactly the axiomatization.

One can easily check that if $C$ and $C'$ are algebraic kernel configurations, then $C = C'$ if and only if $\mipo(C) = \mipo(C')$. Furthermore, we have $\deg(C) = \deg(\mipo(C))$.

\begin{fact} \label{remark_alg_kc}
    Let $C$ be algebraic and let $\theta$ be a $C$-endomorphism. Then
    $\Ker(f^{C(f)}) = \Ker(f^{C(f)+1})$ holds for all $f \in \Kp{}$.
\end{fact}

\begin{fact} \label{fact_trivvivi}
    If $C$ is \textbf{trivial}, that is, if $\deg(C) = 1$, then the models of $T_\theta^C$ and $T$ are interdefinable.
    Hence, the model companion of $T_\theta^C$ is $T_\theta^C$ itself.
\end{fact}

\noindent We will often implicitly assume that $C$ is non-trivial.

\begin{notation}
    We introduce a few more notations for working with a kernel configuration $C \in \Cc$:
    \begin{enumerate}[(i)]
        \item Given $f \in \Kp{}$ with $C(f) < \infty$, we write $f^C$ instead of $f^{C(f)}$, $f^{C+k}$ instead of $f^{C(f) + k}$, and so on.
        \item Given a finite set $F \subseteq \Kp{}$ with $C(f) < \infty$ for all $f \in F$, we set \hbox{$F^C := \prod_{f \in F} f^C$}.
        \item We define the following subsets of $\Kp{}$:
        $$
        \Kp{C<\infty} := \set{f \in \Kp{} : C(f) < \infty}, \quad \Kp{0<C<\infty} := \set{f \in \Kp{} : 0 < C(f) < \infty},
        $$
        $$
        \Kp{C=0} := \set{f \in \Kp{} : C(f) = 0}, \quad \text{and} \quad \Kp{C=\infty} := \set{f \in \Kp{} : C(f) = \infty}.
        $$
    \end{enumerate}
\end{notation}

\noindent With the above, for any algebraic kernel configuration $C \in \Ccalg$, we have $\deg(C) = \deg(\mipo(C))$ and
    $$
    \mipo(C) = \prod\nolimits_{f \in \Kp{0<C<\infty}} f^C = (\Kp{0<C<\infty})^C.
    $$
\noindent One should note that any $\LKThe$-sentence that holds in $\TKvsTheC$ also holds in $T^C_\theta$ (here $\LKThe$ is the language of $K$-vector spaces with an endomorphism).
To be more precise, one has to modify the sentence accordingly, e.g., quantifiers of the form $\exists x$ must be replaced with $\exists x \in \VV$ and the formulas that define addition/scalar multiplication must be used instead of the function symbols in $\LKThe$.
In general, it also turns out that whenever a model $(\mm, \theta)$ of $T^C_\theta$ is existentially closed, $(\VV, \theta)$ is an existentially closed model of $\TKvsTheC$ (see Remark 3.4 in \cite{Chi25}).

\subsection{$C$-image-completeness}

\noindent Note that the condition of $\theta$ being a $C$-endomorphism does not (at least in the transcendental case) imply any kind of equations that involve the image of $\rho[\theta]$ for some $\rho \in K[X]$.
In Remark 2.21 in \cite{Chi25}, we discussed that it is very unlikely that considering expansions of $T_\theta$ that also impose equations on the images (or mixed equations with sums and intersections of both kernels and images) will lead to new model companions.
However, the following condition holds in any existentially closed model of $T^C_\theta$ and is fundamental to understanding the structure of these models:

\begin{definition} \label{def_C_image_comple}
    We say that an endomorphism $\theta \colon \VV \to \VV$ is \textbf{$\mathbf{C}$-image-complete} if it is a $C$-endomorphism and $\Image(f^{C+1}) = \Image(f^C)$  holds for all $f \in \Kp{C<\infty}$ (recall $\Image(f^C) := \Image(f^{C(f)}[\theta])$).
    We may call a model $(\mm, \theta) \models T_\theta$ \textbf{$\mathbf{C}$-image-complete} if the endomorphism $\theta$ is $C$-image-complete.
\end{definition}

\begin{fact} \label{fact_c_image_comple}
    The following statements hold:
    \begin{enumerate}[(i)]
        \item If $C$ is algebraic, then every $C$-endomorphism is $C$-image-complete (see Lemma 3.11 in \cite{Chi25}).
        \item If $C$ is any kernel configuration and $(\mm, \theta) \models T_\theta^C$ is existentially closed, then $(\mm, \theta)$, or equivalently $\theta$, is also $C$-image-complete (see Corollary 3.13 in \cite{Chi25}).
    \end{enumerate}
\end{fact}

\noindent Note that the converse of (ii) in Fact \ref{fact_c_image_comple} above does not hold.
The most important consequence of $C$-image-completeness is that we can decompose $\VV$ into definable direct summands as follows.

\begin{fact}[Lemma 3.14 in \cite{Chi25}] \label{lemma_decomposition}
    If $\theta \colon \VV \to \VV$ is $C$-image-complete and $F \subseteq \Kp{0<C<\infty}$ is a finite set, then we have
    $$
    \VV = \Image(F^C) \oplus \Ker(F^C) = \Image(F^C) \oplus \bigoplus\nolimits_{f\in F} \Ker(f^C).
    $$
    If $C$ is algebraic and $F = \Kp{0<C<\infty}$, then the summand $\Image(F^C)$ is $\set{0}$ and can therefore be omitted.
\end{fact}

\noindent In \cite{Chi25}, we defined additional endomorphisms of $\VV$ using these decompositions.

\begin{notation}
    For any $\rho \in K[X] \setminus \set{0}$, we let $\Fac(\rho) := \set{f \in \Kp{} : f \mid \rho}$ denote the set of all irreducible factors of $\rho$.
    As a convention, we set $\Fac(0) = \varnothing$.
\end{notation}

\begin{fact}[Lemma 3.15 in \cite{Chi25}] \label{fact_endo_gen}
    In the theory \hbox{$\TKvsThe \cup \set{\text{``$\theta$ is $C$-image-complete''}}$}, the following endomorphisms are definable:
    \begin{enumerate}[(i)]
        \item For $F \subseteq \Kp{0<C<\infty}$ finite, we define the \textbf{projection to the image of $\bm{F^C[\theta]}$} by
        $$
        \pi_{\Image(F^C)}(x) := \text{``the unique $u \!\in\! \Image(F^C)$ for which there is $v \in \Ker(F^C)$ with $x \!=\! u \!+\! v$''.}
        $$
        \item For $F \subseteq \Kp{0<C<\infty}$ finite, we define the \textbf{projection to the kernel of $\bm{F^C[\theta]}$} by
        $$
        \pi_{\Ker(F^C)}(x) := \text{``the unique $v \!\in\! \Ker(F^C)$ for which there is $u \in \Image(F^C)$ with $x \!=\! u \!+\! v$''.}
        $$
        We clearly have $\pi_{\Ker(F^C)} = 1[\theta] - \pi_{\Image(F^C)}$.
        \item For every monic polynomial $\eta \in K[X]$ with $\Fac(\eta) \subseteq \Kp{C<\infty}$, we define the \textbf{pseudo-inverse of $\bm{\eta[\theta]}$} by
        $$
        \eta[\theta]^{-1}(x) := \text{``the unique $u \in \Image(\Fac(\eta)^C)$ with $\eta[\theta](u) = \pi_{\Image(\Fac(\eta)^C)}(x)$''.}
        $$
        Notice that $\Fac(\eta)^C = (\Fac(\eta) \cap \Kp{0<C<\infty})^C$.
        In practice, we will also use $\eta[\theta]^{-1}$ for (non-zero) non-monic polynomials by setting $\eta[\theta]^{-1} := \lambda^{-1} \cdot (\eta/\lambda)[\theta]^{-1}$ for the leading coefficient $\lambda$ of $\eta$.
    \end{enumerate}
\end{fact}

\noindent Notice that whenever $\Fac(\eta) \cap \Kp{0<C<\infty} \neq \varnothing$, we obtain $\eta[\theta] \circ \eta[\theta]^{-1} = \pi_{\Image(\Fac(\eta)^C)} \neq \Id = 1[\theta]$, so the notation might be a bit misleading. Here $\Id$ is the identity on $\VV$.
It is also easy to see that $\Id = 1[\theta] = \pi_{\Image(f^C)} + \pi_{\Ker(f^C)}$. We consider the ring generated by all $\LKThe$-definable endomorphisms we have collected so far:

\begin{fact}[Theorem 3.18 in \cite{Chi25}] \label{theorem_r_c_def}
    We let $R_C$ be the set of all endomorphisms definable in the theory $\TKvsThe \cup \set{\text{``$\theta$ is $C$-image-complete''}}$ that are $\set{+, \circ}$-generated by
    $$
    \set{\rho[\theta] : \rho \in K[X]} \cup \set{\pi_{\Image(F^C)} : F \subseteq \Kp{0<C<\infty} \text{ finite}} \cup \set{\eta[\theta]^{-1} : \eta \text{ monic with }\Fac(\eta) \subseteq \Kp{C<\infty}}.
    $$
    The structure $(R_C, 0[\theta], 1[\theta], +, \circ)$ is a unitary commutative ring with $\Char(R_C) = \Char(K)$. In fact, it can also be seen as a $K$-algebra, as $K \subseteq R_C$ (identifying $\lambda \in K$ with the endomorphism $x \mapsto \lambda \cdot x$ which is $\lambda[\theta]$).
\end{fact}

\noindent We may sometimes write $0$ instead of $0[\theta]$ and $\Id$ or $1$ instead of $1[\theta]$. In particular, if we regard $R_C$ purely as a ring, we may write $(R_C, 0, 1, +, \cdot)$. The ring $R_C$ may again look complicated at first, but for the ``easiest to work with'' kernel configurations, we obtain the following:
\begin{enumerate}[(i)]
    \item $(R_{C_\infty}, 0, 1, +, \cdot) \simeq (K[X], 0, 1, +, \cdot)$ for the unique kernel configuration $C_\infty \in \Cctrans$, which satisfies \hbox{$C_\infty(f) = \infty$} for all $f \in \Kp{}$.
    \item $(R_{C_0}, 0, 1, +, \cdot) \simeq (K(X), 0, 1, + , \cdot)$ for the unique kernel configuration $C_0 \in \Cctrans$ with $C_0(f) = 0$ for all $f \in \Kp{}$. 
    Notice that this is a field.
    \item $(R_{C}, 0, 1, +, \cdot) \simeq (K[X]/(\mipo(C)), 0, 1, +, \cdot)$ for all $C \in \Ccalg$.
    This also implies that our ring $(R_{C}, 0, 1, +, \cdot)$ is a field for all algebraic kernel configurations $C$ with $\mipo(C)$ being irreducible.
\end{enumerate}
Other examples of $R_C$ can be found in Corollary 3.25 in \cite{Chi25}. 
There, the case where $C$ is transcendental and $\Kp{0<C<\infty}$ is infinite again turns out to be the most complicated. Multiplication rules, such as $\rho[\theta] \circ \pi_{\Image(F^C)} = \rho[\theta]$ if $F^C \mid \rho$, can be found in Lemma 3.21 in \cite{Chi25}.  

\begin{fact}[see Remark 3.26 in \cite{Chi25}] \label{fact_when_field}
    $R_C$ is a field if and only if $C = C_0$, as in (ii) above, or if $C$ is algebraic with $\mipo(C)$ irreducible. 
\end{fact}

\noindent Also note that the elements of $R_C$ are, as defined in Fact \ref{theorem_r_c_def}, definable functions in the theory $\TKvsThe \cup \set{\text{``$\theta$ is $C$-image-complete''}}$.
By this, we mean that the elements of $R_C$ are equivalence classes of $\LKThe$-formulas modulo the theory \hbox{$\TKvsThe \cup \set{\text{``$\theta$ is $C$-image-complete''}}$} that define an endomorphism in every model of $\TKvsThe \cup \set{\text{``$\theta$ is $C$-image-complete''}}$.
So, in order to prove $r = r'$, we need to show
$$
    r^{(\VV, \theta)} = r'^{(\VV, \theta)} \quad \text{for all $(\VV, \theta) \models \TKvsThe \cup \set{\text{``$\theta$ is $C$-image-complete''}}$},
$$
and, in order to prove $r \neq r'$, we need to find $(\VV, \theta) \models \TKvsThe \cup \set{\text{``$\theta$ is $C$-image-complete''}}$ with
$
    r^{(\VV, \theta)} \neq r'^{(\VV, \theta)}.
$
In Remark 3.20 in \cite{Chi25}, we showed that $r \neq r'$ implies $r^{(\VV, \theta)} \neq r'^{(\VV, \theta)}$ if $(\VV, \theta)$ is an existentially closed model of $\TKvsTheC$.
In Theorem 4.10 in \cite{Chi26} we proved that in an existentially closed model of $T_\theta^C$, any $\LKThe$-definable endomorphism of $\VV$ is, in fact, in $R_C$.

\begin{definition}
    We define $\LRC$ as the language of (left) $R_C$-modules (with $R_C$ as in Fact \ref{theorem_r_c_def}), i.e., $\LRC = (0, +, (r)_{r \in R_C})$, where each $r \in R_C$ is treated as a unary function symbol.
\end{definition}

\noindent We will write $r(x)$ instead of $r \cdot x$, since $r$ will usually be a function such as $\pi_{\Image(F^C)}$.
Given a model $(\VV, \theta) \models \TKvsThe \cup \set{\text{``$\theta$ is $C$-image-complete''}}$, we define an $\LRC$-structure on $\VV$ using the definable functions of which $R_C$ consists. This structure is obviously interdefinable with $(\VV, \theta)$.

\begin{fact}[Common Base Theorem - Theorem 3.24 in \cite{Chi25}]\label{theorem_r_c_element}
    Given $r_1, \dots, r_q \in R_C$ and any finite subset $F_0 \subseteq \Kp{0<C<\infty}$, we can write
    $$
    r_i = \rho_{F,i}[\theta] \circ \eta[\theta]^{-1} \circ \pi_{\Image(F^C)} + \sum\nolimits_{f\in F} \rho_{f,i}[\theta] \circ \pi_{\Ker(f^C)}
    $$
    for all $i \in \set{1, \dots, q}$, where
    \begin{enumerate}[(i)]
        \item $F \subseteq \Kp{0<C<\infty}$ is finite with $F_0 \subseteq F$,
        \item $\eta$ is a monic polynomial with $\Fac(\eta) \subseteq F \cup \Kp{C=0}$,
        \item the $\rho_{F,i}$ satisfy both $\Fac(\rho_{F,i}) \subseteq F \cup \Kp{C=0}\cup\Kp{C=\infty}$ and $\gcd(\rho_{F,1}, \dots, \rho_{F,q}, \eta) = 1$,
        \item the $\rho_{f,i}$ are polynomials with $\deg(\rho_{f,i}) < \deg(f^C)$.
    \end{enumerate}
    If $C$ is algebraic, we can furthermore choose $F = \Kp{0<C<\infty}$, resulting in
    $
    r_i = \sum\nolimits_{f\in F} \rho_{f,i}[\theta] \circ \pi_{\Ker(f^C)}
    $.
\end{fact}

\noindent As previously mentioned, the ring $R_C$ can be seen as a $K$-algebra, and therefore also as a $K$-vector space. We now compute its dimension:

\begin{remark} \label{rem_dim_r_C}
    The following holds:
    \begin{enumerate}[(i)]
        \item If $C$ is algebraic, then $\dim_K(R_C) = \deg(\mipo(C))$.
        \item If $C$ is transcendental, then $\dim_K(R_C) = \max\set{\omega, \Kp{C<\infty}}$.
    \end{enumerate}
\begin{proof}
    In the transcendental case, verify that
    \begin{align*}
        &\Big(X^i[\theta] \circ \Big( \prod\nolimits_{f\in F} f^{n(f)} \Big)[\theta]^{-1} : F \subseteq \Kp{C<\infty} \text{ finite}, 0 \leq i < \min\set{\deg(f) : f \in F}, n \in (\NN_{>0})^F \Big) \\
        &\quad\quad{}^\frown\Big( X^i[\theta] \circ \pi_{\Ker(f^C)} : f \in \Kp{0<C<\infty}, 0 \leq i <\deg(f^C) \Big)
    \end{align*}
    is a $K$-basis of $R_C$ (here we set $\min(\varnothing) = \infty$). To do this, one can use the \commonBaseTheorem{} above and the multiplication rules from Lemma 3.21 in \cite{Chi25}. From there, one can easily calculate the cardinality. For the algebraic case, one just has to omit the first line and recall $\mipo(C) = \prod_{f \in \Kp{0<C<\infty}} f^C$.
\end{proof}
\end{remark}

\subsection{Existentially closed models and first-order axiomatization}

We now state the characterization of the existentially closed models of $T^C_\theta$ from \cite{Chi25}. We start with the remaining ingredients:

\begin{definition}
    \label{def_param_c_sequence_system} A \textbf{parametrized $\mathbf{C}$-sequence-system} is an $\LKThe$-formula of the form
    $$
    S(\ux; \uy) = \bigwedge\nolimits_{k=1}^n f_k^{q_k}[\theta](x_{\ldd, k}) = y_k
    $$
    with $\ux := \ux_\li\ux_\ld := (x_{\lii, k} : 1 \leq k \leq m)(x_{\ldd, k} : 1 \leq k \leq n)$ and $\uy = (y_1, \dots, y_n)$ that satisfies the following conditions:
    \begin{enumerate}[(i)]
        \item If $C$ is algebraic, then $m = 0$.
        \item $f_k \in \Kp{0<C}$ and $q_k \in \set{q \in \NN : 0 < q \leq C(f_k)}$ hold for all $k \in \set{1, \dots, n}$.
    \end{enumerate}
    
\end{definition}
\noindent We will always denote parametrized $C$-sequence-systems by the letter $S$.
Given such a parametrized $C$-sequence-system $S(\ux; \uy)$, we assume that everything is as above, that is, $m$, $n$, and the $f_k$'s and $q_k$'s are defined implicitly, and we set $\ux = \ux_\li\ux_\ld$ as above.
When we partition $\ux = \ux_\li\ux_\ld$ as above, we think of:
    \begin{enumerate}[(i)]
        \item $\ux_\li$ as the linearly independent part of $\ux$, since $S(\ux; \uy)$ does not imply any linear dependencies for the sequence $(\theta^i(x_{\lii, k}) : 1 \leq k \leq m, i \in \omega)$;
        \item $\ux_\ld$ as the linearly dependent part of $\ux$, since the formula $S(\ux; \uy)$ implies that the sequence $(\theta^i(x_{\ldd, k}) : i \in \omega)$ is linearly dependent over $\spanA{y_k}{K}$ for each $k \in \set{1, \dots, n}$.
    \end{enumerate}
The names $\ux_\li$ and $\ux_\ld$ for these tuples are abbreviations for linearly independent and linearly dependent. Notice that in the algebraic case, we require $\ux_\li$ to be empty, which makes sense, as we have $\sum_{i=0}^{\deg(\mipo(C))} (\mipo(C))_i \cdot \theta^i(v) = 0$ for any $v \in \VV$ in that case.

\begin{definition}
    \label{def_c_sequence_system} \label{def_compatible} Let $S(\ux; \uy) = \bigwedge_{k=1}^n f_k^{q_k}[\theta](x_{\ldd, k}) = y_k$ be a parametrized $C$-sequence-system as in Definition \ref{def_param_c_sequence_system}, and let $(\VV, \theta) \models \TKvsThe$ be given.
\begin{enumerate}[(i)]
    \item We say that a tuple $\uu = (u_1, \dots, u_n) \in \VV$ is \textbf{compatible} with $S$ if $u_k \in \Ker(f_k^{C-q_k})$ for every $k \in \set{1, \dots, n}$ with $f_k \in \Kp{0<C<\infty}$.
    \item A \textbf{$\mathbf{C}$-sequence-system} over $(\VV, \theta)$ is an $\LKThe(\VV)$-formula of the form
    \hbox{$
    S(\ux) = S'(\ux; \uu)
    $}
    where $S'$ is a parametrized $C$-sequence-system and $\uu \in \VV$ is compatible with $S'$.
\end{enumerate}
\end{definition}

\noindent We will also denote $C$-sequence-systems over some $(\VV, \theta) \models \TKvsTheC$ by the letter $S$.
Notice that a $C$-sequence-system over $(\VV, \theta)$ is also a $C$-sequence-system over any extension $(\VV', \theta')$ that is also a model of $\TKvsThe$.

\begin{definition}[Placeholder notation] \label{def_placeholder_notation}
    Let $\ux = (x_k : k \in \kk)$ be a tuple of variables.
    We define the \textbf{placeholder sequence} \hbox{$\uxvec := (x^i_k : k \in \kk, i \in \omega)$} to be a new tuple of variables.
    We call each $x^i_k$ a \textbf{placeholder variable} or a \textbf{placeholder} for $\theta^i(x_k)$.
    We furthermore define:
    \begin{enumerate}[(i)]
        \item $\ux^i := (x_k^i : k \in \kk)$ for each $i \in \omega$, and
        \item $\xvec_k := (x^i_k : i \in \omega)$ for each $k \in \kk$.
    \end{enumerate}
    We may sometimes write $(\ux^i : i \in \omega)$ or $(\xvec_k : k \in \kk)$ instead of $\uxvec$.
    For a singleton $x$, we similarly define $\xvec := (x^i : i \in \omega)$.
    If a formula $\psi(\uxvec; \uw)$ or a term $\lambda(\uxvec{}; \uw)$ is given, we define:
    \begin{enumerate}[(i)]
        \setcounter{enumi}{2}
        \item $\psi_\theta(\ux; \uw) := \psi((\theta^i(x_k) : k \in \kk, i \in \omega); \uw)$.
        \item $\lambda_\theta(\ux; \uw) := \lambda((\theta^i(x_k) : k \in \kk, i \in \omega); \uw)$.
    \end{enumerate}
\end{definition}

\begin{definition} \label{def_formual_bounded}
    Let $S(\ux; \uy)= \bigwedge\nolimits_{k=1}^n f_k^{q_k}[\theta](x_{\ldd, k}) = y_k$ be a parametrized $C$-sequence-system as in Definition \ref{def_param_c_sequence_system}.
    If $\psi(\uxvec; \uw)$ is a formula, then we say $\psi(\uxvec; \uw)$ is \textbf{bounded} by $S$ if one of the following equivalent conditions holds:
        \begin{enumerate}[(i)]
            \item For all $k \in \set{1, \dots, n}$, the variable $x^i_{\ldd, k}$ does not appear in $\psi(\uxvec; \uw)$ for \hbox{$i \geq \deg(f_k^{q_k})$}.
            \item For all $k \in \set{1, \dots, n}$, the term $\theta^i(x_{\ldd, k})$ does not appear in $\psi_\theta(\ux; \uw)$ for \hbox{$i \geq \deg(f_k^{q_k})$}.
        \end{enumerate}
    We also say that $\psi_\theta(\ux; \uw)$ is \textbf{bounded} by $S$ if $\psi(\uxvec; \uw)$ is bounded by $S$. Furthermore, we say that a formula is \textbf{bounded} by a $C$-sequence-system $S$ (i.e., a parametrized $C$-sequence-system with some compatible parameters plugged in) if it is bounded by the underlying parametrized $C$-sequence-system.
\end{definition}

\noindent In practice, for a formula $\psi(\uxvec)$ to be bounded by $S$ means that no subterm of the form $\theta^i(x_{\ldd, k})$ appearing in $\psi_\theta(\ux)$ can be replaced by applying a Euclidean division with the equation $f_k^{q_k}[\theta](x_{\ldd, k}) = y_k$ from $S(\ux; \uy)$. Indeed, if $i \geq \deg(f^{q_k}_k)$, then we could replace $\theta^i(x_{\ldd, k})$ with $\chi[\theta](y_k) + r[\theta](x_{\ldd, k})$, where $\chi, r \in K[X]$ are the unique polynomials  satisfying $\chi \cdot f^{q_k}_k + r = X^i$ and $\deg(r) < \deg(f^{q_k}_k) \leq i$.

Now that we have all ingredients, we can state the characterization of existentially closed models of $T^C_\theta$:

\begin{theorem}[Theorem 3.32 in \cite{Chi25}] \label{theorem_big_characterization}
    $(\mm, \theta) \models T^C_\theta$ is existentially closed if and only if it is $C$-image-complete and 
    $$
        (\mm, \theta) \models \exists \ux \in \VV : \psi_\theta(\ux) \wedge S(\ux)
    $$
    holds for any $C$-sequence-system $S(\ux)$ over $(\VV, \theta)$ and $L(M)$-formula $\psi(\uxvec)$ that is bounded by $S$ and does not imply any finite disjunction of non-trivial linear dependencies in $\uxvec$ over $\VV$.
\end{theorem}

\noindent We give two examples where the characterization simplifies quite a lot:
\begin{enumerate}[(i)]
    \item Fix $f \in \Kp{}$ and let $C_f$ be the unique algebraic kernel configuration with $\mipo(C_f) = f$.
    A model $(\mm, \theta) \models T^{C_f}_\theta$ is existentially closed if and only if
    $$
        (\mm, \theta) \models \exists \ux \in \VV : \psi(\theta^0(\ux), \dots, \theta^{\deg(f)-1}(\ux))
    $$
    holds for every $L(M)$-formula $\psi(\ux^0, \dots, \ux^{\deg(f)-1})$ that does not imply any finite disjunction of non-trivial linear dependencies in $\ux^0, \dots, \ux^{\deg(f)-1}$ over $\VV$.
    This is Theorem 3.33 in \cite{Chi25}.
    
    \item Let $C_0$ be the unique transcendental kernel configuration with $C_0(f) = 0$ for all $f \in \Kp{}$.
    A model $(\mm, \theta) \models T^{C_0}_\theta$ is existentially closed if and only if $\rho[\theta]$ is invertible for every $\rho \in K[X] \setminus \set{0}$ and
    $$
        (\mm, \theta) \models \exists \ux \in \VV : \psi_\theta(\ux)
    $$
    holds for every $L(M)$-formula $\psi(\uxvec)$ that does not imply any finite disjunction of non-trivial linear dependencies in $\uxvec$ over $\VV$.
    This is Theorem 3.34 in \cite{Chi25}.
\end{enumerate}

\noindent The next step is to first-order axiomatize this characterization when possible.
For this, we need the following two families of formulas:

\begin{fact}[Lemma 3.36 in \cite{Chi25}]
    Given a parametrized $C$-sequence-system $S(\ux; \uy)$, there is a $\LKThe$-formula $\delta_S(\uy)$ such that $(\mm, \theta) \models \delta_S(\uu)$ holds if and only if $\uu$ is compatible with $S$.
\end{fact}

\begin{definition}[see Definition 1.11 in \cite{dEl21b}] \label{def_hfour}
    We say that $T$ (with the specific choice of $\VV$) satisfies $(\operatorname{H4})$ if, for every $L$-formula $\psi(\ux; \uw)$, there is an $L$-formula $\sigma_\psi(\uw)$ such that, for all $\mm \models T$ and $\ud \in M$, we have $\mm \models \sigma_\psi(\ud)$ if and only if one of the following two equivalent conditions holds:
    \begin{enumerate}[(i)]
        \item The formula $\psi(\ux; \ud)$ implies no finite disjunction of non-trivial linear dependencies in $\ux$ over $\VV$.
        \item There are an elementary extension $\mm' \succ \mm$ and a tuple $\uv' \in \VV'$ linearly independent over $\VV$ such that $\mm' \models \psi(\uv'; \ud)$.
    \end{enumerate}
\end{definition}

\begin{theorem}[Theorem 3.39 in \cite{Chi25}] \label{theorem_first_oder}
If $T$ satisfies \Hfour{}, then $T_\theta^C$ has a model companion $T\theta^C$, i.e., a first-order axiomatization of the class of existentially closed models.
It is axiomatized by the theory \hbox{$T_\theta \cup \set{\text{``$\theta$ is $C$-image-complete''}}$} together with the sentence
$$
\forall \uw:\forall \uy \in \VV : (\sigma_\psi(\uw) \wedge \delta_S(\uy)) \rightarrow \exists \ux \in \VV : \psi_\theta(\ux; \uw) \wedge S(\ux; \uy)
$$
for every parametrized $C$-sequence-system $S(\ux; \uy)$ and every $L$-formula $\psi(\uxvec; \uw)$ that is boun\-ded by $S$.
\end{theorem}

\begin{example} \label{examples_hfour}
    \Hfour{} holds in the following settings:
    \begin{enumerate}[(i)]
        \item The theory $\TKvs$ with the vector space $(\VV, +, 0, (\lambda \cdot)_{\lambda \in K})$ being the entire structure satisfies \Hfour{}.
        This follows easily from quantifier elimination.
        \item Any complete and model-complete o-minimal theory $T$ extending the theory of divisible ordered abelian groups, with $(\VV, +, 0, (q \cdot)_{q \in \QQ})$ being a subinterval of the line (but not necessarily a subgroup) and continuous operations, satisfies \Hfour{} if and only if there is no infinite definable family of germs of $(\VV, +, 0, (q \cdot)_{q \in \QQ})$-endomorphisms at $0_\VV$; combine Theorem 2.4 and Lemma 2.9 in \cite{Blo23}.
        \item Any complete and model-complete o-minimal expansion of $\RCF{}$ with $(\VV, +, 0, (q \cdot)_{q \in \QQ})$ given by $(R_{>0}, \cdot, 1, (x \mapsto x^q)_{q\in \QQ})$ satisfies \Hfour{} if and only if no partial exponential function is definable.
        This is a special case of (ii); see the proof of Theorem A in \cite{Blo23}.
        \item Let $\FF_q$ be a finite field with $q = p^r$.
        If $T$ expands the theory of an $\FF_q$-vector space and $\VV$ is that vector space, then $T$ satisfies \Hfour{} if and only if it eliminates $\exists^\infty$; see the proof of Theorem 5.2 in \cite{dEl21b}.

        For expansions of the theories $\ACF{}_p$, $\operatorname{SCF}_{p, e}$ ($e$ either finite or infinite), $\operatorname{Psf}_p$, $\operatorname{ACFA}_p$, and $\operatorname{DCF}_p$, this implies that whenever $\FF_q$ is contained in every model as constants, \Hfour{} holds with the $\FF_q$-vector space given by addition.
        For more details, see Example 5.10 in \cite{dEl21b}.
        \item If $T$ eliminates $\exists^\infty x \in \VV$ (see Definition 3.3 in \cite{Chi26}; recall that $\VV$ might not be the entire domain, or even an $n$-ary set), and the following condition holds:
        \begin{itemize}
            \item[(V)] For any $\mm' \succ \mm \models T$ and any tuple $\uv'$ in $\VV'$, $\uv'$ is $\acl_L$-independent over $M$ if and only if it is linearly independent over $\VV$.
        \end{itemize}
        then $T$ satisfies \Hfour{} (see Lemma 3.4 in \cite{Chi26}).
        \item If $T$ is a linear o-minimal expansion of the theory of ordered groups, then $T$ satisfies \Hfour{} (see Lemma \ref{lemma_omin_sat_hfour}).
    \end{enumerate}
\end{example}

\noindent We also have a necessary condition for the existence of the model companion, and therefore also for \Hfour{}:

\begin{fact}[Theorem 3.6 in \cite{Chi26}] \label{obser_no_elim_exist_inf_no_model_companion}
    If $T$ does not eliminate $\exists^\infty x \in \VV$, then the model companion of $T^C_\theta$ does not exist for any non-trivial $C \in \Cc$.
\end{fact}

\noindent Assuming \Hfour{}, we obtained some results in \cite{Chi25b} for $T\theta^C$ regarding completions, the algebraic closure, and quantifier elimination.

\begin{definition} \label{def_cl_theta}
    Let $(\mm, \theta)$ be $C$-image-complete and $A \subseteq M$. We let $\cl_\theta(A)$ denote the smallest set containing $A$ closed under $\acl_L$ and multiplication by elements of $R_C$.
\end{definition}

\begin{fact}[Theorem 4.16 and Theorem 4.20 in \cite{Chi25b}] \label{theorem_nes_cond_fixed} \label{theorem_acl}
     Assume that $T$ satisfies \Hfour{}. For any $(\mm, \theta) \models T\theta^C$ and any set $A \subseteq M$, we have $\acl_{L_\theta}(A) = \cl_\theta(A)$. Moreover, the following holds:
    \begin{enumerate}[(i)]
        \item If $\acl_L$ does not have the exchange property in $T$, then $\acl_{L_\theta}$ does not have the exchange property in $T\theta^C$.
        \item Assume that $C$ is non-trivial and that $\acl_L$ has the exchange property in $T$. If $\VV$ is not one-dimensional (with respect to the dimension induced by $\acl_L$), then $\acl_{L_\theta}$ does not have the exchange property in $T\theta^C$.
        \item If the ring $R_C$ is not a field, then $\acl_{L_\theta}$ does not have the exchange property in $T\theta^C$.
    \end{enumerate}
\end{fact}

\begin{fact}[Theorem 4.10 in \cite{Chi25b}] \label{theorem_qe}
    Let $L'$ be a language interdefinable with $L$ (modulo $T$) such that $T$ has quantifier elimination in $L'$ and $\spanA{A}{L'} = \acl_L(A)$ for every $A \subseteq \mm \models T$.
    Then $T\theta^C$ has quantifier elimination in the language $L' \cup \LRC$.
\end{fact}

\begin{example}[Example 4.11, Observation 4.12, and Example 4.17 in \cite{Chi25b}]
    \label{example_kvs} The theory $\TKvs\theta^C$ is a complete theory, has quantifier elimination in the language of $R_C$-modules, and we have $\acl_{\LKThe} = \dcl_{\LKThe} = \cl_\theta$. Furthermore, if $R_C$ is a field, then $\TKvs\theta^C$ is the theory of $R_C$-vector spaces and hence strongly minimal.
\end{example}

\begin{fact}[Lemma 4.19 in \cite{Chi25b}] \label{lemma_both_im_ker_inf}
    Let $C$ be a kernel configuration such that $R_C$ is not a field.
    Then there is $f \in \Kp{}$ such that both $\Image(f)$ and $\Ker(f)$ are infinite in every existentially closed model of $T^C_\theta$.
\end{fact}

\noindent In the previous papers, we often encountered situations where we knew that an $L(M)$-formula $\psi(x^0x^1)$ implies no finite disjunction of non-trivial linear dependencies over $\VV$ and wanted to show that $\psi(x\theta(x))$ is consistent (assuming that $(\mm, \theta) \models T\theta^C$ and that $C$ is non-trivial). If $C$ is transcendental, this is easy to show using our characterization of existentially closed models (Theorem \ref{theorem_big_characterization}). However, in the algebraic case, especially if every $f \in \Kp{0<C<\infty}$ satisfies $\deg(f) = 1$ and $C(f) = 1$, one needs to awkwardly decompose $x$ into each $\Ker(f^C)$-component in order to apply Theorem \ref{theorem_big_characterization} (see, e.g., the proof of Claim 4.20.1 in \cite{Chi25b}). Lemma \ref{lemma_rc_li_plus_li} below will ensure that we no longer have to do such unwieldy constructions.

\begin{fact}[Lemma 2.26 in \cite{Chi26}] \label{lemma_iter_lin_indep}
    Let $t_1(\ux), \dots, t_m(\ux)$ be $\LK(\VV)$-terms in $\ux = (x_1, \dots, x_n)$ that are linearly independent over $\VV$, and let $\uu = (u_1, \dots, u_m) \in \VV'$ be linearly independent over $\VV$, where $\mm' \succ \mm$.
    Then there exists $\uv \in \VV''$, for some elementary extension $\mm'' \succ \mm'$, such that $\uv$ is linearly independent over $\VV$ and $\mm'' \models t_k(\uv) = u_k$ for $k = 1, \dots, m$.
\end{fact}

\begin{lemma} \label{lemma_rc_li_plus_li}
    Suppose that $(\mm, \theta) \models T_\theta^C$ is an existentially closed model.
    Let \hbox{$r_1, \dots, r_n \!\in\! R_C$} be $K$-linearly independent.
    Let $\ux_1, \dots, \ux_n$ all be of length $m$, and suppose that an $L(M)$-formula $\psi(\ux_1, \dots, \ux_n)$ implies no finite disjunction of non-trivial linear dependencies in $\ux_1 \dots \ux_n$ over $\VV$.
    Then $(\mm, \theta) \models \exists \ux \in \VV : \psi(r_1(\ux), \dots, r_n(\ux))$.
\begin{proof}
    Assume that $C$ is transcendental.
    Using the \commonBaseTheorem{} (Fact \ref{theorem_r_c_element}), we may write
    $$
    r_i = \rho_i[\theta] \circ \eta[\theta]^{-1} \circ \pi_{\Image(F^C)} + \sum\nolimits_{f\in F} \rho_{i,f}[\theta] \circ \pi_{\Ker(f^C)}
    $$
    where $F \subseteq \Kp{0<C<\infty}$ is finite, $\Fac(\eta) \subseteq F \cup \Kp{C=0}$, and $\deg(\rho_{i,f}) < \deg(f^C)$ for all $f \in F$.
    Notice that $F$ and $\eta$ do not depend on $i$.
    We define $\ux_\li = (x_{\li, 1}, \dots, x_{\li, m})$ and, for each $f \in F$, define $\ux_f = (x_{f, 1}, \dots, x_{f, m})$ to be a tuple of length $m$.
    Using the linear independence of the $r_i$'s, we see that the terms
    $$
    t_{i, l}(\uxvec{}_{\li}(\uxvec{}_f : f \in F)) := (\rho_i \cdot F^C)[x_{\li, l}] + \sum\nolimits_{f\in F} \rho_{i, f}[x_{f, l}]
    $$
    are also $K$-linearly independent, where $\rho[x] := \sum_{i=0}^{\deg(\rho)} (\rho)_i \cdot x^i$.
    Define $\ut_{i}(\uxvec{}_{\li}(\uxvec{}_f : f \in F))$ as the tuple consisting of all $t_{i, l}(\uxvec{}_{\li}(\uxvec{}_f : f \in F))$.
    Now, by Fact \ref{lemma_iter_lin_indep}, we see that
    $$
    \psi'(\uxvec{}_\li(\uxvec{}_f : f \in F)) := \psi(\ut_1(\uxvec{}_\li(\uxvec{}_f : f \in F)), \dots, \ut_n(\uxvec{}_\li(\uxvec{}_f : f \in F)))
    $$
    implies no finite disjunction of non-trivial linear dependencies in $\uxvec{}_\li(\uxvec{}_f : f \in F)$ over $\VV$.
    Since $\deg(\rho_{i, f}) < \deg(f^C)$ for all $i$ and $f$, it is also easy to see that $\psi'(\uxvec{}_\li(\uxvec{}_f : f \in F))$ is bounded by the $C$-sequence-system
    $$
    S(\ux_\li(\ux_f : f \in F)) := \bigwedge\nolimits_{f\in F} \bigwedge\nolimits_{l=1}^m f^C[\theta](x_{f, l}) = 0.
    $$
    With our characterization of existentially closed models (Theorem \ref{theorem_big_characterization}), we obtain a realization $\uv_\li(\uv_f : f \in F) \in \VV$ of the formula
    $$
    \psi'_\theta(\ux_\li(\ux_f : f \in F)) \wedge S(\ux_\li(\ux_f : f \in F)).
    $$
    Define $\uv := (\eta \cdot F^C)[\theta](\uv_\li) + \sum_{f\in F} \uv_f$.
    Plugging in all definitions, we see that $\psi'_\theta(\uv_\li(\uv_f : f \in F))$ is $\psi(\ux_1, \dots, \ux_n)$ with each occurrence of $x_{i, l}$ replaced by
    \begin{align*}
        (\rho_i \cdot F^C)[\theta](v_{\li, l}) + \!\sum\nolimits_{f\in F} \rho_{i, f}[\theta](v_{f, l}) &= \rho_i[\theta] \!\circ\! \eta[\theta]^{-1} \!\circ \pi_{\Image(F^C)}(v_l) + \!\sum\nolimits_{f \in F} \rho_{i, f}[\theta] \circ \pi_{\Ker(f^C)}(v_l)\\ &= r_i(v_l).
    \end{align*}
    Therefore, we obtain $(\mm, \theta) \models \exists \ux \in \VV : \psi(r_1(\ux), \dots, r_n(\ux))$.
    If $C$ is algebraic, the proof is essentially the same, but without the tuple $\ux_\li$.
\end{proof}
\end{lemma}

\subsection{More on \texorpdfstring{\Hfour{}}{(H4)}}

In this section, we prove some additional properties of \Hfour{}.
First, we show that $T$ satisfying \Hfour{} with a vector space $\VV$ is preserved under definable vector space isomorphisms:

\begin{remark} \label{rem_vec_iso}
    Suppose $(\VV_1, 0_1, +_1, (\lambda\cdot_1)_{\lambda\in K})$ and $(\VV_2, 0_2, +_2, (\lambda\cdot_2)_{\lambda\in K})$ are two $\varnothing$-definable $K$-vector spaces with which $T$ satisfies our general assumption (see Section \ref{sec_setting}).
    Furthermore, suppose that the two vector spaces are isomorphic in $T$ via a $\varnothing$-definable map $\iota \colon \VV_1 \to \VV_2$.
    Let $T^C_{1,\theta}$ denote the theory $T^C_\theta$ obtained by choosing $\VV = \VV_1$, and define $T^C_{2,\theta}$ similarly with $\VV = \VV_2$.
    The following holds:
    \begin{enumerate}[(i)]
        \item A model $(\mm, \theta) \models T^C_{1,\theta}$ is existentially closed if and only if the structure $(\mm, \iota \circ \theta \circ \iota^{-1})$ is an existentially closed model of $T^C_{2,\theta}$.
        \item The theory $T$ satisfies \Hfour{} with $\VV = \VV_1$ if and only if $T$ satisfies \Hfour{} with $\VV = \VV_2$.
    \end{enumerate}
\begin{proof}
    This is easy to verify using our characterization of existentially closed models and the fact that whenever $\psi(\ux; \ud)$ implies no finite disjunction of non-trivial linear dependencies in $\ux$ over $\VV_1$, the formula $\psi(\iota^{-1}(\ux); \ud)$ implies no finite disjunction of non-trivial linear dependencies in $\ux$ over $\VV_2$.
\end{proof}
\end{remark}

\noindent Next, we show that \Hfour{} is preserved under reducts.
By the model completeness of $T$ (and the assumption that $T$ and $T'$ are deductively closed), we have $T = T'_{\restriction L}$ if and only if, for every $\mm \models T$, we can find some $\mm' \models T'$ with $\mm \prec \mm'_{\restriction L}$.

\begin{fact}[Lemma 3.7 in \cite{Chi25b}] \label{lemma_hfour_fml}
    If $T$ satisfies \Hfour{}, then one can choose the formula $\sigma_\psi(\uw)$ (see Definition \ref{def_hfour}) for each $L$-formula $\psi(\ux; \uw)$ with $\ux = (x_1, \dots, x_n)$ to be of the form
    $$
    \exists \ux \in \VV : \psi(\ux; \uw) \wedge \bigwedge\nolimits_{k=1}^m \neg\varphi_k\Big(\sum\nolimits_{l=1}^n \lambda_{k, l} \cdot x_l; \uw\Big),
    $$
    with $(\lambda_{k, 1}, \dots, \lambda_{k, n}) \in K^n \setminus \set{\uzero}$ and $\varphi_k(y; \uw)$ algebraic in $y$ for each $k \in \set{1, \dots, m}$.
\end{fact}

\noindent In the following, we will always assume that $T'$ is an $L'$-theory.

\begin{lemma} \label{lemma_transfer_hfour_to_reduct}
    Suppose that $T'$ is a theory that satisfies \Hfour{}, and that $T$ is a model-complete reduct of $T'$ that still defines the same vector space $\VV$.
    Then $T$ also satisfies \Hfour{}.
\begin{proof}
    Let $\psi(\ux; \uw)$ be an $L$-formula.
    By Fact \ref{lemma_hfour_fml}, there are  $L'$-formulas $\varphi_1(y; \uw), \dots, \varphi_m(y; \uw)$ and tuples $\ulambda{}_1, \dots, \ulambda{}_m \in K^{n} \setminus \set{\uzero}$ such that, for $\ud \in \mm' \models T'$, the formula $\psi(\ux; \ud)$ implies a finite disjunction of non-trivial linear dependencies in $\ux$ over $\VV'$ if and only if
    $$
    \forall \ux \in \VV : \psi(\ux; \ud) \rightarrow \bigvee\nolimits_{k=1}^m \varphi_k\Big(\sum\nolimits_{l=1}^n \lambda_{k, l} \cdot x_l; \ud\Big).
    $$
    Since $T'$ satisfies \Hfour{}, it also eliminates $\exists^\infty x \in \VV$.
    Thus, we can find $q \in \NN$ such that, for any $\ud \in \mm'$, we have $\sum_{k=1}^m |\varphi_k(\VV'; \ud)| < q$.
    Hence $\psi(\ux; \ud)$ implies a finite disjunction of non-trivial linear dependencies in $\ux$ over $\VV'$ if and only if
    \begin{align*}
        \forall \ux_1,\dots,\ux_q  \in \VV :& \Big(\bigwedge\nolimits_{i=1}^q \psi(\ux_i; \ud) \wedge \bigwedge\nolimits_{1 \leq i < j \leq q} \ux_i \neq \ux_j\Big) \\& \hspace{20pt} \rightarrow \Big(\bigvee\nolimits_{1 \leq i < j \leq q}\bigvee\nolimits_{k=1}^m \sum\nolimits_{l=1}^n \lambda_{k, l} \cdot x_{i,l} = \sum\nolimits_{l=1}^n \lambda_{k, l} \cdot x_{j,l}\Big).
    \end{align*}
    It is easy to verify that this formula cannot hold if $\psi(\ux; \ud)$ implies no finite disjunction of non-trivial linear dependencies in $\ux$ over $\VV'$.
    Since $\psi(\ux; \uw)$ is an $L$-formula, the formula above can be written as $\chi(\ud)$ for an $L$-formula $\chi(\uw)$.

    Now, given $\mm \models T$ with $\mm \prec \mm'_{\restriction L}$ and $\ud \in M$, one can easily see that $\psi(\ux; \ud)$ implies no finite disjunction of non-trivial linear dependencies in $\ux$ over $\VV$ if and only if $\psi(\ux; \ud)$ implies no finite disjunction of non-trivial linear dependencies in $\ux$ over $\VV'$, if and only if $\mm' \models \neg \chi(\ud)$, if and only if $\mm \models \neg \chi(\ud)$.
\end{proof}
\end{lemma}

\noindent The proof above shows that the formula $\sigma_\psi(\uw)$ from \Hfour{} for an $L$-formula $\psi(\ux; \uw)$ can be chosen as the same $L$-formula in both $T$ and $T'$.
By looking at the axiomatization of $T\theta^C$ (see Theorem \ref{theorem_first_oder}), we obtain the following:

\begin{theorem} \label{theorem_reduct_}
    If $T'$ satisfies \Hfour{} and $T$ is a reduct of $T'$, then $T\theta^C$ exists for every kernel configuration $C$ and is a reduct of $T'\theta^C$.
\end{theorem}

\section{o-Minimal Preliminaries} \label{sec_o_min_prelim}

\begin{definition}
    Let $L \supseteq \set{<}$ be a language extending the language of ordered sets.
    We say that an $L$-theory $T$ is \textbf{o-minimal} if it extends the theory of dense linear orders without endpoints and, for every model $\mm \models T$, every definable subset of $M$ with parameters is a finite union of points and intervals.
\end{definition}

\noindent Note that some authors also consider theories such as $\Th(\ZZ, <)$ to be o-minimal.
Also note that o-minimality implies that $(\VV, 0, +)$ is an infinite divisible abelian group, or, in other words, that $K$ has characteristic zero (see Fact \ref{fact_vectorpsace_char_0}).
The overall goal of this section is to prove that, when working in an o-minimal expansion $T$ of the theory of ordered groups, either $T$ is linear and $(\VV, 0, +)$ is $\varnothing$-definably isomorphic to $(M^d, 0, +)$, or $T$ is non-linear and there is a certain definable function from an open subset of $\VV^n$ to $\VV$ that is ``nowhere linear''.
With the possible exception of Theorem \ref{theorem_not_lin}, all uncited results and proofs are probably known or considered folklore.

\subsection{Cells and Dimension}

We briefly recall the definitions of cells and cylindrical definable cell decompositions, and collect a few standard consequences that will be used tacitly throughout this chapter.
Unless noted otherwise, everything in this subsection can be found in (or easily deduced from) the book by van den Dries \cite{Dri98}.
Throughout this subsection, assume that $\mm$ is an o-minimal structure and that $A \subseteq M$ is a set of parameters.

\begin{definition}[Cells and cylindrical definable cell decompositions] \label{def_o_min_cells}
    \textbf{Cells} and their \textbf{dimension} are defined recursively as follows:
    \begin{enumerate}[(i)]
        \item An $A$-definable cell $C \subseteq M$ is either a singleton $\set{a}$ with $a \in \dcl(A)$ or an open interval $(a, b)$ with $a, b \in \dcl(A) \cup \set{\pm\infty}$ and $a < b$.
        In the first case, the dimension of the cell is $0$, and in the second case, it is $1$.
        \item If $C \subseteq M^n$ is an $A$-definable cell and $f, g \colon C \to M \cup \set{\pm\infty}$ are continuous $A$-definable functions with $f < g$, then
        $$
        \Gamma(f) := \set{(\ux, y) \in C \times M : y = f(\ux)}
        $$
        (if $f \neq \pm\infty$) and
        $$
        (f, g) := \set{(\ux, y) \in C \times M : f(\ux) < y < g(\ux)}
        $$
        are $A$-definable cells in $M^{n+1}$.
        The cell $\Gamma(f)$ has dimension $\dim(C)$, and the cell $(f, g)$ has dimension $\dim(C) + 1$.
    \end{enumerate}
\end{definition}

\begin{definition}[Cylindrical definable cell decompositions] \label{def_o_min_cell_decomp}
    Let $\pi \colon M^n \to M^{n-1}$ be the projection to the first $n-1$ coordinates.
    A finite partition of $M^n$ into $A$-definable cells is an ($A$-definable) \textbf{cylindrical definable cell decomposition} of $M^n$ if it is obtained recursively as follows:
    \begin{enumerate}[(i)]
        \item For $n=1$, any finite partition $\mathcal{C}$ of $M$ into ($A$-definable) cells is an ($A$-definable) cylindrical definable cell decomposition of $M$.
        \item For $n > 1$, a finite partition $\mathcal{C}$ of $M^n$ into ($A$-definable) cells is an ($A$-definable) cylindrical definable cell decomposition of $M^n$ if there is an ($A$-definable) cylindrical definable cell decomposition $\mathcal{D}$ of $M^{n-1}$ such that $\pi(C) \in \mathcal{D}$ for all $C \in \mathcal{C}$ and for every $D \in \mathcal{D}$ there is a finite sequence of continuous $A$-definable functions
        $$
        f_0 < f_1 < \dots < f_q < f_{q+1} \colon D \to M \cup \set{\pm \infty}
        $$
        with $f_0 := -\infty$ and $f_{q+1} := \infty$, such that the cells of the decomposition of $M^n$ lying above $D$ are exactly
        $
        \Gamma(f_1), \dots, \Gamma(f_q), \; (f_0, f_1), \dots, (f_q, f_{q+1})
        $.
    \end{enumerate}
    Given a family of ($A$-definable) subsets of $M^n$, we say that an $A$-definable cylindrical definable cell decomposition of $M^n$ is \textbf{adapted} to that family if every cell is either contained in or disjoint from each member of the family.
    Similarly, given a family of $A$-definable functions with $A$-definable domains in $M^n$, we say that the decomposition is \textbf{adapted} to those functions if it is adapted to the domains of all functions and each function is continuous on every cell contained in its domain.
\end{definition}

\begin{fact} \label{fact_o_min_cell_decomp}
    Given finitely many $A$-definable subsets of $M^n$ and finitely many $A$-definable functions on $A$-definable subsets of $M^n$, there is an $A$-definable cylindrical definable cell decomposition of $M^n$ adapted to those sets and functions.
\end{fact}

\noindent Given any $A$-definable set $X \subseteq M^n$, one can define $\dim(X)$, the \textbf{dimension} of $X$, as the maximum of the dimensions of the cells contained in $X$ in any cell decomposition of $M^n$ adapted to $X$. By convention, we define $\dim(\varnothing) = -\infty$.
It is standard that this definition depends neither on $A$ nor on the choice of the cell decomposition.
Moreover, this is exactly the $\acl$-dimension (see Fact \ref{fact_o_min_acl_dcl} below).

\begin{fact} \label{fact_o_min_conn_comp}
    Every cell is definably connected, and every $A$-definable set has finitely many definably connected components, each of which is again $A$-definable.
\end{fact}

\begin{fact} \label{fact_o_min_dim}
    Let $X$ and $Y$ be subsets of $M^n$ definable with parameters.
    \begin{enumerate}[(i)]
        \item The set $X$ cannot be covered by finitely many sets definable with parameters and of smaller dimension.
        \item If $Y \subseteq X$ and $\dim(Y) = \dim(X)$, then $Y$ has non-empty interior in $X$ (with respect to the subset topology on $X$).
    \end{enumerate}
\end{fact}

\begin{fact} \label{cor_o_min_open_cells_dense}
    Let $X = C_1 \dotcup \dots \dotcup C_r$ be a partition into cells, and suppose that $C_1, \dots, C_q$ are precisely the cells that are open in $X$.
    Then
    $$
    X \subseteq \bigcup\nolimits_{i=1}^q \operatorname{Cl}(C_i).
    $$
    In particular, every non-empty (with respect to the subset topology) open subset of $X$ meets some $C_i$ with $1 \leq i \leq q$.
\end{fact}

\begin{fact}[\cite{PS86} and see Lemma 1.3 in \cite{Pil88}] \label{fact_o_min_acl_dcl}
    For every set $A$, one has $\acl(A) = \dcl(A)$, and this closure operator has the exchange property. Moreover, the definition of $\dim(X)$ given below Fact \ref{fact_o_min_cell_decomp}, which uses cell decompositions, coincides with the definition of dimension induced by $\acl$ (whenever $\acl$ induces a pregeometry).
\end{fact}

\begin{fact} \label{cor_o_min_indep_cell_open}
    Let $C \subseteq M^n$ be an $A$-definable cell, and let $\uc \in C$ be $\dcl$-independent over $A$.
    Then $C$ is open in $M^n$.
\end{fact}

\subsection{The \texorpdfstring{$t$-Topology}{t-Topology}}

We start with the fact, due to Pillay, that a group definable in any o-minimal structure can be equipped with a ``manifold'' structure.

\begin{fact}[\cite{Pil88}] \label{fact_t_top}
    Let $(G, e, \circ)$ be a $\varnothing$-definable group in an o-minimal theory $T$.
    Then $G$ has a unique $\varnothing$-definable manifold topology with respect to which $G$ is a topological group.
    We call this topology the \textbf{$\mathbf{t}$-topology}.
\end{fact}

\noindent In the above, a definable manifold topology is a topology induced by a finite atlas $\set{(S_i, \phi_i) : i \in \ii}$, where $G = \bigcup_{i \in \ii} S_i$ and the $\phi_i$'s are bijections between $S_i$ and open subsets of $M^{\dim(G)}$ such that the maps $\phi_i \circ \phi_j^{-1}$ are homeomorphisms between open subsets of $M^{\dim(G)}$.
A precise definition of a definable manifold topology can be taken from \cite{Ele07}, right above Fact 2.1.1 there.
The uniqueness follows easily from a translation argument and the fact that, given a second atlas $\set{(S'_j, \phi'_j) : j \in \jj}$, the map $\phi_i \circ \Id_G \circ \phi_j'^{-1}$ is a local homeomorphism somewhere for suitable $i$ and $j$.
Finally, notice that the atlas itself might not be $\varnothing$-definable, but the topology is, in the sense that there is a formula $\varphi(\ux; \uy)$ such that the set $\set{\varphi(\mm; \ud) : \ud \in M}$ is a basis of the topology.

\begin{fact} \label{fact_t_connected_finite_xd}
    Any $M$-definable subset $X$ of an $M$-definable group $(G, e, \circ)$ is a finite disjoint union of definably $t$-connected $M$-definable sets.
    In particular, $G/G_0$ is a finite group, where $G_0$ is the definably $t$-connected component of $e$ in $G$.
\begin{proof}
    This is essentially Lemma 2.9 in \cite{Pil88}.
    Notice that, given a definably $t$-connected definable subset $X \subseteq G$ and $g \in G$, the set $g \circ X$ must also be definably $t$-connected, so $G$ consists of finitely many translates of $G_0$.
\end{proof}
\end{fact}

\begin{fact} \label{fact_vectorpsace_char_0}
    Let $(\VV, 0, +, (\lambda \cdot)_{\lambda \in K})$ be an infinite $\varnothing$-definable $K$-vector space in an o-minimal theory $T$.
    Then $K$ has characteristic $0$.
    In particular, $(\VV, 0, +)$ is an infinite divisible abelian group.
\begin{proof}
    Suppose that $K$ has positive characteristic $p$.
    Then every element of the definably $t$-connected component $\VV_0$ of $0$ in $\VV$ is a $p$-torsion element.
    However, by Theorem C in \cite{Spi24}, there is a neighborhood of $0$ that contains no elements of order $\leq p$ besides $0$.
    This instantly implies $\VV_0 = \set{0}$.
    As $\VV/\VV_0$ is finite by Fact \ref{fact_t_connected_finite_xd} above, we conclude that $\VV$ is finite, a contradiction.
\end{proof}
\end{fact}

\noindent The following definition is due to Peterzil and Steinhorn, and gives us an analog of a compact group in the o-minimal setting:

\begin{definition}[Definition 1.1 in \cite{PS99}] \label{def_def_comp}
    We say that an $M$-definable group $(G, e, \circ)$ is \textbf{definably compact} if, given any continuous $M$-definable embedding $\sigma \colon (a, b) \to G$ from an interval into the group, the limits $\lim_{x \to a^+} \sigma(x)$ and $\lim_{x \to b^-} \sigma(x)$ exist in $G$ (limits and continuity are taken with respect to the order topology on $M$ and the $t$-topology on $G$).
\end{definition}

\noindent It is clear that definable compactness is preserved under taking reducts, as there are fewer curves $\sigma$ to check.

\subsection{Linear Structures} \label{sec_lin_str}

In this section, we recall the definition of linear o-minimal structures, which were first studied in \cite{LP93}, as well as some relevant facts. Before we do so, we would like to point out the following fact to the reader:

\begin{fact}[\cite{Dri98}]
    If $\mm = (M, e, \cdot, <, \dots)$ is an o-minimal expansion of an ordered group, then $(M, e, \cdot)$ is an infinite divisible abelian group and both $\cdot$ and $x \mapsto x^{-1}$ are continuous with respect to the order topology.
\end{fact}

\noindent So whenever we say that $\mm = (M, 0, +, <, \dots)$ is an o-minimal expansion of an ordered group, this also implies that $(M, 0, +)$ is an infinite divisible abelian group. With Fact \ref{fact_t_top} one can also show that the $t$-topology on $M$ as a group is the same as the order topology. We give some general definitions that will be used throughout this chapter:

\begin{definition} \label{def_linear_defs}
    Let $\mm = (M, 0, +, <, \dots)$ be an o-minimal expansion of an ordered group, and let $(G, e, \oplus)$ be an infinite divisible torsion-free abelian $M$-definable group.
    \begin{enumerate}[(i)]
        \item We say that an $M$-definable function $f \colon A \subseteq G^n \to G$ is \textbf{linear} on $U \subseteq A$ if, for all $\ug, \uh, \ut \in G^n$ with $\ug, \uh, \ug \oplus \ut, \uh \oplus \ut \in U$, the equality $f(\ug \oplus \ut) \ominus f(\ug) = f(\uh \oplus \ut) \ominus f(\uh)$ holds.
        We say that $f$ is \textbf{linear} if it is linear on its domain $A$.
        \item We say that an $M$-definable function $f \colon A \subseteq G^n \to G$ is \textbf{piecewise linear} if $A$ can be partitioned into finitely many $M$-definable sets on which $f$ is linear.
        \item We say that an $M$-definable function $f \colon A \subseteq G^n \to G$ is \textbf{locally linear} at $\uu \in A$ if there is an open subset $U \subseteq A$ (with respect to the subset topology induced by the $t$-topology) such that $\uu \in U$ and $f$ is linear on $U$.
        \item We say that an $M$-definable function $f \colon U \to G$ is a \textbf{partial homomorphism} if the domain $U \subseteq G^n$ is open and definably connected (with respect to the $t$-topology), $f$ is linear, $(e, \dots, e) \in U$, and $f((e, \dots, e)) = e$.
        Equivalently, one can also check that the equality $f(\ug \oplus \uh) = f(\ug) \oplus f(\uh)$ holds for all $\ug, \uh \in G^n$ with $\ug, \uh, \ug \oplus \uh \in U$.
        If $U$ is a subset of $G$, i.e., $n = 1$, then we call $f$ a \textbf{partial endomorphism}.
        \item We say that $(G, e, \oplus)$ is \textbf{linear} in $\mm$ if every $M$-definable function $f \colon A \subseteq G^n \to G$ is piecewise linear.
    \end{enumerate}
\end{definition}

\begin{remark} \label{remark_locally_lin_to_partial}
    Let $f \colon A \to G$ be an $M$-definable function with $A \subseteq G^n$ definably connected and open.
    The function $f$ is locally linear at $\ug \in A$ if and only if there exists an open neighborhood $U'$ of $(e, \dots, e)$ contained in $A \ominus \ug$ such that the function $\ut \mapsto f(\ug \oplus \ut) \ominus f(\ug)$ restricted to $U'$ is a partial homomorphism.
    If $f$ is linear on $A$, then the function $\ut \mapsto f(\ug \oplus \ut) \ominus f(\ug)$ is a partial homomorphism on $A \ominus \ug$.
\end{remark}

\noindent In this section, $(G, e, \oplus)$ will always be $(M, 0, +)$.

\begin{definition} \label{def_linear_new}
    An o-minimal expansion $\mm = (M, 0, +, <, \dots)$ of an ordered group is \textbf{linear} if the group $(M, 0, +)$ is linear in $\mm$.
\end{definition}

\noindent We need the following fact, by Peterzil and Starchenko, when working with non-linear o-minimal theories:

\begin{fact}[Theorem 1.2 of \cite{PS98}] \label{fact_def_field}
    If an o-minimal expansion $\mm = (M, 0, +, <, \dots)$ of an ordered group is not linear, then there is an $M$-definable real closed field $(K, \boxzero, \boxone, \boxplus, \boxdot, <)$ where $K \subseteq M$ is an interval and $<$ is the same order as on $M$.
\iffalse\begin{proof}
    Applying Theorem 1.2 of \cite{PS98} to the group interval $([-a, a], 0, +, <)$ for some $a > 0$ (and assuming, without loss of generality, that $\mm$ is sufficiently saturated), we see that one of the following holds:
    \begin{enumerate}[(i)]
        \item There is an ordered vector space $\mathcal{V} := (V, 0, +, (\gamma \cdot)_{\gamma \in D}, <, (u)_{u \in U})$ over an ordered division ring, expanded by some constants, an interval $[-b, b] \subseteq V$, and an order-preserving isomorphism of group intervals $\sigma \colon [-a, a] \to [-b, b]$ such that, given any $\set{a}$-definable subset $S \subseteq [-a, a]^n$, the set $\sigma(S)$ is $\varnothing$-definable in $\mathcal{V}$.
        \item There is a real closed field definable in a subinterval of $[-a, a]$, with the ordering induced by $<$.
    \end{enumerate}
    If (ii) holds for some $a > 0$, we are done.
    Suppose, toward a contradiction, that (i) holds for all $a > 0$.
    Given $a > 0$ and a $\varnothing$-definable function $f \colon A \to M$ with $A \subseteq M^n$ also $\varnothing$-definable, we can easily verify point (ii) of Remark \ref{remark_linear_iff}.
    However, this implies that $\mm$ is linear, contradicting our assumption.
\end{proof}\fi
\end{fact}

\noindent For the rest of this section, we assume that $\mm = (M, 0, +, <, \dots)$ is a linear o-minimal expansion of an ordered group and define $T := \Th(\mm)$.
In particular, $T$ is complete.

\begin{definition} \label{def_part_endo}
    Given a $\varnothing$-definable partial endomorphism $g$, we define the \textbf{total function} of $g$ as
    $$
    \hat{g} \colon M \to M; \quad x \mapsto \begin{cases}
        g(x) & \text{if $x \in \operatorname{dom}(g)$} \\
        0 & \text{otherwise.}
    \end{cases}
    $$
\end{definition}

\noindent Since $T = \Th(\mm)$ is complete and partial endomorphisms are $\varnothing$-definable, it is clear that we obtain the same partial endomorphisms for different models of $T$.

\begin{fact}[Corollary 6.3 in \cite{LP93}] \label{fact_lin_qe}
    $T$ has quantifier elimination in the language
    $$
    \hat{L} := \set{0, +, <, (\hat{g} : \text{$g$ is a $\varnothing$-definable partial endomorphism}), (\mathfrak{c} : \mathfrak{c} \in \dcl_L(\varnothing))}.
    $$
\end{fact}

\begin{lemma} \label{lemma_span_is_acl_linear}
    We have $\acl_L = \spanA{\;}{\hat{L}}$ in $T$.
\begin{proof}
    Work in some $\mm \models T$.
    Since $\acl_L = \dcl_L$ holds in o-minimal theories, we immediately obtain $\acl_L(\varnothing) = \spanA{\varnothing}{\hat{L}}$.
    Suppose we have already shown $\acl_L(a_1, \dots, a_n) = \spanA{a_1, \dots, a_n}{\hat{L}}$ for all $a_1, \dots, a_n \in M$, and let $b \in M$ and $c \in \acl_L(a_1, \dots, a_n, b)$ be given.
    Write $\ua = (a_1, \dots, a_n)$.
    As discussed above Proposition 4.2 in \cite{LP93}, there must be some partial endomorphism $f$ and $b_1, b_2 \in \acl_L(\ua) = \spanA{\ua}{\hat{L}}$ such that $c = f(b + b_1) + b_2$.
    Thus $c \in \spanA{a_1, \dots, a_n, b}{\hat{L}}$.
\end{proof}
\end{lemma}

\begin{definition} \label{def_skew_d}
    Define
    $
    D := \set{\text{$\varnothing$-definable partial endomorphisms}} / \sim,
    $
    where $g_1 \sim g_2$ if $g_1$ and $g_2$ coincide on an open interval containing $0$.
    We order $D$ by
    $$
    g_1 < g_2 \quad: \Leftrightarrow \quad g_1(a) < g_2(a) \text{ for all sufficiently small $a > 0$.}
    $$
    As the addition and composition of two partial endomorphisms yield another partial endomorphism (after restricting domains appropriately), it is easy to verify that $(D, 0, \Id, +, \circ, <)$ is an ordered division ring.
\end{definition}

\noindent Notice that $D$ does not have to be commutative in general.
As an easy counterexample, take any ordered non-commutative division ring $D$ and verify that the theory of ordered $D$-vector spaces is o-minimal and linear.
Nevertheless, we say $D$-vector space instead of $D$-module.

\begin{fact}[Theorem 6.1 in \cite{LP93}] \label{fact_is_red_of_o_vs}
    The theory $T$ is a reduct of a complete theory $T'$ in the language $L' := \set{0, +, <, (\gamma \cdot )_{\gamma \in D}, (\mathfrak{c} )_{\mathfrak{c} \in \dcl_L(\varnothing)}}$ which expands the following theories:
    \begin{enumerate}[(i)]
        \item The theory of ordered $D$-vector spaces.
        \item The theory $T$, interpreting $\hat{g}$ in $T'$ as follows for the corresponding equivalence class $\gamma \in D$ of $g$ and $a, b \in \dcl_L(\varnothing) \cup \set{\pm \infty}$ with $\operatorname{dom}(g) = (a, b)$:
        $$
        x \mapsto \begin{cases}
            \gamma \cdot x & \text{if $x \in\ (a, b)$} \\
            0 & \text{otherwise.}
        \end{cases}
        $$
    \end{enumerate}
\end{fact}

\noindent We now state a few facts about groups definable in linear o-minimal theories.

\begin{fact} \label{fact_def_comp_iff_bound}
    Let $(G, e, \circ)$ be an $M$-definable group.
    Then $G$ is definably compact if and only if $G$ is bounded.
\begin{proof}
    Use Claim 3.4 in \cite{EE09} and Lemma 3.7 in \cite{ES07} (using the latter in the expansion from Fact \ref{fact_is_red_of_o_vs}).
\end{proof}
\end{fact}

\noindent The following fact by Edmundo and Eleftheriou states that every definable group in our setting is an extension of a definably compact group by $(M^m, 0, +)$ for some $m$:

\begin{fact}[Theorem 1.5 in \cite{EE09}] \label{fact_short_ex_seq}
    Let $(G, e, \circ)$ be an $M$-definable group.
    Then there are a definably compact $M$-definable group $H$ and a surjective $M$-definable homomorphism $\pi \colon G \to H$ such that $\Ker(\pi)$ is $M$-definably isomorphic to $(M^m, 0, +)$ for some $m$.
    In other words, there is a short exact sequence
    $$
    0 \;\to \; (M^m, 0, +) \; \to \; G \; \overset{\pi}{\to}\; H \; \to 0.
    $$
\end{fact}

\noindent The following fact by Eleftheriou and Starchenko allows us to study the definably compact group $H$ from the fact above.

\begin{fact}[Theorem 1.4 in \cite{ES07}] \label{fact_def_comp_torus}
    Let $\mm = (M, 0, +, <, (\gamma \cdot )_{\gamma\in D})$ be an ordered vector space over a division ring $D$.
    Let $(G, e, \circ)$ be an $n$-dimensional $M$-definable group that is definably compact and definably connected.
    Then $G$ is $M$-definably isomorphic to an $M$-definable quotient group $U/L$, for some convex ($M$-)$\bigvee$-definable subgroup $U \leq M^n$ and a lattice $L \leq U$ of rank $n$.
\end{fact}

\noindent Notice that $U$ being convex here means $q \cdot a + (1-q) \cdot b \in U$ for all $q \in \QQ \cap [0, 1]$ and $a, b \in U$.
If $n > 0$ in the above, then the group $G$ has torsion.
To see this, take linearly independent $v_1, \dots, v_n \in U$ that generate the lattice $L$.
Now, the equivalence class of $v_1/2$ is a non-trivial torsion element of $U/L$, as $U$ is convex and hence $v_1/2 \in U$.
With the earlier stated facts, this has the following consequence for definably compact groups:

\begin{corollary} \label{corollary_has_tors}
    Let $(G, e, \circ)$ be a definably compact $M$-definable group.
    Then $(G, e, \circ)$ is either trivial or contains torsion.
\begin{proof}
    By Fact \ref{fact_is_red_of_o_vs}, we can assume that $\mm = (M, 0, +, <, (\gamma \cdot )_{\gamma\in D}, (a)_{a \in A})$ is an ordered vector space over a division ring.
    Using Fact \ref{fact_def_comp_iff_bound}, we see that $G_0$, the definably $t$-connected component of $e$ in $G$ (which is definable by Fact \ref{fact_t_connected_finite_xd}), is still definably compact and $M$-definable.
    Now apply Fact \ref{fact_def_comp_torus} to see that $G_0$ is either trivial or contains torsion.
    As the quotient $G/G_0$ is finite by Fact \ref{fact_t_connected_finite_xd}, this translates to $G$ (also in the original structure $\mm = (M, 0, +, <, \dots)$).
\end{proof}
\end{corollary}

\noindent The author believes that the following result should be known, but could not find it anywhere in the literature.
For the sake of completeness, we include a proof, which assumes that there is an isomorphism definable with parameters.

\begin{fact} \label{fact_group_in_th_of_ord_d_vec_space}
    Every $\varnothing$-definable group $G$ in an ordered $D$-vector space $(M, 0, +, <, (\gamma \cdot)_{\gamma \in D})$ that is $M$-definably isomorphic to $(M^d, 0, +)$ is already $\varnothing$-definably isomorphic to $(M^d, 0, +)$.
\begin{proof}
    Take $\varphi(\ux, \uy; w)$ to be a formula such that for any $b \in M$ the formula $\varphi(\ux, \uy; b)$ defines an isomorphism $\iota_b \colon (G, e, \oplus) \to (M^d, 0, +)$ if and only if $b > 0$.
    Note that we can assume that $w$ is a singleton by letting $w > 0$ take the role of a non-zero element needed for a choice function.
    Then one can see, using cell decomposition, that there are finitely many matrices $\Gamma_1, \dots, \Gamma_q \in D^{n \times d}$ and vectors $\underline{\gamma}{}_1, \dots, \underline{\gamma}{}_q \in D^n$ such that for all $\ua \in M^d$ and $b > 0$ we have
    $$
    \iota_b^{-1}(\ua) = \Gamma_i \cdot \ua + \underline{\gamma}{}_i \cdot b
    $$
    for some $i \in \set{1, \dots, q}$.
    Define $\ii := \set{i \in \set{1, \dots, q} : \underline{\gamma}{}_i = 0}$.
    By the uniqueness of the $t$-topology, we see that every $\iota_b$ is a homeomorphism.
    Therefore, $\iota_b^{-1}(B_\epsilon(\uzero))$ is an open neighborhood of $e$ for any $\epsilon > 0$. Here $B_\epsilon(\uzero)$ denotes the open $\epsilon$-ball/box around $\uzero$.
    By choosing $b' \gg b, \epsilon$, we see that
    $$
    \iota_b^{-1}(B_\epsilon(0)) \cap \iota_{b'}^{-1}(B_\epsilon(0)) \subseteq \bigcup\nolimits_{i \in \ii} \Gamma_i \cdot B_\epsilon(0)
    $$
    is also an open neighborhood of $e$ in $G$.
    Now $\iota_{b'}(\iota_b^{-1}(B_\epsilon(0)) \cap \iota_{b'}^{-1}(B_\epsilon(0)))$ contains an open box $B_\delta(\uzero)$ such that every $\ua \in B_\delta(\uzero)$ is mapped to $\Gamma_i \cdot \ua$ for some $i \in \ii$ by $\iota_{b'}^{-1}$.
    Notice that such a $\delta$ must exist for any $b > 0$, and not just for $b'$, since otherwise we obtain a $\varnothing$-definable non-zero element.
    With linearity and the fact that there are no non-zero definable constants, we see that there is some $\gamma \in D \cup \set{\infty}$ such that, given $b > 0$, any $\delta < \gamma \cdot b$ works.
    Using o-minimality, we see that $\lim_{b\to \infty} \iota_b^{-1}(\ua) \in \set{\Gamma_i \cdot \ua : i \in \ii}$ exists for any $\ua \in M$.
    Using o-minimality once more, we can easily see that $\ua \mapsto \lim_{b\to \infty} \iota_b^{-1}(\ua)$ is an injective $\varnothing$-definable homomorphism from $(M^d, 0, +)$ into $(G, e, \oplus)$.
    The image of this map is clearly a $d$-dimensional subgroup of $(G, e, \oplus)$, so because $(G, e, \oplus)$ is divisible (recall that we already know it is isomorphic to $(M^d, 0, +)$ via an isomorphism definable with parameters), we conclude that it is actually an isomorphism.
\end{proof}
\end{fact}

\subsection{Infinite Divisible Torsion-Free Abelian Groups}

\noindent Let $\mm = (M, 0, +, <, \dots)$ be an o-minimal expansion of an ordered group, and let $(G, e, \oplus)$ be an infinite divisible torsion-free abelian group that is $\varnothing$-definable in $\mm$.

\begin{theorem} \label{lemma_group_lin_iso}
    Assume that $\mm$ is linear.
    Then $(G, e, \oplus)$ is $\varnothing$-definably isomorphic to $(M^d, 0, +)$ for $d = \dim(G)$.
\begin{proof}
    By Fact \ref{fact_short_ex_seq}, there is $m \in \NN$, a definably compact $M$-definable group $H$, and a surjective $M$-definable homomorphism $\pi \colon G \to H$ such that $\Ker(\pi)$ is $M$-definably isomorphic to $(M^m, 0, +)$.
    In other words, we have a short exact sequence
    $$
    0 \;\to \; (M^m, 0, +) \; \to \; G \; \overset{\pi}{\to}\; H \; \to 0.
    $$
    Assume, toward a contradiction, that $H$ has torsion.
    Let $h$ be a non-trivial $k$-torsion element and fix $g \in \pi^{-1}(h)$.
    As $k \cdot h = 0$, we see that $v := k \cdot g \in \Ker(\pi) \simeq (M^m, 0, +)$.
    Since $(M^m, 0, +)$ is divisible and $g \notin \Ker(\pi)$ (otherwise $h = 0$), there is $g' \in \Ker(\pi)$ with $g' \neq g$ and $k \cdot g' = v$.
    Now $g \ominus g'$ is a non-trivial $k$-torsion element in $G$, contradicting that $G$ is torsion-free.
    By Corollary \ref{corollary_has_tors}, we see that $H$ must be trivial, so $(M^m, 0, +) \simeq \Ker(\pi) = G$ via an $M$-definable isomorphism.
    This also implies that $m = \dim(G)$.

    We obtain an $L$-formula $\varphi(\ux, \uy; \uw)$ and a tuple of parameters $\ub \in M$ such that $\varphi(\ux, \uy; \ub)$ defines a group isomorphism $\iota \colon (G, e, \oplus) \to (M^d, 0, +)$.
    If there is at least one positive $\varnothing$-definable element in $T := \Th(\mm)$, then $T$ has definable choice.
    In this case, there is a $\varnothing$-definable tuple $\ub'$ for which $\varphi(\ux, \uy; \ub')$ also defines a group isomorphism.
    Otherwise, note that every partial endomorphism ((iv) of Definition \ref{def_linear_defs}) must be defined on all of $M$, so $T$ is actually the theory of an ordered $D$-vector space from Fact \ref{fact_is_red_of_o_vs} without any non-zero constants; hence, we can just apply Fact \ref{fact_group_in_th_of_ord_d_vec_space}.
\end{proof}
\end{theorem}

\noindent Using Theorem \ref{lemma_group_lin_iso}, we immediately see that if $\mm$ is linear, as in Definition \ref{def_linear_new}, then $(G, e, \oplus)$ must also be linear in $\mm$, as defined in (v) of Definition \ref{def_linear_defs}.
The other direction holds as well:

\begin{theorem} \label{theorem_not_lin}
    If $\mm$ is not linear, then $(G, e, \oplus)$ is not linear in $\mm$.
\end{theorem}
\begin{proof}
    Assume, toward a contradiction, that $(G, e, \oplus)$ is linear in $\mm$.
        We use the following setting (see Figure \ref{figure_local_group_chart}):
        \begin{enumerate}[(i)]
            \item There is an open (with respect to the $t$-topology) neighborhood $U \subseteq G$ of $e$ and an $M$-definable homeomorphism $\varphi \colon U \to B_{\epsilon_0}(\uzero) \subseteq M^d$ that maps $e$ to $\uzero$.
            We identify $U$ with $B_{\epsilon_0}(\uzero)$ and $\oplus$ with the continuous map
            $$
            \varphi \circ \oplus \circ (\varphi^{-1}, \varphi^{-1}) \colon \set{(\ua, \ub) \in B_{\epsilon_0}(\uzero) \times B_{\epsilon_0}(\uzero) : \varphi^{-1}(\ua) \oplus \varphi^{-1}(\ub) \in U} \to B_{\epsilon_0}(\uzero)
            $$
            whose domain is open.
            We define $\ominus$ similarly.
            \item There is an $M$-definable real closed field $(K, \boxzero, \boxone, \boxplus, \boxdot, <)$ where $\boxzero = 0$, $K \subseteq M$ is an interval, and $<$ is the same order as on $\mm$.
            \item We choose $\epsilon > 0$ such that $\epsilon < \min\set{|\boxone|, |\scalebox{0.8}{$\boxminus$}\hspace{1pt} \boxone|}$, and $\varphi^{-1}(\ua) \oplus \varphi^{-1}(\ub), \varphi^{-1}(\ua) \ominus \varphi^{-1}(\ub) \in U$ holds for all $\ua, \ub \in B_\epsilon(\uzero)$.
            We let $\boxdot \colon B_\epsilon(\uzero) \times B_\epsilon(\uzero) \to B_\epsilon(\uzero)$ denote the componentwise multiplication, which is a continuous map.
            This $\epsilon$ may be shrunk throughout the proof.
        \end{enumerate}

\begin{figure}[tbp]
    \centering
    \begin{tikzpicture}[
        x=1cm,
        y=1cm,
        group curve/.style={line width=0.55pt},
        chart neighborhood/.style={line width=1.25pt},
        chart arrow/.style={->, line width=0.55pt},
        marker/.style={circle, fill=black, inner sep=1.2pt},
        open endpoint/.style={circle, draw, fill=white, inner sep=1.2pt},
        interval/.style={line width=0.45pt},
        field interval/.style={line width=0.45pt},
        continuation/.style={densely dotted, line width=0.55pt},
        pic label/.style={font=\scriptsize}
    ]
        \node[pic label] at (6.7,3.95) {$G \subseteq M^2$};

        \coordinate (gLeftCont) at (-1.225,3.18);
        \coordinate (gLeftStart) at (-0.875,3.02);
        \coordinate (gRightEnd) at (7.525,3.35);
        \coordinate (gRightCont) at (7.875,3.50);
        \coordinate (chartLeftEnd) at (4.55,1.75);
        \coordinate (chartRightEnd) at (1.80,2.90);

        \draw[continuation] (gLeftCont) .. controls (-1.15,3.12) and (-1.05,3.153) .. (gLeftStart);
        \draw[group curve] (gLeftStart)
            .. controls (-0.35,2.62) and (0.20,1.50) .. (1.45,1.62)
            .. controls (2.32,1.704) and (3.92,1.40) .. (chartLeftEnd)
            .. controls (5.00,2.00) and (3.80,2.70) .. (3.10,2.20)
            .. controls (2.40,1.70) and (1.20,2.20) .. (chartRightEnd)
            .. controls (2.40,3.60) and (4.55,3.62) .. (5.85,3.02)
            .. controls (6.65,2.70) and (7.20,3.23) .. (gRightEnd);
        \draw[continuation] (gRightEnd) .. controls (7.65,3.41) and (7.76,3.46) .. (gRightCont);
        \draw[chart neighborhood] (chartLeftEnd)
            .. controls (5.00,2.00) and (3.80,2.70) .. (3.10,2.20)
            .. controls (2.40,1.70) and (1.20,2.20) .. (chartRightEnd);
        \draw[interval] ($(chartLeftEnd)+(-0.049,0.087)$) -- ($(chartLeftEnd)+(0.049,-0.087)$);
        \draw[interval] ($(chartRightEnd)+(-0.076,0.065)$) -- ($(chartRightEnd)+(0.076,-0.065)$);
        \node[marker] at (3.1,2.2) {};
        \node[pic label, anchor=north east] at (3.02,2.12) {$e$};
        \node[pic label, align=center] at (4.25,2.55) {$\varphi^{-1}(B_{\epsilon_0}(0))$};

        \draw[chart arrow] (3.1,1.88) -- node[right, pic label] {$\varphi$} (3.1,-0.72);

        \draw[continuation] (-1.225,-1.35) -- (-0.875,-1.35);
        \draw[interval] (-0.875,-1.35) -- (7.525,-1.35);
        \draw[continuation] (7.525,-1.35) -- (7.875,-1.35);
        \node[pic label, right] at (7.875,-1.35) {$M$};

        \draw[field interval] (0.35,-1.35) -- (7.18,-1.35);
        \node[open endpoint] at (0.35,-1.35) {};
        \node[open endpoint] at (7.18,-1.35) {};
        \draw[chart neighborhood] (1.075,-1.35) -- (5.125,-1.35);
        \node[pic label] at (3.765,-1.95) {$\underbrace{\hspace{6.7cm}}_{K}$};

        \draw[interval] (1.075,-1.25) -- (1.075,-1.45);
        \node[pic label, anchor=north] at (1.075,-1.48) {$-\epsilon_0$};
        \draw[interval] (1.6,-1.25) -- (1.6,-1.45);
        \node[pic label, anchor=south] at (1.6,-1.22) {$\scalebox{0.8}{$\boxminus$}\hspace{1pt}\boxone$};
        \draw[interval] (2.275,-1.25) -- (2.275,-1.45);
        \node[pic label, anchor=north] at (2.275,-1.48) {$-\epsilon$};
        \draw[interval] (3.1,-1.25) -- (3.1,-1.45);
        \node[pic label, anchor=south] at (3.1,-1.22) {$0$};
        \draw[interval] (3.925,-1.25) -- (3.925,-1.45);
        \node[pic label, anchor=north] at (3.925,-1.48) {$\epsilon$};
        \draw[interval] (5.125,-1.25) -- (5.125,-1.45);
        \node[pic label, anchor=north] at (5.125,-1.48) {$\epsilon_0$};
        \draw[interval] (6.325,-1.25) -- (6.325,-1.45);
        \node[pic label, anchor=south] at (6.325,-1.22) {$\boxone$};
    \end{tikzpicture}
    \caption[The local chart in the case $d=1$]{Our setting in the case $d=1$, with $G$ viewed as a subset of $M^2$. Since the field operations on $K$ differ from those coming from the ambient group structure on $M$, the points $\scalebox{0.8}{$\boxminus$}\hspace{1pt}\boxone$ and $\boxone$, as well as the endpoints of $K$, need not be equidistant from $0$. The constant $\epsilon$ is chosen so that the componentwise multiplication $\boxdot$ and the operations $\oplus$ and $\ominus$ induced via $\varphi$ from the group operations on $G$ are well behaved.}
    \label{figure_local_group_chart}
\end{figure}
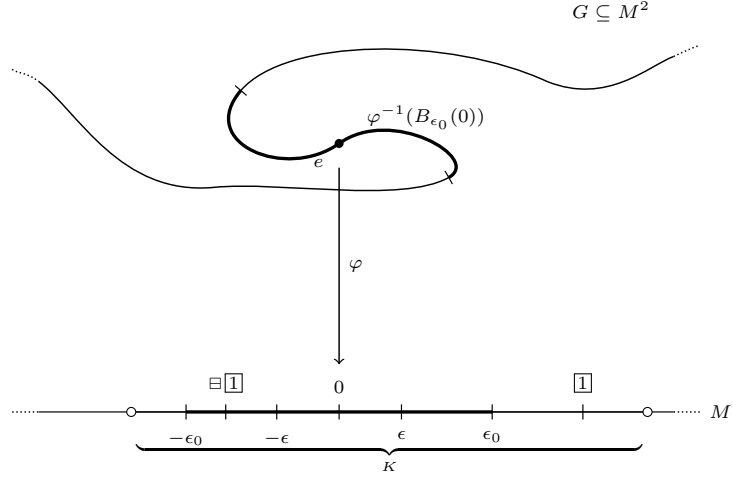
    \noindent The existence of $U$ and $\varphi$ is clear from the definable manifold structure on $(G, e, \oplus)$, and the existence of $(K, \boxzero, \boxone, \boxplus, \boxdot, <)$ follows from Fact \ref{fact_def_field}.
    With the ambient group structure $(M, 0, +)$, we can easily ensure $\varphi(e) = \uzero$ and $\boxzero = 0$.
    Notice that $\epsilon < |\boxone|$ and $\epsilon < |\scalebox{0.8}{$\boxminus$}\hspace{1pt} \boxone|$ ensures that the componentwise multiplication is indeed a map from $B_\epsilon(\uzero) \times B_\epsilon(\uzero)$ into $B_\epsilon(\uzero)$.
    Also note that this setting is preserved when shrinking the constant $\epsilon$.

By the linearity of our group $(G, e, \oplus)$ in $\mm$ (see (v) of Definition \ref{def_linear_defs}), we can partition $B_\epsilon(\uzero) \times B_\epsilon(\uzero)$ into finitely many $M$-definable sets $A_1, \dots, A_q$ such that, for all tuples $\ua, \ua', \ub, \ub', \ut, \ut' \in M^d$ with $(\ua, \ua'), (\ub, \ub'), (\ua \oplus \ut, \ua'\oplus \ut'), (\ub \oplus \ut, \ub'\oplus \ut') \in A_i$, we have
$$
    \boxdot(\ua \oplus \ut, \ua' \oplus \ut') \ominus \boxdot(\ua, \ua') = \boxdot(\ub \oplus \ut, \ub' \oplus \ut') \ominus \boxdot(\ub, \ub').
$$
To do this, pull $\boxdot$ back into $U \times U$, then perform a partition of the respective subset of $G^2$ as in (v) of Definition \ref{def_linear_defs}, and finally push the resulting partition forward into $B_\epsilon(\uzero) \times B_\epsilon(\uzero)$.
Since we identify $U$ with $B_{\epsilon_0}(\uzero)$, we will also call a map $\lambda$, defined on some open neighborhood of $\uzero$, a partial endomorphism of $(G, e, \oplus)$ if the map $\varphi^{-1} \circ \lambda \circ \varphi$ is a partial endomorphism of $(G, e, \oplus)$. Using a cell decomposition, we can assume that each $A_i$ is a cell, which in particular implies that each $A_i$ of full dimension is open in $M^{2d}$.
\begin{subclaim} \label{lemma_is_local_endo}
    Fix $i$ such that $A_i$ is open and $(\uzero, \uzero) \in \operatorname{Cl}(A_i)$.
    There are $\delta > 0$ and $M$-definable partial endomorphisms $\lambda, \mu$ of $(G, e, \oplus)$ such that on $A_i \cap B_{\delta}(\uzero) \times B_\delta(\uzero)$ we have
    $$
    \boxdot(\ut, \ut') = \lambda(\ut) \oplus \mu(\ut').
    $$
\begin{innerproof}
    Take any $(\ua, \ua') \in A_i$.
    Since $A_i$ is open and $\oplus$ is continuous, the sets
    \begin{align*}
        S_1 := ((\ut, \ut') \mapsto (\ua \oplus \ut, \ua'))^{-1}(A_i),& \quad\quad S_2 := ((\ut, \ut') \mapsto (\ua, \ua' \oplus \ut'))^{-1}(A_i), \\ &{}\hspace{-60pt}\text{ and }\quad S_3 := ((\ut, \ut') \mapsto (\ua \oplus \ut, \ua' \oplus \ut'))^{-1}(A_i)
    \end{align*}
    are all open and contain $(\uzero, \uzero)$ (recall that $\varphi(e) = \uzero$ is the neutral element of $\oplus$).
    Hence, the intersection $S_1 \cap S_2 \cap S_3$ contains an open box $B_{\delta}(\uzero) \times B_\delta(\uzero)$.

    Now take $(\ut, \ut') \in A_i \cap B_{\delta}(\uzero) \times B_\delta(\uzero)$.
    Since $(\uzero, \uzero) \in \operatorname{Cl}(A_i)$, we can find an arbitrarily small $(\us, \us') \in A_i$ such that $(\us \oplus \ut, \us' \oplus \ut') \in A_i$.
    Now, with the linearity of $\boxdot$ on $A_i$, we obtain
    \begin{align*}
        \boxdot(\us \oplus \ut, \us' \oplus \ut') \ominus \boxdot(\us, \us') &= \boxdot(\ua \oplus \ut, \ua' \oplus \ut') \ominus \boxdot(\ua, \ua') \\
        &= \boxdot(\ua \oplus \ut, \ua' \oplus \ut') \ominus \boxdot(\ua, \ua' \oplus \ut') \oplus \boxdot(\ua, \ua' \oplus \ut') \ominus \boxdot(\ua, \ua') \\
        &= \big(\underbrace{\boxdot(\ua \oplus \ut, \ua') \ominus \boxdot(\ua, \ua')}_{:= \lambda(\ut)}\big) \oplus \big(\underbrace{\boxdot(\ua, \ua' \oplus \ut') \ominus \boxdot(\ua, \ua')}_{:= \mu(\ut')}\big)
    \end{align*}
    (notice that since $(\ut, \ut') \in B_{\delta}(\uzero) \times B_\delta(\uzero) \cap A_i \subseteq S_1 \cap S_2 \cap S_3 \cap B_\epsilon(\uzero) \times B_\epsilon(\uzero)$, every expression with $\oplus$ and $\ominus$ above is well defined).
    By letting $(\us, \us') \to (\uzero, \uzero)$, we obtain $\boxdot(\ut, \ut') = \lambda(\ut) \oplus \mu(\ut')$.
    With Remark \ref{remark_locally_lin_to_partial}, we see that $\lambda$ and $\mu$ (or rather $\varphi^{-1} \circ \lambda \circ \varphi$ and $\varphi^{-1} \circ \mu \circ \varphi$) are $M$-definable partial endomorphisms of $(G, e, \oplus)$.
    Since $\oplus$ and $\boxdot$ are continuous, we observe that $\lambda$ and $\mu$ are also continuous.
\end{innerproof}
\end{subclaim}

\noindent Let $\ii$ be the set of all $i$ such that $A_i$ is open and $(\uzero, \uzero) \in \operatorname{Cl}(A_i)$, and for all $i \in \ii$, let $\delta_i > 0$ be as in Claim \ref{lemma_is_local_endo} for $A_i$.
By shrinking our constant $\epsilon$, we can assume that $B_\epsilon(\uzero) \times B_\epsilon(\uzero) \subseteq \bigcup_{i \in \ii} \operatorname{Cl}(A_i \cap B_{\delta_i}(\uzero) \times B_{\delta_i}(\uzero))$.
Now, we perform a cylindrical definable cell decomposition of $M^{2d}$ adapted to
$$
B_\epsilon(\uzero) \times B_\epsilon(\underbrace{0, \dots, 0}_{\text{$m$-many}}) \times \{(\hspace{-8pt}\underbrace{0, \dots, 0}_{\text{$(d-m)$-many}}\hspace{-8pt})\} \text{ for $m \in \set{0, \dots, d}$} \text{ and all the $A_i$'s with $i \in \ii$.}
$$
This gives us two open cells $W_0 \subseteq B_{\epsilon}(\uzero) \subseteq M^d$ and $W \subseteq B_{\epsilon}(\uzero) \times B_{\epsilon}(\uzero) \subseteq M^{2d}$ such that we have $W \subseteq A_i$ for some $i \in \ii$ and
$$
W = \set{(\ux, y_1, \dots, y_d) \in B_\epsilon(\uzero) \times B_\epsilon(\uzero) : \ux \in W_0, 0 < y_i < f_i(\ux, y_1, \dots, y_{i-1}) \text{ for $1 \leq i \leq d$}}
$$
for some $M$-definable functions $f_1, \dots, f_d > 0$ (see Figure \ref{figure_open_cell_over_wzero}).
\begin{figure}[tbp]
    \centering
    \begin{tikzpicture}[
        x=1cm,
        y=1cm,
        box boundary/.style={densely dotted, line width=0.55pt},
        cell line/.style={line width=0.55pt},
        graph cell/.style={line width=0.75pt},
        chosen boundary/.style={line width=1.1pt},
        blue region/.style={fill=blue!8},
        red region/.style={fill=red!8},
        green region/.style={fill=green!10},
        point cell/.style={circle, fill=black, inner sep=1.1pt},
        pic label/.style={font=\scriptsize}
    ]
        \coordinate (boxSW) at (-2.5,-2.5);
        \coordinate (boxNE) at (2.5,2.5);
        \coordinate (wleft) at (0,0);
        \coordinate (wright) at (1.45,0);

        \path[blue region] (boxSW) rectangle (boxNE);
        \path[red region] (-2.5,-2.5) -- (2.5,-2.5) -- (2.5,-0.78)
            .. controls (2.08,-0.42) and (1.84,-2.5) .. (1.45,-2.5)
            .. controls (0.82,-2.18) and (0.58,-0.46) .. (0,0)
            .. controls (-0.82,-0.18) and (-1.58,-1.12) .. (-2.5,-0.88) -- cycle;
        \path[green region] (-2.5,-0.88)
            .. controls (-1.58,-1.12) and (-0.82,-0.18) .. (0,0)
            -- (0,2.5) -- (-2.5,2.5) -- cycle;
        \path[green region] (0,0)
            .. controls (0.64,0.76) and (1.42,1.06) .. (2.5,1.42)
            -- (2.5,2.5) -- (0,2.5) -- cycle;

        \draw[box boundary] (boxSW) rectangle (boxNE);

        \draw[cell line] (-2.5,0) -- (2.5,0);
        \draw[cell line] (0,-2.5) -- (0,2.5);
        \draw[cell line] (1.45,-2.5) -- (1.45,2.5);

        \draw[graph cell] (-2.5,-0.88)
            .. controls (-1.58,-1.12) and (-0.82,-0.18) .. (0,0)
            .. controls (0.58,-0.46) and (0.82,-2.18) .. (1.45,-2.5)
            .. controls (1.84,-2.5) and (2.08,-0.42) .. (2.5,-0.78);
        \draw[graph cell] (0,0)
            .. controls (0.64,0.76) and (1.42,1.06) .. (2.5,1.42);

        \draw[chosen boundary] (wleft) -- (wright);
        \begin{scope}
            \clip (0,-2.5) rectangle (1.45,2.5);
            \draw[chosen boundary] (0,0)
                .. controls (0.64,0.76) and (1.42,1.06) .. (2.5,1.42);
        \end{scope}
        \node[pic label] at (0.88,0.34) {$W$};
        \node[pic label, anchor=south] at (0.68,0.73) {$\Gamma(f_1)$};
        \node[pic label, anchor=south east] at (-0.00,0.00) {$(0,0)$};
        \node[point cell] at (0,0) {};
        \node[pic label, anchor=north] at (0.8,0.05) {$W_0 \times \set{0}$};
    \end{tikzpicture}
    \caption[A cell decomposition used to choose $W_0$ and $W$]{The cylindrical cell decomposition used to choose $W_0$ and $W$ in the case $d = 1$. The square represents $B_\epsilon(0) \times B_\epsilon(0)$, and each color represents some $A_i$ with $i \in \ii$. The horizontal line is the slice $B_\epsilon(0) \times \set{0}$, which appears because the decomposition is also adapted to this slice. This implies the existence of a cell $W$ of the form $\set{(x, y) \in B_{\epsilon}(0) \times B_{\epsilon}(0) : x \in W_0, 0 < y < f(x)}$, where $W_0$ is an open interval and $f$ is a positive function on $W_0$.}
    \label{figure_open_cell_over_wzero}
\end{figure}
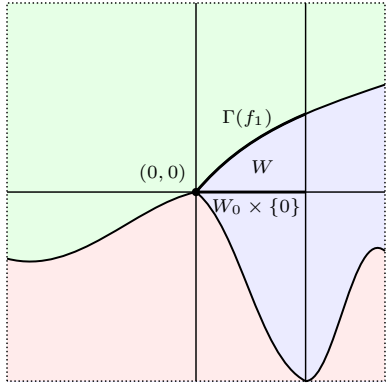
Notice that
\begin{enumerate}[(a)]
    \item $\pi(W) = W_0$ for the projection $\pi$ to the first $d$ coordinates.
    \item There are $M$-definable partial endomorphisms $\lambda, \mu$ such that $\boxdot(\ut, \ut') = \lambda(\ut) \oplus \mu(\ut')$ for all $(\ut, \ut') \in (W_0 \times \set{\uzero}) \cup W$.
\end{enumerate}
The first point is clear.
As $W \subseteq A_i$ for some $i \in \ii$, Claim \ref{lemma_is_local_endo} yields the endomorphisms in (b).
Also, notice that $W_0 \times \set{\uzero} \subseteq \operatorname{Cl}(W)$ by construction; hence, these endomorphisms also work for $W_0 \times \set{\uzero}$, as they and $\boxdot$ are continuous.

\begin{subclaim}
    There are elements $a, b, t \in K$ such that $a \neq b$, $t \neq \boxzero$, and $a \boxdot t = b \boxdot t$.
\begin{innerproof}
    By (a), for any $\ua \in \pi(W)$, we have $(\ua, \uzero) \in W_0 \times \set{\uzero}$.
    With (b), we obtain
    $$
    \uzero = \boxdot(\ua, \uzero) = \lambda(\ua) \oplus \mu(\uzero) = \lambda(\ua) \oplus \uzero = \lambda(\ua),
    $$
    as $\uzero = \varphi(e)$ is the $\oplus$-neutral element, the map $\boxdot$ is the componentwise multiplication in $K^d$, and $\uzero = (0, \dots, 0) = (\boxzero, \dots, \boxzero)$ is the tuple consisting only of the additively neutral element in our field $K$.
    Choose $(\ua, \ut), (\ub, \ut) \in W$ with $\ua \neq \ub$ and $\ut = (t_1, \dots, t_d) \in K^d$ such that each $t_k \neq \boxzero$ (this is possible as $W \subseteq M^{2d}$ is open).
    Now, as $\lambda(\ua) = \uzero = \lambda(\ub)$, we obtain
    \begin{align*}
        \boxdot(\ua, \ut) \;=\; \lambda(\ua) \oplus \mu(\ut) \;=\; \uzero \oplus \mu(\ut) \;=\; \lambda(\ub) \oplus \mu(\ut)
        \;=\; \boxdot(\ub, \ut).
    \end{align*}
    Write $\ua = (a_1, \dots, a_d)$ and $\ub = (b_1, \dots, b_d)$.
    As $\boxdot$ is the componentwise multiplication in our field $K$, we obtain $(a_1 \boxdot t_1, \dots, a_d \boxdot t_d) = (b_1 \boxdot t_1, \dots, b_d \boxdot t_d)$.
    Now, as $\ua \neq \ub$, there is some $k$ such that $a_k \neq b_k$, but $a_k \boxdot t_k = b_k \boxdot t_k$.
    By definition, we have $t_k \neq \boxzero$, so the claim is proven.
\end{innerproof}
\end{subclaim}

\noindent The claim above contradicts that $(K, \boxzero, \boxone, \boxplus, \boxdot, <)$ is a field, so our assumption that $(G, e, \oplus)$ is linear in $\mm$ must be wrong.
\end{proof}

\noindent The proof above may seem wrong at first when taking $T := \RCF{}$, $(G, e, \oplus) := (\rr_{>0},1, \cdot)$, and $(K, \boxzero, \boxone, \boxplus, \boxdot, <) := (\rr, 0, 1, +, \cdot, <)$ for any $\rr \models \RCF{}$.
However, in this case, the map $\varphi$ from (i) at the beginning of the proof must be chosen such that $\varphi(1) = 0$.
An obvious choice for $\varphi$ would be $x \mapsto x - 1$.
Now the map $\varphi \circ \oplus \circ (\varphi^{-1}, \varphi^{-1})$, which we just denoted by $\oplus$, is given by $(x, y) \mapsto (x + 1) \cdot (y + 1) - 1$, while we still have $\boxdot(x, y) = x \cdot y$.

Notice that the proof of Theorem \ref{theorem_not_lin} would also go through if we restrict (v) of Definition \ref{def_linear_defs} to $M$-definable functions $f \colon A \to G$ with $A \subseteq G^n$ open.
While Theorem \ref{theorem_not_lin} is interesting in its own right, we actually need the following variation, which shows that if $\mm$ (or equivalently $(G, e, \oplus)$ in $\mm$) is not linear, then there is a function on an open subset of $G^n$ that behaves nowhere like a linear function.

\begin{lemma} \label{lemma_not_loc_lin_fkt}
    If $\mm$ is not linear, then there is an $M$-definable function $f \colon U \to G$, where $U \subseteq G^2$ is non-empty and open, such that $f$ is not locally linear at any $\ug \in U$.
\begin{proof}
    We work in the same setting as in the proof of Theorem \ref{theorem_not_lin}: we fix an $M$-definable homeomorphism $\varphi \colon U \to B_{\epsilon_0}(\uzero) \subseteq M^d$, an $M$-definable field $(K, \boxzero, \boxone, \boxplus, \boxdot, <)$, and $\epsilon > 0$ as in (i), (ii), and (iii).
    We also define $\oplus$, $\ominus$, and $\boxdot$ on subsets of $B_{\epsilon_0}(\uzero) \times B_{\epsilon_0}(\uzero)$, as described there.
    Assume, toward a contradiction, that the set
    $$
    A := \set{(\ua, \ub) \in B_\epsilon(\uzero) \times B_\epsilon(\uzero) : \text{$\boxdot$ is locally linear at $(\ua, \ub)$}}
    $$
    is dense in $B_\epsilon(\uzero) \times B_\epsilon(\uzero)$.
    Partition $A = A_1 \dotcup \dots \dotcup A_q$ into its definably connected components.
    Fix $i \in \set{1, \dots, q}$.
    By definable connectedness and Remark \ref{remark_locally_lin_to_partial}, we can find a continuous $M$-definable function $f_i \colon B_\delta(\uzero) \times B_\delta(\uzero) \to B_{\epsilon_0}(\uzero)$ such that, for every $(\ua, \ub) \in A_i$, the map $\boxdot_{(\ua, \ub)}$ given by
    $$
    (\us, \ut) \mapsto \boxdot(\ua \oplus \us, \ub \oplus \ut) \ominus \boxdot(\ua, \ub)
    $$
    agrees with $f_i$ on some open neighborhood of $(\uzero, \uzero)$.
    Without loss, we may assume that $\delta$ does not depend on $i$.
    By shrinking the constant $\epsilon$ and then recomputing the definably connected components if necessary, we may also assume that $B_\epsilon(\uzero) \ominus B_\epsilon(\uzero) \subseteq B_\delta(\uzero)$, while keeping all the preceding conclusions.
    Now fix $(\ua, \ub) \in A_i$.
    The set
    $$
    A'_i :=\set{(\us,\ut) \in A_i \ominus (\ua, \ub) : \boxdot(\ua \oplus \us, \ub \oplus \ut) \ominus \boxdot(\ua, \ub) = f_i(\us,\ut)}
    $$
    is non-empty because it contains $(\uzero, \uzero)$, and it is closed in $A_i \ominus (\ua, \ub)$ because all functions involved are continuous.
    Fix $(\us,\ut) \in A'_i$.
    Since $\boxdot_{(\ua', \ub')}$ agrees locally with $f_i$ for all $(\ua', \ub') \in A_i$, we obtain
    \begin{align*}
        \boxdot_{(\ua, \ub)}(\us \oplus \us', \ut \oplus \ut') &= \boxdot_{(\ua, \ub)}(\us, \ut) \oplus \boxdot_{(\ua \oplus \us, \ub \oplus \ut)}(\us', \ut')  \\
        &= f_i(\us,\ut) \oplus f_i(\us', \ut') \\
        &= f_i(\us \oplus \us', \ut \oplus \ut')
    \end{align*}
    for all sufficiently small $(\us',\ut')$.
    Hence, $A'_i$ is also open in $A_i \ominus (\ua, \ub)$.
    Since $A_i \ominus (\ua, \ub)$ is definably connected, we obtain $A'_i = A_i \ominus (\ua, \ub)$.
    Since $(\ua, \ub)$ was arbitrary, we obtain
    $$
    \boxdot(\ua \oplus \us,  \ub \oplus \ut) \ominus \boxdot(\ua, \ub) = \boxdot(\ua' \oplus \us, \ub' \oplus \ut) \ominus \boxdot(\ua', \ub')
    $$
    for all $(\ua, \ub), (\ua', \ub'), (\us, \ut)$ with $(\ua, \ub), (\ua', \ub'), (\ua \oplus \us, \ub \oplus \ut), (\ua' \oplus \us, \ub' \oplus \ut) \in A_i$.
    As $A_1 \cup \dots \cup A_q$ is dense in $B_\epsilon(\uzero) \times B_\epsilon(\uzero)$, we obtain, by continuity of all functions involved, a partition of $B_\epsilon(\uzero) \times B_\epsilon(\uzero)$ as above Claim \ref{lemma_is_local_endo}.
    From that point on, one can follow the proof of Theorem \ref{theorem_not_lin} to obtain a contradiction to $K$ being a field.
    Hence our initial assumption that $A$ is dense in $B_\epsilon(\uzero) \times B_\epsilon(\uzero)$ must be wrong, so there is an open subset of $B_\epsilon(\uzero) \times B_\epsilon(\uzero)$ on which $\boxdot$ is not locally linear at any point.
    After pulling it back, this open subset is the $U$ from the statement of this lemma.
\end{proof}
\end{lemma}

\noindent If $G$ is one-dimensional, we can furthermore ensure that $U \subseteq G$.
To see this, we first show that we may essentially assume that $G$ is an interval with the order topology.
The following is essentially due to Razenj in \cite{Raz91}, but that paper does not assume that $\mm$ is an o-minimal expansion of an ordered group, so it only gives a definable order on any one-dimensional definable group.

\begin{lemma} \label{lemma_one_dim_non_linear_group}
    Suppose $\dim(G) = 1$.
    Then there is an $M$-definable homeomorphism between $G$ and an open interval.
\begin{proof}
    By Corollary 2.4 in \cite{PS05}, $G$ is definably connected.
    By Proposition 2 and Proposition 4 in \cite{Raz91}, this implies that $G \setminus \set{g}$ has exactly two definably connected components for any $g \in G$.
    By the definition of the $t$-topology and the one-dimensionality of $G$, there are finitely many maps $\phi_1, \dots, \phi_m$ such that each $\phi_k \colon S_k \to I_k$ is a homeomorphism between an open subset $S_k \subseteq G$ and an open interval $I_k$, and $G = \bigcup_{k=1}^m S_k$.
    If $m = 1$, there is nothing to show.
    Suppose that $m \geq 2$.
    After reindexing the maps $\phi_k$ and using the definable connectedness of $G$, we may assume without loss of generality that $S_1 \cap S_2 \neq \varnothing$.
    We may also assume that $S_1 \not\subseteq S_2$ and $S_2 \not\subseteq S_1$.
    By Lemma 1 of \cite{Raz91} and its proof, the set $S_1 \cap S_2$ has either one or two definably connected components, and the following hold:
    \begin{enumerate}[(i)]
        \item If $S_1 \cap S_2$ has two definably connected components, then $G = S_1 \cup S_2$.
        \item Let $J \subseteq \phi_1(S_1 \cap S_2)$ be the image under $\phi_1$ of a definably connected component of $S_1 \cap S_2$.
        Then $(\phi_2 \circ \phi_1^{-1})_{\restriction J}$ is a homeomorphism between intervals, which either preserves or reverses the order, and the following hold:
        \begin{enumerate}[(a)]
            \item If $(\phi_2 \circ \phi_1^{-1})_{\restriction J}$ is order preserving, then either $J$ is an initial segment of $I_1$ and $\phi_2 \circ \phi_1^{-1}(J)$ is an end segment of $I_2$, or $J$ is an end segment of $I_1$ and $\phi_2 \circ \phi_1^{-1}(J)$ is an initial segment of $I_2$.
            \item If $(\phi_2 \circ \phi_1^{-1})_{\restriction J}$ is order reversing, then either $J$ and $\phi_2 \circ \phi_1^{-1}(J)$ are initial segments of $I_1$ and $I_2$, or $J$ and $\phi_2 \circ \phi_1^{-1}(J)$ are end segments of $I_1$ and $I_2$.
        \end{enumerate}
    \end{enumerate}
    If $S_1 \cap S_2$ has two definably connected components, then one can use (i) and (ii) to verify that $G \setminus \set{g}$ has only one definably connected component, contradicting the above.
    Hence $S_1 \cap S_2$ has only one definably connected component, which we denote by $S$.
    After possibly reindexing $\phi_1$ and $\phi_2$ or replacing one of the maps $\phi_k$ with $-\phi_k$, we may assume that $\phi_1(S)$ is an end segment of $I_1$, that $\phi_2(S)$ is an initial segment of $I_2$, and that $(\phi_2 \circ \phi_1^{-1})_{\restriction \phi_1(S)}$ is order preserving.
    Take any $g \in S$.
    Using the ambient group structure on $M$, we may translate $I_1$ and $I_2$ so that $\phi_1(g) = \phi_2(g)$.
    It is now easy to check that we can replace the two maps $\phi_1$ and $\phi_2$ by a single map, namely the inverse of
    $$
    (I_1 \cap (-\infty, \phi_1(g)]) \cup (I_2 \cap (\phi_1(g), \infty)) \to S_1 \cup S_2;\quad x \mapsto \begin{cases}
        \phi_1^{-1}(x) & \text{if $x \leq \phi_1(g)$} \\
        \phi_2^{-1}(x) & \text{if $x > \phi_1(g)$,}
    \end{cases}
    $$
    which is a homeomorphism between $S_1 \cup S_2$ and an open interval.
    Thus, by induction on $m$, we can reduce the atlas to a single map.
\end{proof}
\end{lemma}

\noindent In contrast to Theorem \ref{lemma_group_lin_iso}, Lemma \ref{lemma_one_dim_non_linear_group} is wrong if one replaces ``$M$-definable homeomorphism'' with ``$\varnothing$-definable homeomorphism''.
The easiest counterexample is the structure $(\RR, 0, +, \oplus)$, where
$$
x \oplus y =
\begin{cases}
    0 & \text{if $x + y = 0$} \\
    y & \text{if $x = 0$} \\
    x & \text{if $y = 0$} \\
    \dfrac{x \cdot y}{x + y} & \text{otherwise.}
\end{cases}
$$
Then $(\RR, 0, \oplus)$ is a group whose underlying set is an interval, but the $t$-topology is induced by an order $\prec$ such that $a_1 \prec a_2 \prec 0 \prec a_3 \prec a_4$ for any $a_1, a_2, a_3, a_4$ with $a_2 < a_1 < 0 < a_4 < a_3$.

\begin{lemma} \label{lemma_binary_to_unary_nowhere_linear}
    Suppose $\dim(G) = 1$ and that $\mm \models T$ is not linear.
    Then there is an $M$-definable function $f \colon U \subseteq G \to G$ such that $U$ is a non-empty open definably connected neighborhood of $e$, and $f$ is nowhere locally linear.
\begin{proof}
    By Lemma \ref{lemma_not_loc_lin_fkt}, there is an $M$-definable nowhere locally linear function $f \colon U \to G$, where $U \subseteq G^2$ is non-empty and open. By Lemma \ref{lemma_one_dim_non_linear_group}, we may assume that the group $G$ is an interval and that its $t$-topology is the order topology. In particular, all operations on $G$ are continuous with respect to the order topology.
\begin{subclaim}
    Let $A_1$, $A_2$, and $A_3$ denote the sets of all $(a_1, a_2) \in U$ for which, respectively, the functions
    $$
    t \mapsto f(a_1 \oplus t, a_2), \quad t \mapsto f(a_1, a_2 \oplus t), \quad \text{and} \quad t \mapsto f(a_1 \oplus t, a_2 \oplus t)
    $$
    are not locally linear at $e$. Then there is some $i$ such that $A_i$ has non-empty interior in $U$.
\begin{innerproof}
    If not, then the set $U \setminus (A_1 \cup A_2 \cup A_3)$ has non-empty interior. This means that we can choose two non-empty open subsets $B_0 \subseteq B$ of this set, as well as a non-empty open interval $I_0\subseteq G$ containing $e$ such that:
    \begin{enumerate}[(i)]
        \item $B_0, B$ are boxes. By this we mean that there are $(a_1, a_2) \in U$ and an open interval $I \subseteq G$ containing $e$ such that $B = (a_1, a_2) \oplus I^2$, and similarly for $B_0$.
        \item For all $(b_1, b_2) \in B$, the functions
        $$
        t \mapsto f(b_1 \oplus t, b_2), \quad t \mapsto f(b_1, b_2 \oplus t), \quad \text{and} \quad t \mapsto f(b_1 \oplus t, b_2 \oplus t)
        $$
        are linear on the open intervals $I \oplus (a_1 \ominus b_1)$, $I \oplus (a_2 \ominus b_2)$, and $(I \oplus (a_1 \ominus b_1)) \cap (I \oplus (a_2 \ominus b_2))$, respectively. Now the linearity follows from the local linearity of these functions at $e$ and a definable connectedness argument.
        \item We have $B_0 \oplus I_0^2 \oplus I_0^2 \oplus I_0^2 \subseteq B$ and for every $(b_1, b_2) \in B_0$ we have $B_0 \subseteq (b_1, b_2) \oplus I_0^2$.
    \end{enumerate}
    For all $(b_1, b_2) \in B_0 \oplus I_0^2 \oplus I_0^2$ and $s, t \in I_0$, define $F_\hor(b_2; s) := f(b_1 \oplus s, b_2) \ominus f(b_1, b_2)$ and $F_\ver(b_1; t) := f(b_1, b_2 \oplus t) \ominus f(b_1, b_2)$. By the linearity in (ii), $F_\hor(b_2; s)$ does not depend on $b_1$, and $F_\ver(b_1; t)$ does not depend on $b_2$. We also easily see that
    $$
    F_\ver(b_1; t) \oplus F_\hor(b_2 \oplus t; s) = f(b_1 \oplus s, b_2 \oplus t) \ominus f(b_1, b_2) = F_\hor(b_2; s) \oplus F_\ver(b_1 \oplus s; t).
    $$
    This implies that $\Delta(s, t) := F_\ver(b_1 \oplus s; t) \ominus F_\ver(b_1; t) = F_\hor(b_2 \oplus t; s) \ominus F_\hor(b_2; s)$ does not depend on the choice of $(b_1, b_2) \in B_0 \oplus I_0^2 \oplus I_0^2$. We view $\Delta$ as a function from $I_0^2$ to $G$. Note that by Remark \ref{remark_locally_lin_to_partial} the function $s \mapsto F_\hor(b_2; s)$ is a partial endomorphism of $G$ when restricted to $I_0$. Hence, the function $s \mapsto \Delta(s, t)$ is a partial endomorphism of $G$ for any fixed $t \in I_0$. Using the diagonal linearity from (ii), we obtain the following for any $(b_1, b_2) \in B_0$ and $s, t \in I_0$:
    \begin{align*}
        e &= \big(f(b_1 \oplus t \oplus s, b_2 \oplus t \oplus s) \ominus f(b_1 \oplus t, b_2 \oplus t)\big) \ominus \big(f(b_1 \oplus s, b_2 \oplus s) \ominus f(b_1, b_2)\big) \\
          &= \big(F_\hor(b_2 \oplus t; s) \oplus F_\ver(b_1 \oplus t \oplus s; s)\big) \ominus \big(F_\hor(b_2; s) \oplus F_\ver(b_1 \oplus s; s)\big) \\
          &= \underbrace{F_\hor(b_2 \oplus t; s) \ominus F_\hor(b_2; s)}_{=\,\Delta(s,t)}
           \oplus \underbrace{F_\ver(b_1 \oplus t \oplus s; s) \ominus F_\ver(b_1 \oplus s; s)}_{=\,\Delta(t,s)}.
    \end{align*}
    Thus $\Delta(s,t)\oplus\Delta(t,s)=e$.
    Since $G$ is torsion-free, this immediately implies $\Delta(t, t) = e$ for any $t\in I_0$.
    Since $s \mapsto \Delta(s, t)$ is a partial endomorphism defined on $I_0$, we also obtain $\Delta(t/n, t) = e$ for any $n \geq 1$.
    By o-minimality, we see that there is a non-empty open subset of $I_0$ on which $s \mapsto \Delta(s, t)$ is equal to $e$ (in the case $t = e$ this follows directly from the definition of $\Delta$).
    Since this map is a partial endomorphism, and these have open definably connected domains (i.e., intervals) by definition, we actually obtain $\Delta(s, t) = e$ for all $s, t \in I_0$.
    By the definition of $\Delta$, we see that the partial endomorphisms $\mu_\hor,\mu_\ver \colon I_0 \to G$, defined by $\mu_\hor(s) := F_\hor(b_2; s)$ and $\mu_\ver(t) := F_\ver(b_1; t)$, do not depend on $(b_1, b_2) \in B_0$. Hence, for all $(b_1,b_2) \in B_0$ and all $(s,t) \in I_0^2$, we can write
    $$
    f(b_1\oplus s,b_2\oplus t)\ominus f(b_1,b_2)=\mu_\hor(s)\oplus\mu_\ver(t).
    $$
    By (iii), this holds in particular whenever $(b_1,b_2),(b_1\oplus s,b_2\oplus t) \in B_0$. Thus $f$ is linear on $B_0$, contradicting that $f$ is nowhere locally linear.
\end{innerproof}
\end{subclaim}

    \noindent By the claim, there is some $i \in \set{1,2,3}$ such that $A_i$ has non-empty interior in $U$.
    For the sake of simplicity, assume $i = 1$; the other cases are analogous.
    Choose $(a_1,a_2) \in U$ and a non-empty open interval $J \subseteq G$ containing $e$ such that
    $
    (a_1\oplus t,a_2) \in A_1
    $
    for all $t \in J$.
    Define
    $$
        g \colon J \to G;\quad t \mapsto f(a_1 \oplus t, a_2).
    $$
    If $g$ were locally linear at some $t_0 \in J$, then the function
    $
    s \mapsto f(a_1\oplus t_0\oplus s,a_2)
    $
    would be locally linear at $e$, contradicting $(a_1 \oplus t_0, a_2) \in A_1$.
    Thus $g$ is an $M$-definable nowhere locally linear function defined on a non-empty open definably connected neighborhood of $e$.
\end{proof}
\end{lemma}

\section{The Non-Linear Case} \label{sec_non_lin}

We now deal with the case where $\mm := (M, <, \dots)$ is a non-linear o-minimal expansion of an ordered group and set $T = \Th(\mm)$.
We call any theory of such a form \textbf{non-linear}.
Apart from that, we work again in our usual setting from Section \ref{sec_setting}, i.e., $T$ is also model-complete and there is a definable $K$-vector space $\VV$ in $T$.
In particular, we use the symbols $+$, $-$, and $0$ for the operations on the $K$-vector space $\VV$, and not for the operations on the ambient group structure in $\mm$.

\subsection{Neostability} \label{sec_neo}

\noindent Since any o-minimal theory is \NIP{}, and therefore also \NATP{}, the following result from a previous paper yields that $T\theta^C$ is \NATP{} if $T$ satisfies \Hfour{}:

\begin{fact}[Theorem 4.1 in \cite{Chi26b}]
    Suppose that $T$ satisfies \Hfour{}. If $T$ has \NATP{}, then $T\theta^C$ also has \NATP{}.
\end{fact}

\noindent Since both $\operatorname{NTP}_2$ and $\operatorname{NSOP}_1$ imply \NATP{}, this is only interesting if $T\theta^C$ has \TPtwo{} and $\operatorname{SOP}_1$. Because we assume $T$ to be o-minimal, the theory $T$, and therefore also $T\theta^C$, have $\operatorname{SOP}$, which implies having $\operatorname{SOP}_1$. Thus, the goal of this section is to show that $T\theta^C$ has \TPtwo{}.

\begin{definition}
    A formula $\varphi(\ux; \uy)$ has the \textbf{\boldmath tree property of the second kind ${(\operatorname{TP}_2)}$} in a theory $T$ if there are parameters $(\ua_{i, j} : i,j \in \omega)$ in a model of $T$ such that
    \begin{enumerate}[(i)]
        \item for all $\sigma \in \omega^\omega$, the type $\set{\varphi(\ux; \ua_{i, \sigma(i)}) : i \in \omega}$ is consistent.
        \item there is a $k \geq 2$ such that for all $n \in \omega$, the type $\set{\varphi(\ux; \ua_{n, i}) : i \in \omega}$ is $k$-inconsistent.
    \end{enumerate}
\end{definition}

\noindent Note that the following lemma only assumes the general assumptions of our construction, namely, those from Section \ref{sec_setting}.

\begin{lemma} \label{lemma_tptwo}
    Suppose that there is an $L$-formula $\varphi(x; \uz; \uw)$ algebraic in $x$ and a sequence of tuples $(\ud_i : i \in \omega)$ in a model of $T$ such that the type
    $$
    \set{\varphi(x_i; \uz; \ud_i) : i \in \omega}
    $$
    does not imply any finite disjunction of non-trivial linear dependencies in $(x_i : i \in \omega)$ over $\VV$.
    Then the model companion of $T^C_\theta$ has \TPtwo{} if it exists and either of the following holds:
    \begin{enumerate}[(i)]
        \item $R_C$ is not a field.
        \item $C$ is non-trivial, and there is another formula $\psi(\uy; \uw')$ and a sequence $(\ub_j : j \in \omega)$ such that:
        \begin{enumerate}[(a)]
            \item for $i \neq j$, the formulas $\psi(\uy; \ub_i)$ and $\psi(\uy; \ub_j)$ define disjoint subsets of $\VV^{|\uy|}$.
            \item for any $j \in \omega$, the formula $\psi(\uy; \ub_j)$ implies no finite disjunction of non-trivial linear dependencies in $\uy$ over $\VV$.
        \end{enumerate}
    \end{enumerate}
\begin{proof}
    We work in some existentially closed model $(\mm, \theta)$ of $T_\theta^C$ that contains the parameters from the statement.
    We start with case (i).
    If $R_C$ is not a field, then there is some $f \in \Kp{0<C}$ such that both $\Image(f)$ and $\Ker(f)$ are infinite sets (see Fact \ref{lemma_both_im_ker_inf}).
    We show that the formula
    $$
    \exists x \in \VV : \varphi(x; \uz; \uw) \wedge f[\theta](x) = y
    $$
    witnesses \TPtwo{} together with the array $(\ud_i v_j : i, j \in \omega)$, where the $v_j$'s are distinct elements in $\Image(f)$.
    For each $j \in \omega$, choose $u_j \in \VV$ with $f[\theta](u_j) = v_j$.
    Fix $\sigma \in \omega^\omega$.
    Notice that for each $q \in \omega$, the formula
    $$
    \exists\uz :\bigwedge\nolimits_{i =0}^q \varphi(x_i^0 + u_{\sigma(i)}; \uz; \ud_i)
    $$
    implies no finite disjunction of non-trivial linear dependencies in $\xvec{}_0 \dots \xvec{}_{q}$ over $\VV$.
    Now, by Theorem \ref{theorem_big_characterization}, the sentence
    $$
    \exists x_0\dots x_{q} \in \VV : \Big( \exists \uz: \bigwedge\nolimits_{i=0}^q \varphi(\theta^0(x_i) + u_{\sigma(i)}; \uz; \ud_i)\Big) \wedge \Big( \bigwedge\nolimits_{i= 0}^q f[\theta](x_i) = 0 \Big)
    $$
    holds in $(\mm, \theta)$ (notice that the second conjunction is a $C$-sequence-system over $(\VV, \theta)$).
    Since $f[\theta](u_j) = v_j$, we see that $f[\theta](x_i) = 0$ implies $f[\theta](\theta^0(x_i) + u_{\sigma(i)}) = v_{\sigma(i)}$.
    Hence, there is an $\ua \in M$ such that
    $$
    (\mm, \theta) \models \bigwedge\nolimits_{i = 0}^q \exists x \in \VV : \varphi(x; \ua; \ud_i) \wedge f[\theta](x) = v_{\sigma(i)}.
    $$
    Using compactness, we conclude that $\set{ \exists x \in \VV : \varphi(x; \uz; \ud_i) \wedge f[\theta](x) = v_{\sigma(i)} : i \in \omega}$ is consistent for every $\sigma \in \omega^\omega$.
    On the other hand, for each fixed $i \in \omega$, the type
    $$
    \set{\exists x \in \VV : \varphi(x; \uz; \ud_i) \wedge f[\theta](x) = v_j : j \in \omega}
    $$
    must be $(N+1)$-inconsistent for the $N$ that bounds $|\varphi(\VV; \ua; \ud)|$ for all $\ua, \ud \in \mm' \models T$ (recall that $\varphi(x; \uz; \uw)$ is algebraic in $x$ and that $T$ eliminates $\exists^\infty x \in \VV$ by Fact \ref{obser_no_elim_exist_inf_no_model_companion}).
    We conclude that the formula $\exists x \in \VV : \varphi(x; \uz; \uw) \wedge f[\theta](x) = y$ witnesses \TPtwo{}.

    In case (ii), we set $\ux = (x_k : 1 \leq k \leq m)$ for $m := |\uy|$ and similarly show that the $L_\theta$-formula
    $$
    \exists \ux \in \VV : \bigwedge\nolimits_{k=1}^m \varphi(x_k; \uz; \uw_k) \wedge \psi(\theta(x_1)\dots \theta(x_{m}); \uw')
    $$
    witnesses \TPtwo{} together with the parameters $((\ud_{m \cdot i + k} : 1 \leq k \leq m)\ub_j : i, j \in \omega )$.
    First notice that for any path $\sigma \in \omega^\omega$, the formula
    \begin{align}
        \exists \uz : \bigwedge\nolimits_{i =0}^q \Big( \bigwedge\nolimits_{k=1}^m \varphi(x_{i, k}^0; \uz; \ud_{m \cdot i + k}) \wedge \psi(x^1_{i, 1} \dots x^1_{i, m}; \ub_{\sigma(i)})\Big) \label{tag_formula_to_show_li}
    \end{align}
    implies no finite disjunction of non-trivial linear dependencies in $\uxvec{}_{0}\dots\uxvec{}_{q}$ over $\VV$.
    To see this, move the $\psi(\dots)$ outside of the $\exists \uz : \dots$ and observe that the resulting formula is a conjunction of two formulas in disjoint tuples of free variables, each of which implies no finite disjunction of non-trivial linear dependencies.
    Now, with Lemma \ref{lemma_rc_li_plus_li} and compactness, one can easily show the consistency of paths.

    The $k$-inconsistency of rows follows as in case (i), since the subformula $\bigwedge\nolimits_{k =1}^m\! \varphi(x_{i, k}^0; \ua; \ud_{m \cdot i + k})$ is algebraic in the tuple $(x_{i, k}^0 : 1 \leq k \leq m)$ and $\psi(\uy; \ub_j)$ defines disjoint sets for different $j$'s.
\end{proof}
\end{lemma}

\noindent Notice that the proof above only uses the characterization of existentially closed models, so it does not require that $T$ satisfies \Hfour{}.
We are not aware of any $T$ for which the $\varphi(x; \uz; \uw)$ and $(\ud_i : i \in \omega)$ exist as described, but the $\psi(\uy; \uw')$ and $(\ub_j : j \in \omega)$ from (ii) do not.

\begin{example} \label{example_tpwo}
    The following theories have \TPtwo{} for any  kernel configuration $C \in \Cc$ that is non-trivial.
    \begin{enumerate}[(i)]
        \item The theory $\RCF\theta^C$, where $(\VV^\rr, 0, +, (q\cdot)_{q\in \QQ}) := (\rr_{>0}, 1, \cdot, (x\mapsto x^q)_{q\in \QQ})$ is the vector space induced by multiplication.
        To see this, one can apply (ii) of Lemma \ref{lemma_tptwo}, together with $\varphi(x; z; w) := x = z + w$ (where $+$ is the actual addition of a real closed field), \hbox{$(d_i : i \!\in\! \omega) := (i : i \!\in\! \omega)$}, $\psi(y;w'_1w'_2) := w'_1 < y < w'_2$, and $(\ub_j : j \!\in\! \omega) := ((j, j+1) : j \!\in\! \omega)$.
        \item The theory $\ACF{}_p\theta^C$, where $p > 0$ and $\VV$ is the $\FF_p$-vector space given by addition.
        Here we take $\varphi(x; z; w) := x = z \cdot w$, $(d_i : i \in \omega)$ to be any $\FF_p$-linearly independent subset of some algebraically closed field of characteristic $p$, $\psi(y_1y_2; w') := y_2 = y_1^2 + w'$, and $(b_j : j \in \omega)$ to be any sequence of distinct elements.
    \end{enumerate}
\end{example}

\noindent Since $\RCF$ is distal, we see that distality and $\operatorname{NTP}_2$ are, in general, not preserved by our construction.
Similarly, as $\ACF{}_p$ is stable, we see that stability and simplicity are generally not preserved by our construction.

Clearly, the theory $\RCF{}$ is a non-linear o-minimal expansion of the theory of ordered groups.
We now prove that the example above generalizes to all other non-linear o-minimal expansions of ordered groups.

\begin{lemma} \label{lemma_sequence_mix_ff}
    Assume that $\mm$ is sufficiently saturated.
    Then there are an $M$-definable function $f \colon U \subseteq \VV^n \to \VV$ and a sequence $(\uu_i)_{i \in \omega}$ in $U$ such that the following partial type is consistent in some $\mm' \succ \mm$:
    $$
    \Sigma(\ux) := \set{\text{``the sequence $(f(\uu_i + \ux) : i \in \omega)$  is linearly independent over $\VV$''}}.
    $$
\begin{proof}
    By Lemma \ref{lemma_not_loc_lin_fkt} with $G = \VV$ and Remark \ref{remark_locally_lin_to_partial}, there is an $M$-definable function $f \colon U \to \VV$, where $U \subseteq \VV^n$ is $M$-definable, non-empty, and open, such that, for every $\uu \in U$, no restriction of the map $f_{\uu} \colon U - \uu \to \VV; \ut \mapsto f(\uu + \ut) - f(\uu)$ to an open neighborhood of $\uzero$ is a partial homomorphism of $(\VV, 0, +)$.

    Suppose, toward a contradiction, that $\Sigma(\ux)$ is inconsistent for every sequence $(\uu_i : i \in \omega)$ in $U$.
    Let $A := \set{(\uu_1, \dots, \uu_q, \uv) \in \VV^{(q+1)\cdot n} : \bigwedge\nolimits_{i=1}^q \uu_i + \uv \in U}$, where $q$ will be chosen below.
    Since $U$ is open and $+$ is continuous, the set $A$ is also open, where all topological notions are with respect to the $t$-topology.
    By compactness and the definition of linear independence over $\VV$, we obtain
    \begin{align}
        \mm \models \forall (\uy_1, \dots, \uy_q, \ux) \in A: \bigvee\nolimits_{k=1}^m \sum\nolimits_{i=1}^q \lambda_{k, i} \cdot f(\uy_i + \ux) + g_k(\uy_1, \dots, \uy_q) = 0, \label{tag_eqeqeqeq}
    \end{align}
    where $q, m \geq 1$, each $g_k$ is an $M$-definable function, and each $(\lambda_{k, 1}, \dots, \lambda_{k, q})$ lies in $K^q \setminus \set{\uzero}$.
    To be more precise, if no statement of the form (\ref{tag_eqeqeqeq}) holds, then there is a sequence of tuples $(\uu_i : i \in \omega)$ such that the partial type
    $$
    \set{\text{``the sequence $(f(\uu_i + \ux) : i \in \omega)$ is linearly independent over $\VV \cap \dcl(\uu_i : i \in \omega)$''}}
    $$
    is realized in some elementary extension. Using the same arguments as at the end of the proof of Lemma 3.7 in \cite{Chi25b}, we see that this actually implies that $\Sigma(\ux)$ is consistent in some elementary extension (also note that we have $\acl = \dcl$ in o-minimal theories).
    Now, at least one of the equations in (\ref{tag_eqeqeqeq}) defines a subset of $A$ with non-empty interior.
    So, after shrinking $A$, we may assume that $A$ is still open and
    $$
    \mm \models \forall (\uy_1, \dots, \uy_q, \ux) \in A: \sum\nolimits_{i=1}^q \lambda_{i} \cdot f(\uy_i + \ux) + g(\uy_1, \dots, \uy_q) = 0
    $$
    holds, where $\lambda_i := \lambda_{1, i}$ and $g := g_{1}$.
    After reindexing the $\uy_i$ and scaling, we may also assume that $\lambda_1 = 1$.

    Fix some $(\uu_1, \dots, \uu_q, \uv) \in A$.
    Since $A$ is open, we can find an open subset $B \subseteq \VV^n$ containing $\uzero$ such that $(\uu_1, \dots, \uu_q, \uv) + B^{q+1} \subseteq A$.
    For any $\ut \in B$, we obtain
    \begin{align*}
        f_{\uu_1+\uv}(\ut) &:= f(\uu_1 + (\uv + \ut)) - f(\uu_1 + \uv)\\
        &= -\sum\nolimits_{i=2}^q \lambda_{i} \cdot f(\uu_i + (\uv + \ut)) - g(\uu_1, \dots, \uu_q) + \sum\nolimits_{i=2}^q \lambda_{i} \cdot f(\uu_i + \uv) + g(\uu_1, \dots, \uu_q) \\
        &= \sum\nolimits_{i=2}^q \lambda_{i} \cdot (f(\uu_i + \uv) - f(\uu_i + (\uv + \ut))).
    \end{align*}
    Similarly, $f_{\uu_1 + \uv}(\ut) = f((\uu_1 + \ut) + \uv) - f(\uu_1 + \uv) = g(\uu_1, \dots, \uu_q) - g(\uu_1 + \ut, \uu_2, \dots, \uu_q)$.
    Thus, for any $\us, \ut \in \VV^n$ with $\us, \ut, \us + \ut \in B$, we have
    \begin{align*}
        f_{\uu_1+\uv}(\us + \ut) &:= f((\uu_1 + \ut) + (\uv + \us)) - f(\uu_1 + \uv)\\
        &= -\sum\nolimits_{i=2}^q \lambda_{i} \cdot f(\uu_i + (\uv + \us)) - g(\uu_1 + \ut, \uu_2, \dots, \uu_q) \\
        & \hspace{160pt}+ \sum\nolimits_{i=2}^q \lambda_{i} \cdot f(\uu_i + \uv) + g(\uu_1, \dots, \uu_q) \\
        &= \underbrace{\sum\nolimits_{i=2}^q \lambda_{i} \cdot (f(\uu_i + \uv) - f(\uu_i + (\uv + \us)))}_{= f_{\uu_1+\uv}(\us)} + \underbrace{g(\uu_1, \dots, \uu_q) - g(\uu_1 + \ut, \uu_2, \dots, \uu_q)}_{= f_{\uu_1+\uv}(\ut)}.
    \end{align*}
    Hence $f_{\uu_1+\uv}$ restricts to a partial homomorphism on some open neighborhood of $\uzero$, contradicting the choice of $f$.
    We conclude that $\Sigma(\ux)$ is consistent for some sequence $(\uu_i : i \in \omega)$ in $U$.
\end{proof}
\end{lemma}

\begin{corollary} \label{corollary_tP_two}
    Assume that $T$ is non-linear. The model companion of $T^C_\theta$ has \TPtwo{} if it exists and $C$ is non-trivial.
\begin{proof}
    Let $f$ and the sequence $(\uu_i : i \in \omega)$ be as in the conclusion of Lemma \ref{lemma_sequence_mix_ff}.
    We apply Lemma \ref{lemma_tptwo} with the formula $\varphi(x; \uz; \uw) := x = f(\uz + \uw)$ and the sequence $(\uu_i : i \in \omega)$.
    Since we are working in an o-minimal theory, we can apply case (ii) of Lemma \ref{lemma_tptwo}, with the formulas $\psi(y; \ub_j)$ defining disjoint $\epsilon$-balls in $\VV$.
\end{proof}
\end{corollary}

\subsection{Failure of Exchange for the Algebraic Closure}

\noindent Assuming that $T$ satisfies \Hfour{} and that $C$ is non-trivial, we have already seen that $\cl_\theta = \acl_{L_\theta}$ can have the exchange property in $T\theta^C$ only if $\acl_L$ has the exchange property in $T$, $\dim(\VV) = 1$, and $R_C$ is a field (see Fact \ref{theorem_nes_cond_fixed}).
We have also seen that these conditions are necessary but not sufficient.
The counterexample for $\RCF\theta^C$ was given in Example 4.21 in \cite{Chi25b}.
In this section, we generalize this to all non-linear o-minimal expansions of an ordered group.

\begin{lemma} \label{lemma_fkn_stupid_function}
    Assume that $\dim(\VV) = 1$, $\MM \models T$ is a monster model, and let $A \subset M$ be small.
    Let $f \colon U \to \VV$ be a nowhere locally linear $A$-definable function, where $U \subseteq \VV$ is non-empty and open.
    Then there is an $r \in \VV \setminus \acl_L(A)$ such that the formula $f(x_1) + f(x_2) = r$ does not imply any finite disjunction of non-trivial linear dependencies in $(x_1, x_2)$ over $\VV$.
\begin{proof}
    By Lemma \ref{lemma_one_dim_non_linear_group}, after adding to $A$ the parameters needed to define the homeomorphisms from that lemma, we may assume that $\VV$ is an interval with the order topology.
    By o-minimality, after shrinking $U$, we may assume that $U$ is an open interval and that $f$ is continuous and strictly monotone (if $f$ were constant on some interval, then it would be locally linear there).
    In particular, $f \colon U \to f(U)$ is a bijection with continuous inverse.
    
    Now take any $r \in \Image(f) + \Image(f) \setminus \acl_L(A)$ and suppose that the formula $f(x_1) + f(x_2) = r$ implies a finite disjunction of non-trivial linear dependencies in $(x_1, x_2)$ over $\VV$.
    This means that we have
    \begin{align*}
        \MM \models \forall x_1, x_2 \in U : f(x_1) + f(x_2) = r \rightarrow \bigvee\nolimits_{k=1}^m \lambda_{k, 1} \cdot x_1 + \lambda_{k, 2} \cdot x_2 = v_k
    \end{align*}
    where each $(\lambda_{k, 1}, \lambda_{k, 2}) \in K^2 \setminus \set{\uzero}$ and each $v_k \in \VV$.
    We may also assume that this disjunction contains no unnecessary disjuncts.
\begin{subclaim}
    There are two non-empty open intervals $I, J \subseteq \VV$ and two $M$-definable functions $g \colon I \to \VV$ and $h \colon J \to \VV$ such that
    $$
    \MM \models \forall (x, y) \in I \times J : f(x + y) = g(x) + h(y).
    $$
\begin{innerproof}
    Note that $f(x_1) + f(x_2) = r$ implies $x_2 = f^{-1}(r - f(x_1))$, and that each $v_k$ is $Ar$-definable.
    Let $I_r$ be the leftmost non-empty definably connected component of the interior of the set
    $$
    \set{u \in U : f^{-1}(r - f(u)) \text{ is defined and } \lambda_{k, 1} \cdot u + \lambda_{k, 2} \cdot f^{-1}(r - f(u)) = v_k},
    $$
    for some fixed $k$ for which this set has interior.
    Because $f^{-1}(r - f(x_1))$ is non-constant in $x_1$, we can assume that both $\lambda_{k, 1}$ and $\lambda_{k, 2}$ are non-zero.
    In what follows, we set $\lambda := \lambda_{k, 1}$ and assume $\lambda_{k, 2} = 1$.
    It is clear that $I_r$ is an $Ar$-definable interval.
    Because $v_k$ is $Ar$-definable, there is an $A$-definable function $\tilv \colon S \to \VV$, with $r \in S$, such that $v_k = \tilv(r)$.
    We may shrink $S$ so that the following hold:
    \begin{enumerate}[(i)]
        \item $S$ is an open interval that still contains $r$.
        \item For every $s \in S$, the set $I_s$ is non-empty, where we define $I_s$ as the leftmost definably connected component of the interior of the set
        $$
        \set{u \in U : f^{-1}(s - f(u)) \text{ is defined and } \lambda \cdot u + f^{-1}(s - f(u)) = \tilv(s)}
        $$
        and as $\varnothing$ if no such definably connected component exists.
        \item The functions $s \mapsto \text{``left endpoint of $I_s$''}$ and $s \mapsto \text{``right endpoint of $I_s$''}$ are continuous on $S$.
    \end{enumerate}
    If there were no such interval satisfying (i) and (ii), then $r$ would lie in the boundary of the $A$-definable set $\set{s \in S : \text{``$I_s$ is non-empty''}}$, contradicting $r \not\in \acl_L(A)$.
    Since the points of discontinuity of the functions in (iii) are also $A$-definable, we can similarly assume that (iii) holds.
    
    Using (iii), we may further shrink $S$ so that there is an open interval $I_*$ such that $I_* \subseteq I_s$ for all $s \in S$.
    Note that both $I_*$ and $S$ may no longer be $A$-definable.
    Pick any $u_0 \in I_*$.
    Set $J := \tilv(S)$.
    Since, for every $s \in S$, the equality $\tilv(s) = \lambda \cdot u_0 + f^{-1}(s - f(u_0))$ holds and $f^{-1}$ is strictly monotone and continuous, the function $\tilv \colon S \to J$ is a bijection with continuous inverse $h \colon J \to S$.
    It is clear that $J$ is also an open interval.
    Set $I := -\lambda \cdot I_*$ and define $g \colon I \to \VV$ by $g(u) := -f(-\lambda^{-1} \cdot u)$.
    For any $u \in I$ and $v \in J$, we have $-\lambda^{-1} \cdot u \in I_*$, and therefore
    \begin{align*}
        &&v &= \lambda \cdot (-\lambda^{-1} \cdot u) + f^{-1}(h(v) - f(-\lambda^{-1} \cdot u)) \\
        \Leftrightarrow&&v &= -u + f^{-1}(g(u) + h(v)) \\
        \Leftrightarrow&&u + v &= f^{-1}(g(u) + h(v)) \\
        \Leftrightarrow&&f(u + v) &= g(u) + h(v),
    \end{align*}
    which completes the proof of this claim.
\end{innerproof}
\end{subclaim}
    \noindent Now take $u \in I$, $v \in J$, and some $\epsilon > 0$ such that $B_\epsilon(u) \subseteq I$ and $B_\epsilon(v) \subseteq J$.
    Here these open balls are taken with respect to the operations of $\VV$, not the ambient group.
    Given any $s, t \in B_\epsilon(0)$, we obtain
    \begin{align*}
        f_{u+v}(s)\ :=&\ f(u+v + s) - f(u+v)\\
        =&\ f((u+s) + v) - f(u + v) \\
        =&\ g(u+s) + h(v) - g(u) - h(v)\\
        =&\ g(u+s) - g(u),
    \end{align*}
    and similarly $f_{u+v}(t) = h(v+t)-h(v)$.
    On the other hand, we also obtain
    \begin{align*}
        f_{u+v}(s + t) &= f((u + s) + (v+t)) - f(u+v) \\
        &= g(u+s) + h(v+t) - g(u) - h(v) \\
        &= \underbrace{g(u+s) - g(u)}_{= f_{u+v}(s)} + \underbrace{h(v+t) - h(v)}_{= f_{u+v}(t)}.
    \end{align*}
    This shows that the function $t \mapsto f(u+v + t) - f(u+v)$ restricted to $B_\epsilon(0)$ is a partial endomorphism of $\VV$.
    By Remark \ref{remark_locally_lin_to_partial}, this implies that $f$ is locally linear at $u+v$, contradicting our assumption on $f$.
\end{proof}
\end{lemma}

\begin{theorem} \label{theorem_no_exchange}
    Suppose that $T$ is non-linear and satisfies \Hfour{}.
    Then the algebraic closure $\acl_{L_\theta} = \cl_\theta$ has the exchange property in $T\theta^C$ if and only if $C$ is trivial.
\begin{proof}
    Recall that, in the trivial case, $T$ and $T\theta^C$ are interdefinable.
    By Fact \ref{theorem_nes_cond_fixed}, we only need to show that $\cl_\theta = \acl_{L_\theta}$ does not have the exchange property when $\VV$ is one-dimensional and $C$ is non-trivial.
    Work in any sufficiently large monster $(\MM, \theta) \models T\theta^C$.
    By Lemma \ref{lemma_binary_to_unary_nowhere_linear} and Lemma \ref{lemma_fkn_stupid_function}, there is some parameter set $A$, an $L(A)$-definable function $f \colon U \to \VV$, and some $r \in \VV \setminus \acl_L(A)$ such that $f(x_1) + f(x_2) = r$ does not imply any finite disjunction of non-trivial linear dependencies in $(x_1, x_2)$ over $\VV$.
    By \Hfour{}, there is an $L(A)$-formula $\sigma(y)$ that holds if and only if $f(x_1) + f(x_2) = y$ does not imply any finite disjunction of non-trivial linear dependencies in $(x_1, x_2)$ over $\VV$.
    So, since $r \not\in \acl_L(A)$, we may assume without loss of generality that $r \not\in \cl_\theta(A)$.
    We can now use Lemma \ref{lemma_rc_li_plus_li} (and the fact that excluding a finite set introduces no linear dependencies) to find infinitely many realizations of the $L_\theta(Ar)$-formula
    $$
    f(x) + f(\theta(x)) = r.
    $$
    In particular, we can find some $v \in \VV \setminus \cl_\theta(Ar)$ that satisfies $f(v) + f(\theta(v)) = r$.
    This implies $r \in \cl_\theta(Av) \setminus \cl_\theta(A)$ and $v \not\in \cl_\theta(Ar)$.
    In general, $v$ and $r$ are tuples, but since $\dim(\VV) = 1$, we can pick entries $v_0 \in M$ and $r_0 \in M$ of these tuples such that $r_0 \in \cl_\theta(Av_0) \setminus \cl_\theta(A)$ and $v_0 \not\in \cl_\theta(Ar_0)$ (or identify $\VV$ with some interval using Lemma \ref{lemma_one_dim_non_linear_group}).
    Hence $\cl_\theta$ does not have the exchange property in $T\theta^C$.
\end{proof}
\end{theorem}

\subsection{A Criterion for the Non-Existence of the Model Companion}
\label{sec_reve_hfour}
\noindent Recall that, for general theories $T$, \Hfour{} implies the existence of the model companion of $T^C_\theta$ for all $C \in \Cc$. It is still open whether the converse is true.
We now show that the converse is true in a special case.

\begin{fact}[Theorem 5.5, Corollary 5.3, and Remark 5.7 in \cite{Chi26}] \label{theorem_kernel_generic}
    Let $V$ be a new predicate symbol intended for a subspace of $\VV$ and set
    $$
    T_V := T \cup \set{\text{``$V$ defines a vector subspace of $\VV$"}}.
    $$
    If $\rho$ is a polynomial and $(\mm, \theta) \models T_\theta^C$ is an existentially closed model, then one of the following holds:
    \begin{enumerate}[(i)]
        \item $(\mm, \Ker(\rho)) = (\mm, \set{0})$;
        \item $(\mm, \Ker(\rho)) = (\mm, \VV)$;
        \item $(\mm, \Ker(\rho))$ is an existentially closed model of $T_V$.
    \end{enumerate}
    Moreover, if $T$ satisfies \Hfour{}, then $T_V$ has a model companion, denoted $TV$, and every completion of $TV$ is of the form $\Th(\mm, \Ker(\rho))$ for some existentially closed model $(\mm, \theta) \models T^C_\theta$ and $\rho \in K[X]$.
\end{fact}

\begin{theorem} \label{theorem_prove_conj}
    Suppose that $C \in \Cc$ is non-trivial, $T$ is a complete and model-complete o-minimal theory, $K = \QQ$, $\VV$ is an open interval, and all operations of $(\VV, 0, +, (q \cdot)_{q\in \QQ})$ are continuous with respect to the order topology.
    Then $T^C_\theta$ has a model companion if and only if $T$ satisfies \Hfour{}.
\begin{proof}
    One implication follows from the general result recalled above.
    For the other direction, assume toward a contradiction that $T$ does not satisfy \Hfour{} and that the model companion of $T_\theta^C$ exists.
    Denote this model companion by $T\theta^C_*$, and let $(\mm, \theta) \models T\theta^C_*$ be an $\aleph_1$-saturated model.
    In this setting, Block Gorman proved in \cite{Blo23} that $T$ satisfies \Hfour{} if and only if $T$ does not have the uniform endomorphism property; see Definition 2.2 and Lemma 2.9 in \cite{Blo23}.
    Thus $T$ has the uniform endomorphism property, which by the proof of Theorem 2.4 in \cite{Blo23} has the following consequences:
\begin{enumerate}[(i)]
    \item Given $\mm \models T$, there is an open interval $I \subseteq \VV$ with $0 \in I$, a positive element $\gamma \in I \setminus \set{0}$, and an $M$-definable family $\set{h_y \colon I \to \VV : y \in (0, \gamma)}$ such that, for every $b \in (0, \gamma)$, the map $x \mapsto h_b(x)$ is a partial endomorphism of $(\VV,0,+)$ on $I$ with $h_b(\gamma) = b$.
    Moreover, for every $q \in \QQ \cap (0,1)$, we have $h_{q\cdot \gamma}(t) = q\cdot t$ for all $t \in I$.
    \item The theory $T_{\mathcal{G}} := T \cup \set{\text{``$\mathcal{G}$ is a dense and codense divisible subgroup of $\VV$"}}$ in the language $L$ expanded by a unary predicate for the group $\mathcal{G}$ does not have a model companion.
    More precisely, every existentially closed model $(\mm, \mathcal{G}) \models T_{\mathcal{G}}$ has a definable set of cardinality $\aleph_0$.
\end{enumerate}
    First, assume that $R_C$ is a field.
    This means that either $C = C_0$, that is, $C$ is transcendental and $C(f) = 0$ for all $f \in \Kp{}$, or $C$ is algebraic and $\mipo(C)$ is irreducible of degree at least $2$.
    Consider
    $$
    X := \set{ b \in (0, \gamma) : \forall x \in I \setminus \set{0} : \theta(x) \neq h_b(x)}
    $$
    where $I$ and $\set{h_y \colon I \to \VV : y \in (0, \gamma)}$ are as in (i).
    For any $b \in (0, \gamma)$, the map $x \mapsto h_b(x)$ is a partial endomorphism on $I$, and hence is continuous.
    Since $\VV$ is an interval, such maps are completely determined by their value at any single element of $I \setminus \set{0}$, for example at $\gamma$.
    Hence the formula $x^1 = h_b(x^0)$ implies a finite disjunction of non-trivial linear dependencies in $\xvec = (x^i : i \in \omega)$ over $\VV$ if and only if $h_b(x^0) = q \cdot x^0$ for some $q \in \QQ \cap (0,1)$, if and only if $b = q \cdot \gamma$ for some $q \in \QQ \cap (0,1)$.
    Using Lemma \ref{lemma_rc_li_plus_li}, we see that for any $b \in (0, \gamma) \setminus \QQ\cdot \gamma$ there is an element $v \in I \setminus \set{0}$ such that $\theta(v) = h_b(v)$.
    Thus $b \not\in X$.
    Now fix $q \in \QQ \cap (0,1)$.
    If $C$ is algebraic, then $q$ is not a root of $\mipo(C)$, since this polynomial is irreducible over $\QQ$ and has degree at least $2$.
    In both cases, $C(X - q) = 0$; see Definition \ref{def_kernel_conf}.
    Therefore the map $(X - q)[\theta]$ is injective; see Definition \ref{def_T_C_theta} and Fact \ref{remark_alg_kc}.
    If $\theta(v) = h_{q\cdot \gamma}(v) = q \cdot v$ for some $v \in I \setminus \set{0}$, then this would contradict the injectivity of $(X - q)[\theta]$.
    Hence $q \cdot \gamma \in X$.
    We conclude that $X = (\QQ \cap (0,1)) \cdot \gamma$ is a definable set of cardinality $\aleph_0$ in an $\aleph_1$-saturated model, a contradiction.

    Now assume that $R_C$ is not a field.
    By Fact \ref{lemma_both_im_ker_inf}, there is $f \in \Kp{}$ such that both $\Ker(f)$ and $\Image(f)$ are infinite.
    In particular, $\Ker(f)$ is neither $\set{0}$ nor $\VV$.
    By Fact \ref{theorem_kernel_generic} the reduct $(\mm, \Ker(f))$ is an existentially closed model of the theory
$
T_V$.
    In every existentially closed model of $T_V$, the set $V$ is easily seen to be dense and codense in $\VV$.
    Since $K = \QQ$, this shows that $(\mm, \Ker(f))$ is an existentially closed model of the theory $T_{\mathcal{G}}$ from (ii).
    Therefore there is a set of cardinality $\aleph_0$ definable in $(\mm, \Ker(f))$.
    Since $(\mm, \Ker(f))$ is a reduct of $(\mm, \theta)$, this set is also definable in the $\aleph_1$-saturated model $(\mm, \theta)$, again a contradiction.
\end{proof}
\end{theorem}

\noindent Looking at the contents of \cite{Blo23}, it seems very likely to the author that Theorem \ref{theorem_prove_conj} can be generalized to other fields $K \neq \QQ$ (using o-minimality and $\dim(\VV) = 1$, it follows that $K$ must be an ordered field).
Similarly, one could probably also drop the assumption that all operations of $(\VV, 0, +, (\lambda \cdot)_{\lambda\in K})$ are continuous with respect to the order topology and work in the $t$-topology instead (perhaps using Lemma \ref{lemma_one_dim_non_linear_group}).
Dropping the assumption $\dim(\VV) = 1$ seems to be much more challenging, and dropping the assumption of o-minimality seems even harder.

\begin{remark} \label{rem_t_g_natp}
    Suppose that we are in the setting of Theorem \ref{theorem_prove_conj} and that $T$ satisfies \Hfour{}.
    As discussed in the proof of Theorem \ref{theorem_prove_conj}, this also implies that the model companion of the theory $T_{\mathcal{G}}$ from \cite{Blo23} exists.
    By the proof above and Fact \ref{theorem_kernel_generic}, we see that the model companion of $T_{\mathcal{G}}$ is a reduct of $T\theta^C$ for a suitable $C$.
    Since $T\theta^C$ is \NATP{}, the model companion of $T_{\mathcal{G}}$ is also \NATP{}.
\end{remark}

\section{The Linear Case} \label{sec_linear}

We now assume that $\mm := (M, 0, +, <, \dots)$ is a linear o-minimal expansion of an ordered group, and we set $T := \Th(\mm)$.
Throughout this section, $T'$ will denote the expansion of $T$ from Fact \ref{fact_is_red_of_o_vs} in the language $L' := \set{0, +, <, (\gamma \cdot)_{\gamma \in D}, (\mathfrak{c})_{\mathfrak{c} \in \dcl_L(\varnothing)}}$, where $D$ is the ordered division ring of all $\varnothing$-definable partial endomorphisms of $\mm$ modulo equivalence on an open neighborhood of $0$ (see Definition \ref{def_part_endo}).
Also note that $T$ must be model-complete, since it has quantifier elimination in the language from Fact \ref{fact_lin_qe}.
As ordered $D$-vector spaces are also model-complete, both $T$ and $T'$ satisfy our setting (see the beginning of Section \ref{sec_setting}).
Using Remark \ref{rem_vec_iso} and Theorem \ref{lemma_group_lin_iso}, we may assume from now on that
$$
(\VV, 0, +) = (M^d, 0, +).
$$
Since multiplication by any $\lambda \in K$ is an endomorphism of $\VV$, we see that for all $k, l \in \set{1, \dots, d}$, the map $\lambda_{k, l} \colon M \to M; x \mapsto \pi_k(\lambda \cdot (0, \dots, x, \dots, 0))$, where $\pi_k \colon M^d \to M$ is the projection to the $k$-th coordinate and $x$ is in the $l$-th entry, is an endomorphism of $(M, 0, +)$.
Since $\VV = M^d$, we can write every element $v \in \VV$ as a tuple $(a_1, \dots, a_d) \in M^d$, where $a_1 = \pi_1(v)$ and so on.
Treating these tuples as $M^{d \times 1}$ matrices, we see that multiplication by $\lambda$ is given by
$$
\lambda \cdot v =  \begin{tikzpicture}[baseline=(Frame.base)]
    \drawText{0}{0}{\lambda_{1,1}}
    \drawText{0}{2}{\lambda_{d,1}}
    \drawText{2}{0}{\lambda_{1,d}}
    \drawText{2}{2}{\lambda_{d,d}}
    \drawHDots{1}{0}{1}
    \drawHDots{1}{2}{1}
    \drawVDots{0}{1}{1}
    \drawVDots{2}{1}{1}
    \drawBorder{0}{0}{3}{3}
    \end{tikzpicture}
    \;\cdot\!\!
    \begin{tikzpicture}[baseline=(Frame.base)]
    \drawText{0}{0}{a_1}
    \drawText{0}{2}{a_d}
    \drawVDots{0}{1}{1}
    \drawBorder{0.2}{0}{0.8}{3}
    \end{tikzpicture}
    \;=\;
    \begin{tikzpicture}[baseline=(Frame.base)]
    \drawText{0}{0}{\lambda_{1, 1}(a_1) + \dots + \lambda_{1, d}(a_d)}
    \drawText{0}{2}{\lambda_{d, 1}(a_1) + \dots + \lambda_{d, d}(a_d)}
    \drawVDots{0}{1}{1}
    \drawBorder{-1.5}{0}{2.5}{3}
    \end{tikzpicture}.
$$
Since any partial $\varnothing$-definable endomorphism of $(M, 0, +)$ is an element of $D$, we see that each $\lambda = (\lambda_{k, l})_{1 \leq k, l \leq d}$ is actually a $D^{d \times d}$ matrix.
It is easy to check that addition and multiplication of $\lambda, \mu \in K$ coincide with addition and multiplication of the corresponding matrices $(\lambda_{k, l})_{1 \leq k, l \leq d}$ and $(\mu_{k, l})_{1 \leq k, l \leq d}$.
Hence the field $(K, 0, 1, +, \cdot)$ is actually a subring of $(D^{d \times d}, 0, \Id_d, +, \cdot)$.
It follows easily that $D^{d \times d}$ is a $K$-vector space.
If $\mm \models T'$, then every $\gamma \in D$ is an endomorphism of all of $M$, so, given a matrix $\Gamma = (\gamma_{k, l})_{1 \leq k, l \leq d} \in D^{d \times d}$ and $v = (a_1, \dots, a_d) \in M^d = \VV$, we can define $\Gamma \cdot v$ similar to $\lambda \cdot v$. We provide an alternative description of $D^{d \times d}$ in terms of the vector space $\VV$:

\begin{definition} \label{def_germs}
    We define the set $\mathfrak{G}$ of \textbf{ germs of $L(\varnothing)$-definable endomorphisms of $\VV$ at $\mathbf{0}$} as
    $$
    \set{f \colon U \to \VV : \text{``$f$ is a $\varnothing$-definable partial endomorphism of $(\VV, 0, +)$"}}/ \sim
    $$
    where we set $f \sim g$ if there is some open neighborhood $V \subseteq \VV$ of $0 \in \VV$ such that $f_{\restriction V} = g_{\restriction V}$.
\end{definition}

\noindent Clearly, every $\lambda \in K$ is also an element of $\mathfrak{G}$, so we obtain the following:

\begin{observation}
    The set $\mathfrak{G}$ is a $K$-vector space.
    Furthermore, the map
    \begin{align*}
        \Phi \colon \mathfrak{G} \to D^{d \times d}; \quad f \mapsto \big(x \mapsto \pi_k(f(\underbrace{0, \dots, x, \dots, 0}_{x \text{ in $l$-th entry}}))\big)_{1 \leq k, l \leq d}
    \end{align*}
    is a $K$-vector space isomorphism.
\begin{proof}
    Note that the map $f_{k, l}$ given by $x \mapsto \pi_k(f(0, \dots, x, \dots, 0))$ is, after properly restricting the domain, a $\varnothing$-definable partial endomorphism of $(M, 0, +)$, and hence has a corresponding element in $D$.
    Clearly $f \sim g$, with the relation $\sim$ from Definition \ref{def_germs}, if and only if $f_{k, l} \sim g_{k, l}$ for all $k, l$, with the relation $\sim$ from Definition \ref{def_skew_d}; since it is essentially the same equivalence relation, but on $M$ instead of $M^d$, we use the same symbol.
    With this, one easily sees that $\Phi$ is well defined and bijective.
    The rest is straightforward.
\end{proof}
\end{observation}

\begin{observation} \label{obser_dim_small}
    We have $\dim_K(\mathfrak{G}) \geq d = \dim(\VV)$.
\begin{proof}
    Identify $\mathfrak{G}$ with $D^{d \times d}$.
    Define $P_i := \Diag(0, \dots, 1, \dots, 0) \in D^{d \times d}$ to be the matrix with a $1$ on the $i$-th diagonal entry and $0$ everywhere else.
    It is easy to see that, for any $\lambda \in K$, the matrix $\lambda \cdot P_i$ can have nonzero entries only in the $i$-th column.
    With this, one easily shows that $P_1, \dots, P_d$ are $K$-linearly independent.
\end{proof}
\end{observation}

\noindent As we will see below (Lemma \ref{lemma_omin_sat_hfour}), both $T$ and $T'$ satisfy \Hfour{}.
This implies that the model companions $T\theta^C$ and $T'\theta^C$ both exist.
Recall that we have a ring $R_C$ of $\LKThe$-definable endomorphisms of $\VV$ in these models.
Since any $r \in R_C$ is an endomorphism of $\VV$, we see that for all $k, l \in \set{1, \dots, d}$, the map $r_{k, l} \colon M \to M; x \mapsto \pi_k(r(0, \dots, x, \dots, 0))$, where $\pi_k \colon M^d \to M$ is the projection to the $k$-th coordinate and $x$ is in the $l$-th entry, is an endomorphism of $(M, 0, +)$.
With this, we can treat any $r \in R_C$ as a matrix $(r_{k, l})_{1 \leq k, l \leq d}$ where $r \cdot v$ can be expressed similarly to $\lambda \cdot v$ by treating $v$ as an $M^{d\times 1}$ matrix.
Again, it is easy to see that matrix addition and multiplication coincide with the addition and multiplication in $R_C$.
Treating elements $r \in R_C$ as matrices, we may also write $r \cdot v$ instead of $r(v)$.
Note that $K$ is a subring of $R_C$, so, as previously mentioned, $R_C$ is a $K$-vector space and even a $K$-algebra.

\begin{lemma} \label{lemma_omin_sat_hfour}
    The theories $T$ and $T'$ satisfy \Hfour{}.
\begin{proof}
    Since $T$ is a reduct of $T'$, Lemma \ref{lemma_transfer_hfour_to_reduct} shows that it is enough to prove the claim for $T'$.
    By quantifier elimination in $T'$, every $L'$-formula $\psi(\ux; \uw)$, where $\ux = (x_1, \dots, x_m)$ and each $x_k$ is a variable for an element of $\VV = M^d$, i.e., a $d$-tuple of variables, is equivalent to a finite disjunction of formulas of the form
    $$
    \underbrace{\bigwedge\nolimits_{i=1}^q\sum\nolimits_{k=1}^m \sum\nolimits_{l=1}^d \gamma_{i, k, l} \cdot \pi_l(x_{k})  - s_i(\uw) = 0}_{\psi_0(\ux; \uw)}\; \wedge\; \underbrace{\bigwedge\nolimits_{i=1}^{q'}\sum\nolimits_{k=1}^m \sum\nolimits_{l=1}^d \gamma'_{i, k, l} \cdot \pi_l(x_{k})  - s'_i(\uw) > 0}_{\psi_1(\ux; \uw)}
    $$
    where all coefficients $\gamma_{i, k, l}$ and $\gamma'_{i, k, l}$ lie in $D$ and all $s_i(\uw)$ and $s'_i(\uw)$ are $L'$-terms in $\uw$.
    Fix any $\mm \models T'$ and $\ub \in M^{|\uw|}$.
    The conjunction $\psi_0(\ux; \ub)$ implies a finite disjunction of non-trivial $K$-linear dependencies in $\ux$ over $\VV$ if and only if either it is inconsistent or there are $(\lambda_1, \dots, \lambda_m) \in K^m \setminus \set{\uzero}$ and a $d$-tuple of terms $\ut(\uw) = (t_1(\uw), \dots, t_d(\uw))$ such that
    $$
    \bigwedge\nolimits_{i=1}^d \sum\nolimits_{k=1}^m \sum\nolimits_{l=1}^d (\lambda_k)_{i, l} \cdot \pi_l(x_k) = t_i(\uw)
    $$
    is obtained by taking $D$-linear combinations of the conjuncts of $\psi_0(\ux; \uw)$.
    This follows easily with basic linear algebra (recall that each $(\lambda_k)_{i, l}$ is the $(i, l)$-th entry of the $D^{d \times d}$-matrix that corresponds to the multiplication with $\lambda_{k}$).
    The choices of $(\lambda_1, \dots, \lambda_m)$ and $\ut(\uw)$ need not be unique, but the set of possible choices depends only on the coefficients $\gamma_{i, k, l}$ and the terms $s_i(\uw)$.
    Moreover, $\psi_0(\ux; \ub) \wedge \psi_1(\ux; \ub)$ defines a open subset (with respect to the subset topology) of the affine $D$-subspace defined by $\psi_0(\ux; \ub)$.
    Hence $\psi_0(\ux; \ub) \wedge \psi_1(\ux; \ub)$ is either inconsistent or implies exactly the same $D$-linear dependencies, and therefore the same $K$-linear dependencies, in $\ux$ over $\VV$ as $\psi_0(\ux; \ub)$.

    Write
    $
    \psi(\ux; \uw) \equiv \bigvee\nolimits_{i=1}^N \big(\psi_{i, 0}(\ux; \uw) \wedge \psi_{i, 1}(\ux; \uw)\big),
    $
    with each $\psi_{i, 0}(\ux; \uw)$ and $\psi_{i, 1}(\ux; \uw)$ as above.
    After reindexing (and redefining $q$), suppose that, for $i \leq q$, we have chosen a non-trivial $K$-linear dependency
    $$
    \lambda_{i, 1} \cdot x_1 + \dots + \lambda_{i, m} \cdot x_m = \ut_i(\uw)
    $$
    implied by $\psi_{i, 0}(\ux; \uw)$ as above, and that no such dependency is implied by $\psi_{i, 0}(\ux; \uw)$ for $i > q$.
    Then the formula
    $
    \exists \ux \in \VV : \psi(\ux; \uw) \wedge \bigwedge\nolimits_{i=1}^q \lambda_{i, 1} \cdot x_1 + \dots + \lambda_{i, m} \cdot x_m \neq \ut_i(\uw)
    $
    holds exactly when $\psi(\ux; \uw)$ implies no finite disjunction of non-trivial linear dependencies in $\ux$ over $\VV$. Hence, this formula can be chosen as $\sigma_\psi(\uw)$, the formula from condition \Hfour{} for $\psi(\ux; \uw)$.
\end{proof}
\end{lemma}

\begin{corollary}
    Both $T\theta^C$ and $T'\theta^C$ exist, and $T\theta^C$ is a reduct of $T'\theta^C$.
\begin{proof}
    This follows immediately from Theorem \ref{theorem_first_oder} and Theorem \ref{theorem_reduct_}.
\end{proof}
\end{corollary}

\begin{lemma} \label{lemma_qe_lin}
    The theory $T'\theta^C$ has quantifier elimination in the language $L' \cup L_{R_C}$, and the theory $T\theta^C$ has quantifier elimination in the language $\hat{L} \cup \LRC$, where $\hat{L}$ is the language from Fact \ref{fact_lin_qe}.
\begin{proof}
    It is easy to see that the equality $\spanA{A}{L'} = \acl_{L'}(A)$ holds for any $A \subseteq M' \models T'$.
    Since $T'$ has quantifier elimination in the language $L'$, Fact \ref{theorem_qe} yields the desired result.
    The result for $T\theta^C$ follows similarly from Fact \ref{theorem_qe} and Lemma \ref{lemma_span_is_acl_linear}.
\end{proof}
\end{lemma}

\noindent Recall that $\LRC$ is the language of $R_C$-modules.
Obviously, $0$ and $+$ are already in $L'$, so $\LRC$ only adds function symbols for the elements of the ring $R_C$.
Because $\VV$ is a $d$-ary set, we need to add multiple function symbols for each $r \in R_C$.
For convenience, we write each $r \in R_C$ as a matrix $r = (r_{k, l})_{1 \leq k, l \leq d}$, as in the paragraph above Lemma \ref{lemma_omin_sat_hfour}, and add a unary function symbol for each $r_{k, l}$.

\subsection{Neostability} \label{sec_neo_stab}

Distality was first introduced by Simon in \cite{Sim13} as a notion of ``pure instability'' for \NIP{} theories.
The following is essentially the external characterization of distality given in Lemma 2.7 of \cite{Sim13}, but without the assumption that $T$ is \NIP{}.

\begin{definition} \label{def_distal}
    A theory $T$ is \textbf{distal} if, for every model $\mm \models T$ and
    \begin{enumerate}[(i)]
        \item a set $U \subseteq M$,
        \item a sequence of tuples $I = I_1{}^\frown I_2$ in $M$ that is indiscernible over $U$, with $I_1$ cofinal and $I_2$ coinitial,
        \item a tuple $\ua$ for which $I_1{}^\frown(\ua)^\frown I_2$ is indiscernible over $\varnothing$,
    \end{enumerate}
    the sequence $I_1{}^\frown(\ua)^\frown I_2$ is indiscernible over $U$.
\end{definition}

\noindent The following result is attributed to Hieronymi and Nell and tells us that the definition of distality above coincides with the original one:

\begin{fact}[Proposition 2.9 in \cite{GK20}] \label{fact_nip_not_needed}
    If a theory $T$ is distal, as in Definition \ref{def_distal}, then it is also \NIP{}.
\end{fact}

\noindent From now on, we let $T$ and $T'$ be as in the beginning of Section \ref{sec_linear} once again.

\begin{theorem}
    The theory $T\theta^C$ is distal, and hence \NIP{}. \label{theorem_still_distal}
\begin{proof}
    Let $U \subseteq M$, $I = I_1{}^\frown I_2$, and $\ua$ be given as in (i), (ii), and (iii) of Definition \ref{def_distal}.
    We start with a bit of notation.
    Write $I_1 = (\ub_i : i \in \ii_1)$ and $I_2 = (\ub_i : i \in \ii_2)$.
    If $I = I_1{}^\frown I_2$ is constant, then the implication in Definition \ref{def_distal} is trivial; hence, we assume that $I$ is non-constant, so every subsequence of $I$ corresponds to a unique subsequence of $\ii_1{}^\frown \ii_2$.
    Let $\ud$ be a tuple of the form $\ud_1\ud_2$, where $\ud_1 = \ub_{i_{1,1}}\dots\ub_{i_{1,m}}$ with $i_{1, 1}, \dots, i_{1, m} \in \ii_1$, and $\ud_2 = \ub_{i_{2,1}}\dots\ub_{i_{2,n}}$ with $i_{2, 1}, \dots, i_{2, n} \in \ii_2$.
    We call such a tuple $\ud$ \textbf{special}.
    Given such a special tuple $\ud$, we define the following sequence:
    $$
    I_\ud := (\ub_i : i \in \ii_1, i > \max\set{i_{1,1}, \dots, i_{1, m}})^\frown(\ua)^\frown(\ub_i : i \in \ii_2, i < \min\set{i_{2,1}, \dots, i_{2, n}}),
    $$
    where we set $(\ub_i : i \in \ii_1, i > \max \varnothing) := I_1$ and $(\ub_i : i \in \ii_2, i < \min \varnothing) := I_2$.
    Since $I_1$ is cofinal and $I_2$ is coinitial, the sequence $I_\ud$ is always infinite and contains $\ua$.
    \begin{subclaim}
        Given an $\hat{L} \cup \LRC(IU)$-term $t(\ux; \ud\uu)$, where $\uu$ is a tuple from $U$ and $\ud$ is special, there are a special tuple $\ud' \supseteq \ud$, an $\hat{L} \cup \LRC(I)$-term $t'(\ux; \ud')$, and an $\hat{L} \cup \LRC(IU)$-term $t''(\ud'\uu)$ such that, for all $\uc \in I_{\ud'}$:
        $$
        t(\uc; \ud\uu) = t'(\uc; \ud') + t''(\ud'\uu).
        $$
    \begin{innerproof}
        We prove this by induction on the term $t(\ux; \ud\uu)$.
        The base cases are clear, as are the induction steps for $+$ and for a function symbol $r_{k, l}$ for an entry of the matrix $r \in R_C$; see the paragraph after Lemma \ref{lemma_qe_lin} and note that these function symbols define endomorphisms on all of $M$.
        The only interesting case is the induction step for a function symbol $\hat{g}$, where $g$ is a $\varnothing$-definable partial endomorphism.

        Suppose that $t(\ux; \ud\uu) = \hat{g}(t_0(\ux; \ud\uu))$.
        By the induction hypothesis, there are a special tuple $\ud_1 \supseteq \ud$, an $\hat{L} \cup \LRC(I)$-term $t'_0(\ux; \ud_1)$, and an $\hat{L} \cup \LRC(IU)$-term $t''_0(\ud_1\uu)$ such that
        $$
        t_0(\uc; \ud\uu) = t'_0(\uc; \ud_1) + t''_0(\ud_1\uu)
        $$
        holds for all $\uc \in I_{\ud_1}$.
        Choose $\uc^-, \uc^+ \in I_{\ud_1} \setminus \set{\ua}$ with $\uc^- <_s \ua <_s \uc^+$, where $<_s$ denotes the order of the entries in the sequence $I_1{}^\frown(\ua)^\frown I_2$.
        By the indiscernibility of $I$ over $U$, the element $t_0(\uc^-; \ud_1\uu)$ lies in the domain of $g$ if and only if $t_0(\uc^+; \ud_1\uu)$ lies in the domain of $g$, which is an interval.
        Similarly $t_0(\uc^-; \ud_1\uu)$ lies to the left or right of the domain of $g$ if and only if $t_0(\uc^+; \ud_1\uu)$ does.
        Since the indiscernibility of $I_1{}^\frown(\ua)^\frown I_2$ implies that $t'_0(\ua; \ud_1)$ lies in the closed interval spanned by $t'_0(\uc^-; \ud_1)$ and $t'_0(\uc^+; \ud_1)$, we see that $t_0(\ua; \ud_1\uu)$ lies in the domain of $g$ if and only if $t_0(\uc; \ud_1\uu)$ lies in the domain of $g$ for any, or equivalently every, $\uc \in I_{\ud_1} \setminus \set{\ua}$.
        If $t_0(\uc; \ud_1\uu)$ does not lie in the domain of $g$ for some, or equivalently every, $\uc \in I_{\ud_1}$, then we can define both terms in the statement to be $0$, as $\hat{g}$ is defined to be $0$ outside the domain of $g$, and set $\ud' := \ud_1$.

        Now assume that $t_0(\uc; \ud_1\uu)$ lies in the domain of $g$ for every $\uc \in I_{\ud_1}$.
        Recall that $g$ is a $\varnothing$-definable partial endomorphism on an interval $(-q_1, q_2)$ with $q_1, q_2 > 0$.
        We can easily find a $\varnothing$-definable partial endomorphism $\sigma$ that extends the given partial endomorphism $g$ to the interval $(-2\max(|q_1|, |q_2|), 2\max(|q_1|, |q_2|))$.
        Choose any $\ub_i \in I_{\ud_1} \setminus \set{\ua}$ and let $\ud'$ be $\ud_1\ub_i$, ordered as a special tuple.
        For any $\uc \in I_{\ud'}$, we obtain:
        \begin{align*}
            t(\uc; \ud\uu) &= \hat{g}(t_0(\uc; \ud\uu)) \\
            &= \hat{g}(t'_0(\uc; \ud_1) + t''_0(\ud_1\uu)) \\
            &= \hat{\sigma}(t'_0(\uc; \ud_1) + t''_0(\ud_1\uu)) \\
            &= \hat{\sigma}(t'_0(\uc; \ud_1) + t''_0(\ud_1\uu) - t'_0(\ub_i; \ud_1) - t''_0(\ud_1\uu)) + \hat{\sigma}(t'_0(\ub_i; \ud_1) + t''_0(\ud_1\uu)) \\
            &= \underbrace{\hat{\sigma}(t'_0(\uc; \ud_1) - t'_0(\ub_i; \ud_1))}_{:= t'(\uc; \ud')} + \underbrace{\hat{\sigma}(t'_0(\ub_i; \ud_1) + t''_0(\ud_1\uu))}_{:= t''(\ud'\uu)}.
        \end{align*}
        This completes the induction step.
    \end{innerproof}
    \end{subclaim}
    \noindent Recall that $T\theta^C$ eliminates quantifiers in the language $\hat{L} \cup \LRC$ and that $<$ is the only relation symbol in that language.
    Since $I$ is indiscernible over $U$, it therefore suffices to show that
    $$
    \text{``there is $\uc \in I_{\ud} \setminus \set{\ua}$ such that $t(\ua; \ud\uu) < 0 \Leftrightarrow t(\uc; \ud\uu) < 0$''}
    $$
    holds for all $\hat{L} \cup \LRC(IU)$-terms $t(\ux; \ud\uu)$ with $\ud$ special, in order to prove that $I_1{}^\frown(\ua)^\frown I_2$ is indiscernible over $U$.
    Here we use that the inequality $t_1(\ux; \ud\uu) < t_2(\ux; \ud\uu)$ is equivalent to $t_1(\ux; \ud\uu) - t_2(\ux; \ud\uu) < 0$, and that $x = y$ is equivalent to $\neg (x < y) \wedge \neg(y < x)$.
    Fix such a term $t(\ux; \ud\uu)$ and apply the claim to it, obtaining $t'(\ux; \ud')$ and $t''(\ud'\uu)$.
    As in the proof of the claim, choose $\uc^-, \uc^+ \in I_{\ud'} \setminus \set{\ua}$ with $\uc^- <_s \ua <_s \uc^+$, where $<_s$ denotes the order of the entries in the sequence $I_1{}^\frown(\ua)^\frown I_2$.
    Since $I_1{}^\frown(\ua)^\frown I_2$ is indiscernible, the element $t'(\ua; \ud')$ lies in the closed interval spanned by $t'(\uc^-; \ud')$ and $t'(\uc^+; \ud')$.
    Therefore, $t(\ua; \ud\uu) = t'(\ua; \ud') + t''(\ud'\uu)$ lies in the closed interval spanned by $t(\uc^-; \ud\uu) = t'(\uc^-; \ud') + t''(\ud'\uu)$ and $t(\uc^+; \ud\uu) = t'(\uc^+; \ud') + t''(\ud'\uu)$.
    Since $I$ is indiscernible over $U$, we have $t(\uc^-; \ud\uu) < 0$ if and only if $t(\uc^+; \ud\uu) < 0$.
    We conclude that $t(\uc^-; \ud\uu) < 0$ holds if and only if $t(\ua; \ud\uu) < 0$ holds, finishing the proof.
\end{proof}
\end{theorem}

\noindent Note that we have not directly used the fact that $T$ is distal.
Instead, we used that $T$ is a linear o-minimal theory.
As we have already seen, distality is, in general, not preserved by our construction; e.g., $\RCF{}$ is distal, but $\RCF{\theta}^C$ is not distal.

\subsection{\texorpdfstring{$\operatorname{dp}$-Ranks}{dp-Ranks}} \label{sec_linear_dp}

Next, we calculate the $\operatorname{dp}$-ranks of both $T\theta^C$ and $T'\theta^C$.
First, we recall the definition of a randomness pattern:

\begin{definition} \label{def_random_pattern}
    Let $T$ be a complete theory, let $\MM \models T$ be a sufficiently large monster, and let $\kappa$ be a cardinal.
    A \textbf{randomness pattern} of depth $\kappa$ is a collection of formulas $\set{\phi_\alpha(x; \uy_\alpha) : \alpha \in \kappa}$ and elements $\set{\ub_{\alpha, i} : \alpha \in \kappa, i \in \omega}$ such that for every function $\eta \colon \kappa \to \omega$, there is an element $a_\eta$ with
    $$
    \eta(\alpha) = i \quad \Leftrightarrow \quad \MM \models \phi_\alpha(a_\eta; \ub_{\alpha, i})
    $$
    for all $\alpha \in \kappa$ and $i \in \omega$.
    The $\operatorname{dp}$-rank of $T$, denoted by $\dprk{(T)}$, is the supremum of all cardinals $\kappa$ for which there is a randomness pattern of depth $\kappa$.
\end{definition}

\noindent Notice that randomness patterns and $\dprk{}$ are usually defined for (type-)definable sets $X$.
In this case, one replaces $x$ by a tuple $\ux$ of the arity of $X$, and $\ua_\eta$ has to lie in $X$.

Throughout this section, we work in a sufficiently large monster model $(\MM,\theta) \models T'\theta^C$, where $T$ and $T'$ are as discussed in the beginning of Section \ref{sec_linear}.
We may also assume that the respective reducts of $(\MM,\theta)$ are sufficiently large monster models of $T\theta^C$ and $T$.
Note that we actually compute the $\operatorname{dp}$-rank of the fixed completion $\Th(\MM, \theta)$ of $T'\theta^C$, but since this rank will be independent of the chosen completion, we will simply talk about the $\operatorname{dp}$-rank of $T'\theta^C$.

\begin{definition}
    We say that an $L' \cup \LRC$-term $t(y)$ (where $y$ is a single variable for an element in $M$ and not $\VV = M^d$) is \textbf{linear} if $t(a + b) = t(a) + t(b)$ holds for all $(\mm,\theta) \models T'\theta^C$ and $a, b \in M$.
\end{definition}

\noindent We identify two $L' \cup \LRC$-terms $t_1(y), t_2(y)$ if they are equivalent modulo $T'\theta^C$.
Since an $L' \cup \LRC$-term $t(y)$ is linear if and only if, after simplification, it contains no constant symbols besides $0$, and since we also have a function symbol for the function $y \mapsto 0$, we may treat linear terms as those that contain no constant symbols.

\begin{lemma} \label{lemma_simp_randomness_pattern}
    Every randomness pattern in $T'\theta^C$ with object variable $y$ can be simplified to a randomness pattern of the form $\set{t_\alpha(y) \in (z_1, z_2) : \alpha \in \kappa}$, $\set{b_{\alpha, i, 0} b_{\alpha, i, 1} : \alpha \in \kappa, i \in \omega}$, where $t_\alpha(y)$ is a linear $L' \cup \LRC$-term for every $\alpha \in \kappa$, and $I_{\alpha, i} := (b_{\alpha, i, 0}, b_{\alpha, i, 1})$ is an open interval for every $\alpha \in \kappa$ and $i \in \omega$.
    Furthermore, the terms $(t_\alpha(y) : \alpha \in \kappa)$ must be $D$-linearly independent, and we can assume that $I_{\alpha, i} < I_{\alpha, j}$ for $i < j$.
\begin{proof}
    Let $\set{\phi_\alpha(y; \uz_\alpha) : \alpha \in \kappa}$ and $\set{\ub_{\alpha, i} : \alpha \in \kappa, i \in \omega}$ be a randomness pattern in $M$ of depth $\kappa$.
    That is, for every $\eta \colon \kappa \to \omega$, there is an element $a_\eta \in M$ such that
    $$
    \eta(\alpha) = i \quad \Leftrightarrow \quad (\MM,\theta) \models \phi_\alpha(a_\eta; \ub_{\alpha, i})
    $$
    holds for all $\alpha \in \kappa$ and $i \in \omega$.
    Since $T'\theta^C$ eliminates quantifiers, we can assume that each $\phi_\alpha(y; \uz_\alpha)$ is of the form
    $$
    \bigvee\nolimits_{k=1}^{m_\alpha} \bigwedge\nolimits_{l=1}^{n_\alpha} \phi_{\alpha, k, l}(y; \uz_\alpha)
    $$
    with all $\phi_{\alpha, k, l}(y; \uz_\alpha)$ atomic.
    Indeed, no negated atomic formulas are needed, because we can replace
    \begin{enumerate}[(i)]
        \item $t_1(y; \uz_\alpha) \neq t_2(y; \uz_\alpha)$ with $t_1(y; \uz_\alpha) < t_2(y; \uz_\alpha) \vee -t_1(y; \uz_\alpha) < -t_2(y; \uz_\alpha)$, and
        \item $\neg t_1(y; \uz_\alpha) < t_2(y; \uz_\alpha)$ with $t_1(y; \uz_\alpha) = t_2(y; \uz_\alpha) \vee -t_1(y; \uz_\alpha) < -t_2(y; \uz_\alpha)$.
    \end{enumerate}
    We can assume that the array $\set{\ub_{\alpha, i} : \alpha \in \kappa, i \in \omega}$ is mutually indiscernible, i.e., that the sequences $(\ub_{\alpha, i} : i \in \omega)$ are mutually indiscernible.
    Let $\eta_0 \colon \kappa \to \omega$ be given by $\eta_0(\alpha) = 0$ for all $\alpha$.
    Then there is $a_{\eta_0} \in M$ such that
    $$
    i = 0 \quad \Leftrightarrow \quad (\MM,\theta) \models \phi_\alpha(a_{\eta_0}; \ub_{\alpha, i})
    $$
    holds for all $\alpha \in \kappa$ and $i \in \omega$.
    We observe that, for each $\alpha \in \kappa$, there is:
    \begin{enumerate}[(i)]
        \item some $k_\alpha$ for which $(\MM,\theta) \models \bigwedge\nolimits_{l=1}^{n_\alpha} \phi_{\alpha, k_\alpha, l}(a_{\eta_0}; \ub_{\alpha, 0})$ holds;
        \item some $l_{\alpha}$ and an infinite subset $\ii_{\alpha} \subseteq \omega_{>0}$ such that $(\MM,\theta) \models \neg \phi_{\alpha, k_\alpha, l_{\alpha}}(a_{\eta_0}; \ub_{\alpha, i})$ holds for all $i \in \ii_{\alpha}$.
    \end{enumerate}
    By the mutual indiscernibility of $\set{\ub_{\alpha, i} : \alpha \in \kappa, i \in \omega}$, for any $\eta \colon \kappa \to \omega$, we have
    $$
    \tp((\ub_{\alpha, 0})^\frown(\ub_{\alpha, i} : i \in \ii_{\alpha}) : \alpha \in \kappa) = \tp((\ub_{\alpha, \eta(\alpha)})^\frown(\ub_{\alpha, i} : i \in \omega_{>\eta(\alpha)}) : \alpha \in \kappa),
    $$
    so there must be an $a_\eta \in M$ such that
    \begin{align*}
        \eta(\alpha) = i &\quad \Rightarrow \quad (\MM,\theta) \models \phi_{\alpha, k_\alpha, l_{\alpha}}(a_\eta; \ub_{\alpha, i})\; \text{and} \\
        \eta(\alpha) < i &\quad \Rightarrow \quad (\MM,\theta) \models \neg \phi_{\alpha, k_\alpha, l_{\alpha}}(a_\eta; \ub_{\alpha, i})
    \end{align*}
    hold for all $\alpha \in \kappa$ and $i \in \omega$.
    Choose such an $a_\eta$ for every $\eta \colon \kappa \to \omega$.
    By the linearity of all functions in our language, and after replacing each parameter tuple $\ub_{\alpha, i}$ by a single parameter $b_{\alpha, i} := s_\alpha(\ub_{\alpha, i})$ for a suitable $L' \cup \LRC$-term $s_\alpha$, we can assume that either:
    \begin{enumerate}[(i)]
        \item The formula $\phi_{\alpha, k_\alpha, l_{\alpha}}(y; \uz_\alpha)$ is of the form $t_\alpha(y) = z$, where $t_\alpha(y)$ is a linear $L' \cup \LRC$-term.
        In this case, we have $t_\alpha(a_{\eta_0}) = b_{\alpha, 0}$ and $t_\alpha(a_{\eta_0}) \neq b_{\alpha, 1}$.
        Since $(b_{\alpha, i} : i \in \omega)$ is indiscernible, we obtain $b_{\alpha, i} \neq b_{\alpha, j}$ for $i \neq j$.
        Choose open intervals $I_{\alpha, i}$ around each $b_{\alpha, i}$ that are small enough to intersect no other $I_{\alpha, j}$.
        We obtain
        $$
        i = \eta(\alpha) \quad \Leftrightarrow \quad t_\alpha(a_\eta) \in I_{\alpha, i}
        $$
        for all $i \in \omega$ and $\eta \colon \kappa \to \omega$.
        By reindexing the $I_{\alpha, i}$'s, and possibly replacing $t_\alpha(y)$ and all $I_{\alpha, i}$'s with $-t_\alpha(y)$ and $-I_{\alpha, i}$, we can also assume that $I_{\alpha, 0} < I_{\alpha, 1} < \dots$ holds.
        \item The formula $\phi_{\alpha, k_\alpha, l_{\alpha}}(y; \uz_\alpha)$ is of the form $t_\alpha(y) < z$, where $t_\alpha(y)$ is a linear $L' \cup \LRC$-term.
        In this case, we have $t_\alpha(a_{\eta_0}) < b_{\alpha, 0}$ and $t_\alpha(a_{\eta_0}) \geq b_{\alpha, 1}$.
        We see that $(b_{\alpha, i} : i \in \omega)$ is a strictly decreasing sequence.
        In particular, $t_{\alpha}(a_{\eta}) = b_{\alpha, i}$ can only hold for $i = \eta(\alpha)+1$.
        By replacing the row $\set{b_{\alpha, i} : i \in \omega}$ with $\set{b_{\alpha, 2i} : i \in \omega}$, and each $a_\eta$ with $a_{\eta'}$ for the function $\eta' \colon \kappa \to \omega$ with $\eta'(\alpha) = 2\eta(\alpha)$ and $\eta'(\beta) = \eta(\beta)$ for all $\beta \neq \alpha$, we can also assume that $t_{\alpha}(a_{\eta}) \neq b_{\alpha, \eta(\alpha)+1}$.
        Now set $I_{\alpha, i} := (b_{\alpha, i+1}, b_{\alpha, i})$.
        Again, we obtain
        $$
        i = \eta(\alpha) \quad \Leftrightarrow \quad t_\alpha(a_\eta) \in I_{\alpha, i}
        $$
        for all $i \in \omega$.
        Replacing $t_{\alpha}(y)$ with $-t_{\alpha}(y)$ and each $I_{\alpha, i}$ with $-I_{\alpha, i}$, we can also assume that $I_{\alpha, 0} < I_{\alpha, 1} < \dots$ holds.
    \end{enumerate}
    \begin{subclaim}
        The terms $(t_\alpha(y) : \alpha \in \kappa)$ are $D$-linearly independent.
    \begin{innerproof}
        Without loss of generality, assume that
    \begin{align}
        t_0(y) = \sum\nolimits_{k=1}^m \gamma_{k} \cdot t_{k}(y) + \sum\nolimits_{l=1}^n \gamma_{m +l} \cdot t_{m+l}(y), \label{tag_D_lin_dep}
    \end{align}
    where $\gamma_1, \dots, \gamma_{m+n} \in D$ with $\gamma_1, \dots, \gamma_m < 0$ and $\gamma_{m+1}, \dots, \gamma_{m+n} > 0$.
    We also write the interval $I_{\alpha, i} = (b_{\alpha, i, 0}, b_{\alpha, i, 1})$ as in the statement, with $I_{\alpha, i} < I_{\alpha, j}$ for $i < j$.
    Let $\eta \colon \kappa \to \omega$ denote the function with $\eta(\alpha) = 1$ for all $\alpha \in \kappa$.
    Now, there is an element $a_\eta$ such that $t_\alpha(a_{\eta}) \in I_{\alpha, i} \Leftrightarrow i = 1 = \eta(\alpha)$.
    Using the equality (\ref{tag_D_lin_dep}), we can show
    \begin{align*}
        I_{0, 0} &\leq b_{0, 1, 0} \\
        &< t_0(a_\eta) \\
        &= \sum\nolimits_{k=1}^m \gamma_{k} \cdot t_{k}(a_\eta) + \sum\nolimits_{l=1}^n \gamma_{m +l} \cdot t_{m+l}(a_\eta) \\
        &< \sum\nolimits_{k=1}^m \gamma_{k} \cdot b_{k, 1, 0} + \sum\nolimits_{l=1}^n \gamma_{m +l} \cdot b_{m+l, 1, 1}.
    \end{align*}
    Now consider the function $\eta' \colon \kappa \to \omega$ with $\eta'(\alpha) = 0$ for $\alpha \in \set{0, \dots, m}$ and $\eta'(\alpha) = 2$ otherwise.
    Choose $a_{\eta'}$ such that $t_\alpha(a_{\eta'}) \in I_{\alpha, i} \Leftrightarrow i = \eta'(\alpha)$.
    We obtain
    \begin{align*}
        t_0(a_{\eta'}) &= \sum\nolimits_{k=1}^m \gamma_{k} \cdot t_{k}(a_{\eta'}) + \sum\nolimits_{l=1}^n \gamma_{m +l} \cdot t_{m+l}({a_{\eta'}}) \\
        & > \sum\nolimits_{k=1}^m \gamma_{k} \cdot b_{k, 0, 1} + \sum\nolimits_{l=1}^n \gamma_{m +l} \cdot b_{m+l, 2, 0} \\
        & \geq \sum\nolimits_{k=1}^m \gamma_{k} \cdot b_{k, 1, 0} + \sum\nolimits_{l=1}^n \gamma_{m +l} \cdot b_{m+l, 1, 1} \\
        & > I_{0, 0},
    \end{align*}
    contradicting $t_0(a_{\eta'}) \in I_{0, \eta'(0)} = I_{0, 0}$.
    Hence, (\ref{tag_D_lin_dep}) must be wrong, so we conclude that the sequence of linear $L' \cup \LRC$-terms $(t_\alpha(y) : \alpha \in \kappa)$ is $D$-linearly independent.
    \end{innerproof}
    \end{subclaim}
    \noindent This completes the proof of Lemma \ref{lemma_simp_randomness_pattern}.
\end{proof}
\end{lemma}

\noindent In the following, $\mathfrak{G}$ is the $K$-vector space of all germs of $L(\varnothing)$-definable endomorphisms of $\VV$ at $0$ from Definition \ref{def_germs}, and $R_C$ is the ring of all $\LKThe$-definable endomorphisms from Fact \ref{theorem_r_c_def}.

\begin{theorem} \label{corollary_dp_rank_btter_fml}
    The following holds:
    \begin{enumerate}[(i)]
        \item If $C$ is trivial (i.e., if $C$ is algebraic with $\deg(\mipo(C)) = 1$), then $\dprk{}(T\theta^C) = 1$.
        \item If $\dim_K(\mathfrak{G}) = 1$, then $\dprk{}(T\theta^C) = \dim_K(R_C)$.
        \item Otherwise, $\dprk{(T\theta^C)} = \max\set{\omega, \dim_K(\mathfrak{G}), \dim_K(R_C)}$.
    \end{enumerate}
    Furthermore, we have $\dprk{}(T\theta^C) = \dprk{}(T'\theta^C)$.
\begin{proof}
For now, we work in $T'\theta^C$.
Our goal is to find a $D$-basis of all linear $L' \cup \LRC$-terms in a single variable $y$ and show that this set of terms can be used to define a randomness pattern as in Lemma \ref{lemma_simp_randomness_pattern}.
To find this basis, we actually look at $d$-tuples of linear $L' \cup \LRC$-terms in a single variable $y$.
The reason is that multiplication by elements of $K$ and by elements of $R_C$ is defined on $\VV = M^d$.
Recall that each $r \in R_C$ can be written as a matrix $(r_{k, l})_{1 \leq k,l \leq d}$, and that we have a unary function symbol for every such $r_{k, l}$.
Similarly, we have a function symbol for every $\gamma \in D$.
Since all these function symbols define linear maps, given some $\sigma = (\sigma_{k, l})_{1 \leq k,l \leq d} \in D^{d\times d} \cup R_C$ and some $d$-tuple $\ut(y)$ of linear $L' \cup \LRC$-terms in a single variable $y$, the product
$$
\sigma \cdot \ut(y) =  \begin{tikzpicture}[baseline=(Frame.base)]
    \drawText{0}{0}{\sigma_{1,1}}
    \drawText{0}{2}{\sigma_{d,1}}
    \drawText{2}{0}{\sigma_{1,d}}
    \drawText{2}{2}{\sigma_{d,d}}
    \drawHDots{1}{0}{1}
    \drawHDots{1}{2}{1}
    \drawVDots{0}{1}{1}
    \drawVDots{2}{1}{1}
    \drawBorder{0}{0}{3}{3}
    \end{tikzpicture}
    \;\cdot\;
    \begin{tikzpicture}[baseline=(Frame.base)]
    \drawText{0}{0}{t_1(y)}
    \drawText{0}{2}{t_d(y)}
    \drawVDots{0}{1}{1}
    \drawBorder{0}{0}{1}{3}
    \end{tikzpicture}
    \;=\;
    \begin{tikzpicture}[baseline=(Frame.base)]
    \drawText{0}{0}{\sigma_{1, 1}(t_1(y)) + \dots + \sigma_{1, d}(t_d(y))}
    \drawText{0}{2}{\sigma_{d, 1}(t_1(y)) + \dots + \sigma_{d, d}(t_d(y))}
    \drawVDots{0}{1}{1}
    \drawBorder{-1.87}{0}{2.87}{3}
    \end{tikzpicture}
$$
is again a $d$-tuple of linear $L' \cup \LRC$-terms (as always, we identify $d$-tuples with $d\times1$-matrices).
This shows that the set of $d$-tuples of linear $L' \cup \LRC$-terms in $y$ is a $K$-vector space.
Since each $\lambda \in K$ is identified with a matrix in $D^{d\times d}$, we see that
$$
D^d(y) := \set{(\gamma_1(y), \dots, \gamma_d(y)) : \gamma_1, \dots, \gamma_d \in D}
$$
is a subspace of the set of $d$-tuples of linear $L' \cup \LRC$-terms in $y$.
    
\begin{subdefinition} \label{def_term_basis}
     We fix a $K$-basis $\Gg$ of $D^d(y)$, a set $\dd$ such that $\dd \cup \set{\Id_d}$ is a $K$-basis of $D^{d\times d}$, and a set $\rr$ such that $\rr \cup \set{1}$ is a $K$-basis of $R_C$.
     We recursively define (see Figure \ref{figure_tree_like_construction_b}):
    \begin{enumerate}[(i)]
        \item $\Cc_0 := \Gg$,
        \item $\bb^*_{i} := \rr \cdot \Cc_i := \set{ r \cdot c : r \in \rr, c \in \Cc_i }$, and
        \item $\Cc_{i+1} = \dd \cdot \bb^*_i$.
    \end{enumerate}
    We finally set $\bb := \bigcup_{i\in \omega}\bigcup_{k=1}^d \pi_k(\bb^*_i)$, where $\pi_k$ is the projection to the $k$-th component.
\end{subdefinition}

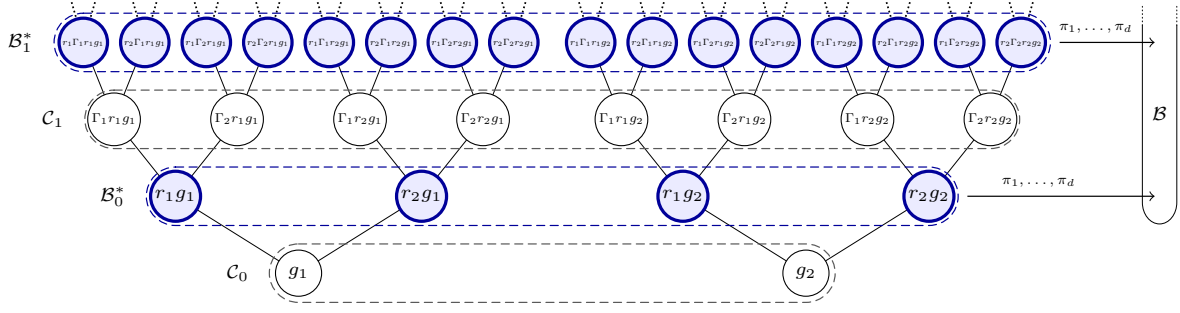
\begin{figure}[tbp]
    \centering
    \resizebox{\textwidth}{!}{\begin{tikzpicture}[
        x=0.80cm,
        y=1.00cm,
        term node/.style={circle, draw, fill=white, inner sep=1pt, minimum size=6mm, font=\scriptsize},
        bstar node/.style={term node, draw=blue!60!black, fill=blue!8, very thick},
        tree edge/.style={line width=0.35pt},
        continuation edge/.style={densely dotted, line width=0.65pt},
        projection arrow/.style={->, line width=0.45pt},
        set outline/.style={line width=0.35pt, line cap=round, line join=round},
        set continuation/.style={densely dotted, line width=0.35pt, line cap=round, line join=round},
        package outline/.style={draw=black!65, densely dashed, line width=0.45pt, line cap=round, line join=round},
        level label/.style={font=\scriptsize}
    ]
        \def\drawblock#1#2#3#4{\draw[package outline,#4]
                ([yshift=#3]#1)
                -- ([yshift=#3]#2)
                arc[start angle=90, end angle=-90, radius=#3]
                -- ([yshift=-#3]#1)
                arc[start angle=-90, end angle=-270, radius=#3]
                -- cycle;
        }

        \begin{scope}
            \coordinate (g1) at (0,0);
            \coordinate (g10) at (-2,1);
            \coordinate (g11) at (2,1);
            \coordinate (g100) at (-3,2);
            \coordinate (g101) at (-1,2);
            \coordinate (g110) at (1,2);
            \coordinate (g111) at (3,2);
            \coordinate (g1000) at (-3.5,3);
            \coordinate (g1001) at (-2.5,3);
            \coordinate (g1010) at (-1.5,3);
            \coordinate (g1011) at (-0.5,3);
            \coordinate (g1100) at (0.5,3);
            \coordinate (g1101) at (1.5,3);
            \coordinate (g1110) at (2.5,3);
            \coordinate (g1111) at (3.5,3);

            \draw[tree edge] (g1) -- (g10);
            \draw[tree edge] (g1) -- (g11);
            \draw[tree edge] (g10) -- (g100);
            \draw[tree edge] (g10) -- (g101);
            \draw[tree edge] (g11) -- (g110);
            \draw[tree edge] (g11) -- (g111);
            \draw[tree edge] (g100) -- (g1000);
            \draw[tree edge] (g100) -- (g1001);
            \draw[tree edge] (g101) -- (g1010);
            \draw[tree edge] (g101) -- (g1011);
            \draw[tree edge] (g110) -- (g1100);
            \draw[tree edge] (g110) -- (g1101);
            \draw[tree edge] (g111) -- (g1110);
            \draw[tree edge] (g111) -- (g1111);
            \foreach \leaf in {g1000,g1001,g1010,g1011,g1100,g1101,g1110,g1111} {
                \draw[continuation edge] (\leaf) -- ++(-0.2,0.55);
                \draw[continuation edge] (\leaf) -- ++(0.2,0.55);
            }

            \node[term node] at (g1) {$g_1$};
            \node[bstar node] at (g10) {$r_1g_1$};
            \node[bstar node] at (g11) {$r_2g_1$};
            \node[term node] at (g100) {\scalebox{0.68}{$\Gamma_1r_1g_1$}};
            \node[term node] at (g101) {\scalebox{0.68}{$\Gamma_2r_1g_1$}};
            \node[term node] at (g110) {\scalebox{0.68}{$\Gamma_1r_2g_1$}};
            \node[term node] at (g111) {\scalebox{0.68}{$\Gamma_2r_2g_1$}};
            \node[bstar node] at (g1000) {\scalebox{0.48}{$r_1\Gamma_1r_1g_1$}};
            \node[bstar node] at (g1001) {\scalebox{0.48}{$r_2\Gamma_1r_1g_1$}};
            \node[bstar node] at (g1010) {\scalebox{0.48}{$r_1\Gamma_2r_1g_1$}};
            \node[bstar node] at (g1011) {\scalebox{0.48}{$r_2\Gamma_2r_1g_1$}};
            \node[bstar node] at (g1100) {\scalebox{0.48}{$r_1\Gamma_1r_2g_1$}};
            \node[bstar node] at (g1101) {\scalebox{0.48}{$r_2\Gamma_1r_2g_1$}};
            \node[bstar node] at (g1110) {\scalebox{0.48}{$r_1\Gamma_2r_2g_1$}};
            \node[bstar node] at (g1111) {\scalebox{0.48}{$r_2\Gamma_2r_2g_1$}};
        \end{scope}

        \begin{scope}[xshift=6.60cm]
            \coordinate (g2) at (0,0);
            \coordinate (g20) at (-2,1);
            \coordinate (g21) at (2,1);
            \coordinate (g200) at (-3,2);
            \coordinate (g201) at (-1,2);
            \coordinate (g210) at (1,2);
            \coordinate (g211) at (3,2);
            \coordinate (g2000) at (-3.5,3);
            \coordinate (g2001) at (-2.5,3);
            \coordinate (g2010) at (-1.5,3);
            \coordinate (g2011) at (-0.5,3);
            \coordinate (g2100) at (0.5,3);
            \coordinate (g2101) at (1.5,3);
            \coordinate (g2110) at (2.5,3);
            \coordinate (g2111) at (3.5,3);

            \draw[tree edge] (g2) -- (g20);
            \draw[tree edge] (g2) -- (g21);
            \draw[tree edge] (g20) -- (g200);
            \draw[tree edge] (g20) -- (g201);
            \draw[tree edge] (g21) -- (g210);
            \draw[tree edge] (g21) -- (g211);
            \draw[tree edge] (g200) -- (g2000);
            \draw[tree edge] (g200) -- (g2001);
            \draw[tree edge] (g201) -- (g2010);
            \draw[tree edge] (g201) -- (g2011);
            \draw[tree edge] (g210) -- (g2100);
            \draw[tree edge] (g210) -- (g2101);
            \draw[tree edge] (g211) -- (g2110);
            \draw[tree edge] (g211) -- (g2111);
            \foreach \leaf in {g2000,g2001,g2010,g2011,g2100,g2101,g2110,g2111} {
                \draw[continuation edge] (\leaf) -- ++(-0.2,0.55);
                \draw[continuation edge] (\leaf) -- ++(0.2,0.55);
            }

            \node[term node] at (g2) {$g_2$};
            \node[bstar node] at (g20) {$r_1g_2$};
            \node[bstar node] at (g21) {$r_2g_2$};
            \node[term node] at (g200) {\scalebox{0.68}{$\Gamma_1r_1g_2$}};
            \node[term node] at (g201) {\scalebox{0.68}{$\Gamma_2r_1g_2$}};
            \node[term node] at (g210) {\scalebox{0.68}{$\Gamma_1r_2g_2$}};
            \node[term node] at (g211) {\scalebox{0.68}{$\Gamma_2r_2g_2$}};
            \node[bstar node] at (g2000) {\scalebox{0.48}{$r_1\Gamma_1r_1g_2$}};
            \node[bstar node] at (g2001) {\scalebox{0.48}{$r_2\Gamma_1r_1g_2$}};
            \node[bstar node] at (g2010) {\scalebox{0.48}{$r_1\Gamma_2r_1g_2$}};
            \node[bstar node] at (g2011) {\scalebox{0.48}{$r_2\Gamma_2r_1g_2$}};
            \node[bstar node] at (g2100) {\scalebox{0.48}{$r_1\Gamma_1r_2g_2$}};
            \node[bstar node] at (g2101) {\scalebox{0.48}{$r_2\Gamma_1r_2g_2$}};
            \node[bstar node] at (g2110) {\scalebox{0.48}{$r_1\Gamma_2r_2g_2$}};
            \node[bstar node] at (g2111) {\scalebox{0.48}{$r_2\Gamma_2r_2g_2$}};
        \end{scope}

        \coordinate (capCzeroLeft) at (0,0);
        \coordinate (capCzeroRight) at ([xshift=6.60cm]0,0);
        \coordinate (capBzeroLeft) at (-2,1);
        \coordinate (capBzeroRight) at ([xshift=6.60cm]2,1);
        \coordinate (capConeLeft) at (-3,2);
        \coordinate (capConeRight) at ([xshift=6.60cm]3,2);
        \coordinate (capBoneLeft) at (-3.5,3);
        \coordinate (capBoneRight) at ([xshift=6.60cm]3.5,3);
        \drawblock{capCzeroLeft}{capCzeroRight}{3.8mm}{}
        \drawblock{capBzeroLeft}{capBzeroRight}{3.8mm}{draw=blue!60!black}
        \drawblock{capConeLeft}{capConeRight}{3.8mm}{}
        \drawblock{capBoneLeft}{capBoneRight}{3.8mm}{draw=blue!60!black}
        \node[level label, anchor=east] at (-0.65,0) {$\Cc_0$};
        \node[level label, anchor=east] at (-2.65,1) {$\bb^*_0$};
        \node[level label, anchor=east] at (-3.65,2) {$\Cc_1$};
        \node[level label, anchor=east] at (-4.15,3) {$\bb^*_1$};

        \coordinate (bSetLeftBottom) at (13.72,0.92);
        \coordinate (bSetRightBottom) at (14.28,0.92);
        \coordinate (bSetLeftTop) at (13.72,3.22);
        \coordinate (bSetRightTop) at (14.28,3.22);
        \coordinate (bSetInZero) at (13.92,1);
        \coordinate (bSetInOne) at (13.92,3);
        \draw[set outline] (bSetLeftTop) -- (bSetLeftBottom)
            arc[start angle=180, end angle=360, radius=0.28]
            -- (bSetRightTop);
        \draw[set continuation] (bSetLeftTop) -- (13.72,3.47);
        \draw[set continuation] (bSetRightTop) -- (14.28,3.47);
        \node[level label] at (14,2.07) {$\bb$};
        \draw[projection arrow] ([xshift=5mm]capBzeroRight) -- node[pos=0.38, above, level label] {\scalebox{0.75}{$\pi_1,\dots,\pi_d$}} (bSetInZero);
        \draw[projection arrow] ([xshift=5mm]capBoneRight) -- node[pos=0.38, above, level label] {\scalebox{0.75}{$\pi_1,\dots,\pi_d$}} (bSetInOne);
    \end{tikzpicture}}
    \caption[The tree-like construction of $\bb$]{The tree-like construction of the set $\bb$ in the case where $\Gg=\set{g_1,g_2}$, $\rr=\set{r_1,r_2}$, and $\dd=\set{\Gamma_1,\Gamma_2}$ each have precisely two elements. Starting from each $g_i \in \Cc_0$, one alternates multiplication by elements of $\rr$ and $\dd$. The set $\bb$ consists precisely of the coordinate projections of the elements of the sets $\bb^*_i$.}
    \label{figure_tree_like_construction_b}
\end{figure}

\noindent With the above, we see that each $\Cc_i$ is a set of $d$-tuples of linear $L' \cup \LRC$-terms in $y$, that each $\bb^*_i$ is another such set, and that $\bb$ is a set of linear $L' \cup \LRC$-terms in $y$.
In a series of claims, we show that $\bb \cup \set{\Id}$ is a $D$-basis of all linear $L' \cup \LRC$-terms in $y$.

\begin{subclaim} \label{lemma_independent_and_dense}
    The set $\bb$ is $D$-linearly independent over $\Id$.
    Furthermore, the set 
    $$
    \set{(t(a) : t(y) \in \bb \cup \set{\Id}) : a \in M}
    $$
    is dense in $M^{|\bb|+1}$.
    More explicitly, given intervals $(I_t : t(y) \in \bb \cup \set{\Id})$ in the monster model $(\MM,\theta)$ fixed for this section, we can find some $a \in M$ with $t(a) \in I_t$ for all $t(y) \in \bb \cup \set{\Id}$.
\begin{innerproof}
    Fix finite subsets $\Gg_0 \subseteq \Gg$, $\dd_0 \subseteq \dd$, and $\rr_0 \subseteq \rr$, and fix some $N \geq 0$.
    For each $q \geq 0$, we set 
    $$
    \ii_q :=  \Set{\big\langle (r_i, \Gamma_i) : 0 \leq i < q \big\rangle : r_i \in \rr_0, \Gamma_i \in \dd_0},
    $$
    and define $\ii_{\leq N} := \bigcup_{q=0}^N \ii_q$ and $\ii_{<N}$ similarly.
    For every element $r \in \rr_0 \cup \set{1}$, we define the tuple $\ux_r := (x^{\sigma}_{g, r} : \sigma \in \ii_{\leq N}, g \in \Gg_0)$, where each $x^{\sigma}_{g, r}$ is a variable in $\VV$, i.e., a $d$-tuple of variables. Given $\sigma = \big\langle (r_1, \Gamma_1), \dots, (r_n, \Gamma_n) \big\rangle$, we will more or less use the variable $x^{\sigma}_{g, r}$ as a placeholder for $r \cdot \Gamma_n \cdot r_n \cdots \Gamma_1 \cdot r_1 \cdot g(y)$, and, since $r$ can also be $1$, this means that we have placeholders for all nodes in Figure \ref{figure_tree_like_construction_b} of depth $2N+1$ generated by $\Gg_0$, $\dd_0$, and $\rr_0$.
    
    For each $(g, \sigma, r, k) \in \Gg_0 \times \ii_{\leq N} \times \rr_0 \times \set{1, \dots, d}$, choose a non-empty open interval $I_{g, \sigma, r, k}$ and some coefficient $\gamma_{g, \sigma, r, k} \in D$.
    Also let $I$ be another non-empty open interval.
    Now consider the formula
    \begin{align*}
        \psi\big((\ux_r : r \in \rr_0 \cup \set{1})\big) := \exists y \in I : &\bigwedge\nolimits_{g \in \Gg_0} x^{\emptyseq}_{g, 1} = g(y) \\
        &\wedge \bigwedge\nolimits_{g \in \Gg_0} \bigwedge\nolimits_{\sigma \in \ii_{\leq N}}\bigwedge\nolimits_{r \in \rr_0} \bigwedge\nolimits_{k=1}^d  \pi_{k}(x^\sigma_{g, r}) \in I_{g, \sigma, r, k} \\
        &\wedge \bigwedge\nolimits_{g \in \Gg_0} \bigwedge\nolimits_{\sigma \in \ii_{< N}}\bigwedge\nolimits_{r \in \rr_0} \bigwedge\nolimits_{\Gamma \in \dd_0} x^{\sigma^\frown \scalebox{0.5}{$\langle(r, \Gamma)\rangle$}}_{g, 1} = \Gamma \cdot x^\sigma_{g, r}
    \end{align*}
    It is easy to verify that the formula above implies no finite disjunction of non-trivial $K$-linear dependencies over $\VV$.
    Here we use that the $\Gamma$'s are $K$-linearly independent over $\Id_d$, and that the $g$'s are $K$-linearly independent.

    Now Lemma \ref{lemma_rc_li_plus_li} yields $(\MM, \theta) \models \exists \ux \in \VV : \psi\big( r(\ux) : r \in \rr_0 \cup \set{1})\big)$, where we define $\ux := (x^\sigma_g : \sigma \in \ii_{\leq N}, g \in \Gg_0)$.
    Looking at the first and third lines in the definition of the formula $\psi\big((\ux_r : r \in \rr_0 \cup \set{1})\big)$, we see that $\psi\big( r(\ux) : r \in \rr_0 \cup \set{1})\big)$ implies
    $$
    x^\sigma_{g} = \Gamma_{q-1} \cdot r_{q-1} \cdots \Gamma_0 \cdot r_0 \cdot g(y) \; \text{and therefore} \; \pi_k(r \cdot \Gamma_{q-1} \cdot r_{q-1} \cdots \Gamma_0 \cdot r_0 \cdot g(y)) \in I_{g, \sigma, r, k}
    $$
    for $\sigma = \langle (r_0, \Gamma_0), \dots, (r_{q-1}, \Gamma_{q-1}) \rangle \in \ii_q$ and $r \in \rr_0$.
    For this $\sigma$, note that the product $r \cdot \Gamma_{q-1} \cdot r_{q-1} \cdots \Gamma_0 \cdot r_0 \cdot g(y)$ lies in $\bb^*_{q}$ and that its image under $\pi_k$ lies in $\bb$.
    One can easily check that, given any finite subset of $\bb$, we can write all of its elements in such a form for some $\Gg_0, \dd_0, \rr_0$, and $N > 0$.
    By varying the intervals and using compactness, we obtain that $\set{(t(a) : t(y) \in \bb \cup \set{\Id}) : a \in M}$ is dense in $M^{|\bb| + 1}$ (recall that we work in a sufficiently large monster model).
    This density implies that $\bb \cup \set{\Id}$ is $D$-linearly independent.
    Indeed, if not, then there are distinct $t_1(y), \dots, t_n(y) \in \bb \cup \set{\Id}$ and coefficients $(\gamma_1, \dots, \gamma_n) \in D^n \setminus \set{\uzero}$ such that
    $$
    \sum\nolimits_{l=1}^n \gamma_l \cdot t_l(y) = 0
    $$
    holds for all $y$.
    The image of the map $x \mapsto (t_1(x), \dots, t_n(x))$ is then contained in an $L'$-definable set of dimension less than $n$, contradicting the density just proved.
\end{innerproof}
\end{subclaim}

\begin{subclaim} \label{lemma_product_easy_lol}
    Every product of the form $\sigma_q \cdots \sigma_1 \cdot g(y)$, with $\sigma_i \in D^{d\times d} \cup R_C$ for all $i$ and $g \in D^d(y)$, can be written as a $K$-linear combination of elements in $\bigcup_{i \in \omega}(\Cc_i \cup \bb^*_i)$.
\begin{innerproof}
    The case $q = 0$ is clear, as $\Cc_0 = \Gg$ is a $K$-basis of $D^d(y)$.
    Now assume that we have already shown 
    $$
    \sigma_{q-1} \cdots \sigma_1 \cdot g(y) = \sum\nolimits_{l=1}^n \lambda_l \cdot c_l(y)
    $$
    with each $\lambda_l \in K$ and each $c_l(y) \in \bigcup_{i \in \omega}(\Cc_i \cup \bb^*_i)$.
    It remains to show that a product of the form $\sigma \cdot \lambda \cdot c(y)$ with $\sigma \in D^{d\times d} \cup R_C$, $\lambda \in K$, and $c(y) \in \bigcup_{i \in \omega}(\Cc_i \cup \bb^*_i)$ can be written as a $K$-linear combination of elements in $\bigcup_{i \in \omega}(\Cc_i \cup \bb^*_i)$.
    To do so, distinguish between the cases $\sigma \in D^{d\times d}$ and $\sigma \in R_C$, and also between the cases $c(y) \in \bb^*_i$, $c(y) \in \Cc_0$, and $c(y) \in \Cc_{i+1}$.

    We present the proof in the case $\sigma \in R_C$ and $c(y) \in \bb^*_i$.
    First, notice that $c(y) = r \cdot c'(y)$ for some $r \in \rr$ and $c'(y) \in \Cc_i$.
    Since the product $\sigma \cdot \lambda \cdot r$ lies in $R_C$, and since $\rr \cup \set{1}$ is a $K$-basis of $R_C$, we can write $\sigma \cdot \lambda \cdot r = \lambda'_0 + \sum_{k=1}^m \lambda'_k \cdot r'_k$, where $\lambda'_0, \dots, \lambda'_m \in K$ and $r'_1, \dots, r'_m \in \rr$.
    With this, we obtain that
    $$
    \sigma \cdot \lambda \cdot c(y) = \sigma \cdot \lambda \cdot r \cdot \underbrace{c'(y)}_{\in \Cc_i} = \lambda'_0 \cdot c'(y) + \sum\nolimits_{k=1}^m \lambda'_k \cdot \underbrace{r'_k \cdot c'(y)}_{\in \bb^*_i}
    $$
    is a $K$-linear combination of elements in $\bigcup_{i \in \omega}(\Cc_i \cup \bb^*_i)$.
    The other cases are similar.
\end{innerproof}
\end{subclaim}

\begin{subclaim} \label{lemma_basis_of_linear_terms}
    The set $\bb \cup \set{\Id}$ is a $D$-basis of all linear $L' \cup \LRC$-terms in a single variable $y$.
\begin{innerproof}
    We already know that $\bb$ is $D$-linearly independent over $\Id$, so it remains to show that every linear $L' \cup \LRC$-term in $y$ can be written as a $D$-linear combination of elements in $\bb \cup \set{\Id}$.
    We can easily see that every such term can be written as a $D$-linear combination of terms of the form
    $$
    t(y) = (r_{q})_{k_q, l_q} \cdot \gamma_{q} \cdots (r_1)_{k_1, l_1} \cdot \gamma_1(y) 
    $$
    where $\gamma_1, \dots, \gamma_{q} \in D$, and $(r_i)_{k, l}$ denotes the $(k, l)$-th entry of the $d \times d$-matrix $r_i \in R_C$.
    For each $i \in \set{1, \dots, q}$, let $P_i, Q_i \in D^{d\times d}$ be the matrices consisting of zeros and ones such that $P_i \cdot r_i \cdot Q_i$ is the matrix that has $(r_i)_{k_i, l_i}$ as the $(1, 1)$-th entry and $0$ everywhere else.
    Setting $\Gamma_i := Q_{i+1} \cdot \gamma_{i+1} \cdot P_{i}$ for $i \in \set{1, \dots, q-1}$ (treating $\gamma_{i+1}$ as the corresponding diagonal matrix), $\Gamma_{q} := P_q$, and $g(y) := Q_1 \cdot (
        \gamma_1(y),
        0, 
        \dots,
        0
    )$, we see that 
    $$
    t(y) = \pi_{1}(\Gamma_q \cdot r_q \cdot \Gamma_{q-1} \cdot r_{q-1} \cdots \Gamma_1 \cdot r_1 \cdot g(y)).
    $$
    Now Claim \ref{lemma_product_easy_lol} yields $\Gamma_q \cdot r_q \cdot \Gamma_{q-1} \cdot r_{q-1} \cdots \Gamma_1 \cdot r_1 \cdot g(y) = \sum_{k=1}^m \lambda_k \cdot c_k(y)$, where $\lambda_k \in K$ and $c_k(y) \in \bigcup_{i \in \omega}(\Cc_i \cup \bb^*_i)$ hold for all $k$.
    If $c_k(y) \in \Cc_0$, then, as $\Cc_0 \subseteq D^d(y)$, we can write $\pi_1(\lambda_k \cdot c_k(y)) = \gamma(y)$ for some $\gamma \in D$.
    If $c_k(y) \in \Cc_{i+1}$, then, as $\Cc_{i+1} = \dd \cdot \bb^*_i$ and $\dd \subseteq D^{d\times d}$, we can write $\pi_1(\lambda_k \cdot c_k(y)) = \pi_1(\Gamma' \cdot c'(y))$ for some $\Gamma' \in D^{d\times d}$ and $c'(y) \in \bb^*_i$.
    The same is trivially true if $c_k(y) \in \bb^*_i$, so we obtain
    $$
    t(y) = \gamma \cdot y + \sum\nolimits_{l = 1}^{n} \pi_1(\Gamma'_l \cdot c'_l(y)) 
    $$
    where $\gamma \in D$, each $\Gamma'_l \in D^{d\times d}$, and each $c'_l(y) \in \bb^*_i$ for some $i \in \omega$.
    It is easy to verify that $\pi_1 \circ \Gamma'_l = \sum_{k=1}^d (\Gamma'_l)_{1, k} \cdot \pi_k$, so we can conclude that $t(y)$ can be written as a $D$-linear combination of elements in $\bb \cup \set{\Id}$ (see Definition \ref{def_term_basis}).
\end{innerproof}
\end{subclaim}

\noindent We can now calculate the $\dprk{}$ of $T'\theta^C$, and with a small trick also of $T\theta^C$.

\begin{subclaim} \label{clai_dp_rk}
    The following holds:
    $$
    \dprk{(T\theta^C)} = \dprk{(T'\theta^C)} = \dim_D(\text{``linear $L'\cup\LRC$-terms in a single variable $y$"}) = |\bb| + 1.
    $$
\begin{innerproof}
    By Lemma \ref{lemma_simp_randomness_pattern}, every randomness pattern in $T'\theta^C$ must have a depth equal to the cardinality of some $D$-linearly independent set of linear terms in a single variable $y$.
    Together with Claim \ref{lemma_basis_of_linear_terms}, this yields
    $
    \dprk{(T'\theta^C)} \leq |\bb| + 1.
    $
    To show ``$\geq$", choose disjoint non-empty open intervals $(I_i : i \in \omega)$.
    Using the denseness from Claim \ref{lemma_independent_and_dense}, we see that the partial type
    $
    \set{ t(y) \in I_{\eta(t)} : t \in \bb \cup \set{\Id} }
    $
    has a realization $a_\eta$ for every function $\eta \colon \bb \cup \set{\Id} \to \omega$.
    Since the intervals are disjoint, we obtain $t(a_\eta) \in I_i \Leftrightarrow i = \eta(t)$ for all $t \in \bb \cup \set{\Id}$ and $i \in \omega$.
    Thus we have found a randomness pattern of depth $|\bb| + 1$ in $T'\theta^C$, so $\dprk{(T'\theta^C)} \geq |\bb| + 1$.

    Since $T\theta^C$ is a reduct of $T'\theta^C$, we immediately obtain $\dprk{(T\theta^C)} \leq \dprk{(T'\theta^C)}$.
    Recall that every element of $D$ is represented by a $\varnothing$-definable partial endomorphism in $T$, and these representatives are extended to total endomorphisms in $T'$.
    Thus, by choosing all $I_i$ in an infinitesimal neighborhood of $0$, the same randomness pattern also works in $T\theta^C$.
    Hence $\dprk{(T\theta^C)} \geq |\bb| + 1$.
\end{innerproof}
\end{subclaim}

\noindent This proves the final assertion of Theorem \ref{corollary_dp_rank_btter_fml}.
It remains to prove the explicit formulas for $\dprk{}\!(T\theta^C)$ in terms of $\mathfrak{G}$ and $R_C$.
Recall that the set $\mathfrak{G}$ of all germs of $L(\varnothing)$-definable endomorphisms of $\VV$ at $0$ is, as a $K$-vector space, isomorphic to $D^{d\times d}$.
In particular, we have $\dim_K(\mathfrak{G}) = |\dd| + 1$, as $\dd \cup \set{\Id_d}$ is a $K$-basis of $D^{d\times d}$.
Using Claim \ref{clai_dp_rk} and our explicit construction of $\bb$ in Definition \ref{def_term_basis}, we now calculate $\dprk{}(T\theta^C)$ in the cases from the statement of Theorem \ref{corollary_dp_rank_btter_fml}:
\begin{enumerate}[(i)]
    \item The kernel configuration $C$ is trivial.
    In this case, $T$ and $T\theta^C$ are interdefinable, so, since $T$ is o-minimal, we have $1 = \dprk{}(T) = \dprk{}(T\theta^C)$.
    \item We have $\dim_K(\mathfrak{G}) = 1$, or equivalently $\dim_K(D^{d\times d}) = 1$, and $C$ is non-trivial.
    Note that $\dim_K(D^{d\times d}) = 1$ implies $\dim_K(D^d(y)) = 1$, where $D^d(y)$ was defined above Definition \ref{def_term_basis}, since the matrices with non-zero entries only in the first column form a subspace of $D^{d \times d}$ that is isomorphic to $D^d(y)$.
    Hence, in Definition \ref{def_term_basis}, we obtain $\Cc_0 = \set{1(y)}$, $\bb^*_0 = \rr$, and $\Cc_i = \bb^*_i = \varnothing$ for $i \geq 1$.
    Because $\dim_K(\mathfrak{G}) = 1$ also implies $d = 1$ (see Observation \ref{obser_dim_small}), we conclude
    $$
    \dprk{}(T\theta^C) = |\bb| + 1 = |\rr \cup \set{1}| = \dim_K(R_C).
    $$
    \item We have $\dim_K(\mathfrak{G}) > 1$ and $C$ is non-trivial.
    Looking at Definition \ref{def_term_basis}, one can easily verify 
    $$
    |\bb| = d \cdot |\Gg| \cdot \Big( \sum\nolimits_{i \in \omega} |\rr| \cdot (|\dd| \cdot |\rr|)^i \Big).
    $$
    Using the column argument from (ii), one can show that $d \cdot |\Gg| = \dim_K(D^{d\times d}) = |\dd| + 1$.
    Using cardinal arithmetic, one can check that the sum evaluates to $\max \set{\omega, |\dd|, |\rr|}$.
    Since this cardinal is always infinite, we obtain
    $$
    \dprk{(T\theta^C)} = |\bb| + 1 = \max\set{\omega, |\dd| + 1, |\rr| + 1} = \max \set{\omega, \dim_K(\mathfrak{G}), \dim_K(R_C)}.
    $$
\end{enumerate}
This completes the proof of Theorem \ref{corollary_dp_rank_btter_fml}.
\end{proof}
\end{theorem}

\begin{remark}
    If one also defines the $\operatorname{dp}$-rank for all type-definable sets, then Observation \ref{obser_dim_small}, Theorem \ref{corollary_dp_rank_btter_fml}, and the fact that $\dprk{}(M^d) = d \cdot \dprk{(M)}$ always holds yield $\dprk{}(T\theta^C) = \dprk{}(\VV)$ if $C$ is non-trivial.
\end{remark}

\subsection{Exchange for the Algebraic Closure} \label{sec_acl_o_min}

\noindent In Fact \ref{theorem_acl}, we showed that the algebraic closure in $T\theta^C$ is $\cl_\theta$.
Here, $\cl_\theta(A)$ is the smallest set containing $A$ that is closed under both $\acl_L$ and multiplication by every $r \in R_C$.
We have also seen that $\cl_\theta$ cannot have the exchange property unless $R_C$ is a field.
In this section, we show that, when $T$ is linear, $\cl_\theta$ almost never has the exchange property; in fact, a single partial endomorphism of $\VV$ that is not multiplication by some $\lambda \in K$ on its domain is enough to ensure that $\cl_\theta$ does not have the exchange property.

\begin{theorem} \label{theorem_acl_fail}
    The algebraic closure $\acl_{L_\theta} = \cl_\theta$ has the exchange property in $T\theta^C$ if and only if either
    \begin{enumerate}[(i)]
        \item $C$ is trivial.
        \item $\dim_K(\mathfrak{G}) = 1$ and $R_C$ is a field.
    \end{enumerate}
\begin{proof}
    The case where $C$ is trivial is clear, since in this case $T$ and $T\theta^C$ are interdefinable. For the rest of this proof, we assume that $C$ is non-trivial. First, assume that $\acl_{L_\theta} = \cl_\theta$ has the exchange property.
    By Fact \ref{theorem_nes_cond_fixed}, we instantly see that $\dim(\VV) = 1$ and that $R_C$ is a field. This means that we have $\VV = M$. Assume, toward a contradiction, that $\dim_K(\mathfrak{G}) > 1$.
    This implies that there is an $L(\varnothing)$-definable partial endomorphism $f \colon I \subseteq \VV \to \VV$ that, on its domain, is not equal to multiplication by any $\lambda \in K$. Fix any $(\mm, \theta) \models T\theta^C$ and $a \in M \setminus \cl_\theta(\varnothing)$.
    \begin{subclaim} \label{claim_pi_is_dogshit}
        The formula $x_0 \in I \wedge x_1 - f(x_0) = a$ implies no finite disjunction of non-trivial linear dependencies in $x_0x_1$ over $\VV$.
    \begin{innerproof}
        Suppose we had $\mm \models \forall x_0, x_1 : x_0 \in I \wedge x_1 - f(x_0) = a \rightarrow \bigvee_{i=1}^q \lambda_{i, 0} \cdot x_0 + \lambda_{i, 1} \cdot x_1 = u_i$, where each $(\lambda_{i, 0}, \lambda_{i, 1}) \in K^2 \setminus \set{\uzero}$ and each $u_i \in \VV$.
        The subset on which each disjunct holds is definable.
        By o-minimality, after simplifying according to whether $\lambda_{i, 0}$ or $\lambda_{i, 1}$ is zero, we can find an interval $I' = (0, b) \subseteq I$, $\lambda \in K$, and $u \in \VV$ such that
        $$
        f(v) = \lambda \cdot v + u
        $$
        holds for all $v \in I'$.
        (Note that the case $\lambda_{i, 1} = 0$ cannot hold on an interval.)
        Since $f$ is a partial endomorphism, this implies $u = 0$.
        Thus $f = \lambda$ on $I'$, and this already implies $f = \lambda$ on all of $I$, contradicting our assumption on $f$.
    \end{innerproof}
    \end{subclaim}
    \noindent Now, since $\theta^0$ and $\theta^1$ are $K$-linearly independent in $R_C$, and since the formula
    $$
    x_0 \in I \setminus U_0 \wedge x_1 - f(x_0) = a
    $$
    still implies no finite disjunction of non-trivial linear dependencies in $x_0x_1$ for any finite $U_0 \subseteq \VV$, Lemma \ref{lemma_rc_li_plus_li} yields infinitely many $b \in M$ with $a = \theta(b) - f(b)$.
    Thus, we can find an elementary extension $(\mm', \theta') \succ (\mm, \theta)$ and some $b' \in M' \setminus M$ with $a = \theta(b') - f(b')$.
    This implies $a \in \cl_\theta(b') \setminus \cl_\theta(\varnothing)$.
    Since $a \in M$ and $b' \not\in M = \cl_\theta(M)$, we have $b' \not\in \cl_\theta(a)$.
    This contradicts our assumption that $\cl_\theta = \acl_{L_\theta}$ has the exchange property in $T\theta^C$, so we must have $\dim_K(\mathfrak{G}) = 1$.

    Finally, suppose that $\dim_K(\mathfrak{G}) = 1$ and that $R_C$ is a field.
    We show that in this case $\cl_\theta$ does have the exchange property.
    Note that $\dim_K(\mathfrak{G}) = 1$ implies $\dim(\VV) = 1$, by Observation \ref{obser_dim_small}.
    Recall that, by Definition \ref{def_cl_theta}, $\cl_\theta(A)$ is the smallest set that is both closed under $\acl_L$ and multiplication with elements of $R_C$.
    By Lemma \ref{lemma_span_is_acl_linear}, we have $\acl_L = \spanA{\;}{\hat{L}}$, where
    $$
    \hat{L} := \set{0, +, <, (\hat{g} : g \text{ is a } \varnothing\text{-definable partial endomorphism}), (a : a \in \dcl_L(\varnothing))}
    $$
    is the language introduced in Fact \ref{fact_lin_qe}.
    Since every partial endomorphism is the multiplication by some $\lambda \in K \subseteq R_C$ on its domain, by the assumption $\dim_K(\mathfrak{G}) = 1$, one easily verifies that $\cl_\theta(A)$ is the $R_C$-linear span of $A\dcl_L(\varnothing)$ for every $A \subseteq M$.
    Since $R_C$ is a field by assumption, it follows that $\acl_{L_\theta}$ has the exchange property.
\end{proof}
\end{theorem}

\begin{remark}
    Set $T := \Th(\RR, 0, 1, +, (q\cdot)_{q \in \QQ}, <)$, with $\VV$ being the $\QQ$-vector space given by addition.
    \begin{enumerate}[(i)]
        \item The theory $T\theta^{C_0}$ is precisely the theory $T_t$ from Definition 3.5 in \cite{BGCH21}, if one works in the language $\set{0, 1, +, (r \cdot)_{r \in R_{C_0}}, <}$.
        Recall that $R_{C_0} = \QQ(X)$.
        To see that $T\theta^{C_0}$ implies $T_t$, apply Lemma \ref{lemma_rc_li_plus_li}.
        The other implication follows because, in $T$, the sets that imply no finite disjunction of non-trivial linear dependencies over $\VV$ are precisely the sets that contain an open cell.
        We recover all the results from \cite{BGCH21} for this theory $T_t$: it is not o-minimal, and it has quantifier elimination, an o-minimal open core (Theorem 5.3 in \cite{Chi25b}), definable completeness (a consequence of having an o-minimal open core), \NIP{}, infinite $\operatorname{dp}$-rank, and the exchange property.
        The theory $T_t$ was constructed in \cite{BGCH21} to show that there are expansions of an ordered group that are definably complete, have the exchange property, and are \NIP{}, but are not o-minimal.
        \item Let $f \in \QQp{}$ have degree $d \geq 2$, and let $C$ be algebraic with $\mipo(C) = f$.
        We similarly see that $T\theta^C$ is not o-minimal, and that it has quantifier elimination, an o-minimal open core, definable completeness, \NIP{}, and the exchange property.
        However, $\dprk{}(T\theta^C) = d$, by Theorem \ref{corollary_dp_rank_btter_fml} and Remark \ref{rem_dim_r_C}.
        Hence, for any $d \geq 2$, we have shown that there are also expansions of an ordered group that are definably complete, have the exchange property, and have $\operatorname{dp}$-rank $d$, but are not o-minimal.
    \end{enumerate}
\end{remark}

\noindent In combination with Theorem \ref{theorem_no_exchange}, we see that in the non-trivial o-minimal case $\acl_{L_\theta} = \cl_\theta$ has the exchange property if and only if $T$ is the theory of ordered $K$-vector spaces, potentially expanded by some parameters, $\VV$ is precisely that vector space, and $R_C$ is a field.

\bibliographystyle{alphaurl}
\bibliography{sample}

\Addresses

\end{document}